%% file: Multicritical_Schur_measures.tex
\documentclass[11pt]{scrartcl}
\usepackage{graphicx}
\usepackage[colorlinks=true,citecolor=blue,linkcolor=black,urlcolor=blue]{hyperref}
\usepackage{amsmath,amssymb,amsthm}
\usepackage{comment}
\usepackage{esint}
\usepackage{fancyhdr}
\usepackage{fullpage}
\usepackage{enumerate}
\usepackage[font={small}]{caption}
\def\DMO{\DeclareMathOperator}
\def\fg{F_{\scalebox{0.5}{\text{GUE}}}} 
\def\lel{\ell(\lambda)}
\DMO{\Ai}{Ai}
\DMO{\Be}{Be}
\DMO{\Prob}{Prob}
\DMO{\Var}{Var}
\renewcommand{\P}{\mathbb{P}}
\newcommand{\Z}{\mathbb{Z}}

\newcommand{\R}{\mathbb{R}}

\def\E{\mathbb E}

\def\aa{\mathrm{a}}
\def\ss{\mathrm{s}}

\renewcommand{\binom}[2]{\left(\begin{smallmatrix} #1 \\ #2 \end{smallmatrix}\right)}
\newcommand{\tr}{\mathrm{tr}}
\newcommand{\normord}[1]{:\mathrel{\mkern2mu #1 \mkern2mu}:}
\newcommand{\ket}[1]{|\mathbin{#1} \rangle}
\newcommand{\bra}[1]{\langle\mathbin{#1}|}
\newcommand{\braket}[2]{\langle\mathbin{#1}|\mathbin{#2} \rangle}
\newcommand{\lh}{{\lambda}}
\newcommand{\eps}{\varepsilon}
\def\ind{\mathbf{1}}
\def\xedge{x_{\mathrm{edge}}}
\def\pedge{p_{\mathrm{edge}}}

\renewcommand{\Im}{\mathrm{Im}}
\renewcommand{\Re}{\mathrm{Re}}

\newtheorem{thm}{Theorem}
\newtheorem{lem}[thm]{Lemma}
\newtheorem{prop}[thm]{Proposition}

\newtheorem{corr}[thm]{Corollary}

\newenvironment{miniproof}[1][Proof]{\begin{proof}[#1]}{\begingroup\end{proof}\endgroup}
\theoremstyle{definition}
\newtheorem{defn}[thm]{Definition}
\numberwithin{equation}{section}

\title{Multicritical Schur measures and higher-order analogues of the Tracy--Widom distribution}
\date{\vspace{-5ex}}
\author{Dan Betea\thanks{\, \href{mailto:dan.betea@gmail.com}{dan.betea@gmail.com} } 
\and
  Jérémie Bouttier\thanks{\, Université Paris-Saclay, CNRS, CEA, Institut de physique théorique,
91191 Gif-sur-Yvette, France \newline \href{mailto:jeremie.bouttier@ipht.fr}{jeremie.bouttier@ipht.fr} } 
\thanks{\, Université Lyon, École Normale Supérieure de Lyon, Université Claude Bernard, CNRS, Laboratoire de Physique, 69342 Lyon, France}  \and
  Harriet Walsh\thanks{\, Université d’Angers, CNRS, LAREMA, SFR MATHSTIC, 49045 Angers, France \newline \href{mailto:harriet.walsh@univ-angers.fr}{harriet.walsh@univ-angers.fr} (corresponding author)
  \newline
 This project has received funding from the European Research Council (ERC) under the European Union’s Horizon 2020 research and innovation programme under Grant Agreements No. ERC-2016-STG 716083, ``CombiTop'' and No. ERC-2017-STG 759702, ``COMBINEPIC'', from the FWO Flanders project EOS 30889451, and from the Agence Nationale de la Recherche via the grants ANR-18-CE40-0033 ``Dimers'' and ANR-19-CE48-0011 ``Combiné'' and the Centre Henri Lebesgue  ANR-11-LABX-0020-01.} \footnotemark[3] }

\begin{document}

\maketitle

\begin{center}
\today{}
\end{center}

\section*{Abstract}

{ We introduce multicritical Schur measures, which are
  probability laws on integer partitions which give rise to
  non-generic fluctuations at their edge. They are in the same
  universality classes as one-dimensional momentum-space models of
  free fermions in flat confining potentials, studied by Le Doussal,
  Majumdar and Schehr. These universality classes involve critical
  exponents of the form $1/(2m+1)$, with $m$ a positive integer, and
  asymptotic distributions given by Fredholm determinants constructed
  from higher order Airy kernels, extending the generic Tracy--Widom
  GUE distribution recovered for $m=1$.  We also compute limit shapes
  for the multicritical Schur measures, discuss the finite temperature
  setting, and exhibit an exact mapping to the multicritical unitary
  matrix models previously encountered by Periwal and Shevitz.}
\newpage 

\tableofcontents

\section{Introduction}

\subsection{Context, motivations and outline}
An important class of models in statistical physics exhibit universal fluctuations governed by the Tracy--Widom distribution~\cite{Tracy_Widom_1993}, first encountered for the largest eigenvalue of a random matrix in the Gaussian unitary ensemble. In suitable scaling limits, this distribution  generally describes fluctuations in random interfaces marking a transition between a strongly and weakly coupled
regimes~\cite{Majumdar_Schehr_2014}, notably appearing in random growth
models~\cite{Praehofer_Spohn_2002} (where it is has been observed
experimentally~\cite{Takeuchi_2011}), directed polymers, tilings of the plane by dominoes and lozenges,
asymmetric exclusion processes, and monotone subsequences of random
permutations to name a few, see e.g.~\cite{Borodin_Gorin_2016} and
references therein. Common to several of the models with this universal interface behaviour, however, is the existence of
maps to non-interacting fermions in one dimension or, in mathematical
terms, to determinantal point processes, where all correlation
functions can be computed exactly as determinants.

This paper is concerned with alternative statistics arising at interfaces of models with the same determinantal structure, and the universality classes associated with them.
Le Doussal, Majumdar and Schehr~\cite{LDMS_2018}
recently observed new edge statistics in models of non-interacting
fermions on the real line in flatter-than-harmonic trap
potentials. 
The limiting distributions they found for the momentum of the most
energetic fermion under suitable rescaling can be seen as higher order
analogues of the TW distribution. These distributions were shown in~\cite[Appendix~G of the arXiv version]{LDMS_2018} and in~\cite{Cafasso_Claeys_Girotti_2019} to be related to solutions of the
Painlevé II hierarchy, which were
previously encountered by Periwal and Shevitz in the double scaling
limit of multicritical unitary random matrix
models~\cite{Periwal_Shevitz_1990,Periwal_Shevitz_1990_2}. 
One of the aims of this paper is to explain this connection by introducing models
of non-interacting fermions on a discrete one-dimensional lattice
which, on the one hand, are in the \emph{same} universality classes as
the models considered by Le Doussal \emph{et al.} and, on the other
hand, exhibit an \emph{exact} mapping to the multicritical unitary
matrix models of Periwal and Shevitz.

Our models belong to the class of \emph{Schur measures}, introduced by
Okounkov~\cite{Okounkov_2001}. These are probability measures on
integer partitions, that generalise the so-called Plancherel measure
which is related with the analysis of the longest increasing
subsequence problem.
Schur measures depend on infinitely many parameters, and one is
usually interested in asymptotics where some or all parameters tend to
infinity in a certain way, such the expectation of the size of a random partition also tends to infinity. 
Generically, the corresponding Young diagram converges after a
suitable rescaling to a deterministic, non universal, limit
shape. Universal phenomena occur when considering local properties,
see for instance the discussion in~\cite{Okounkov_2003}. In this
paper, we are mostly interested in the \emph{edge behaviour}, namely
the fluctuations of the first parts of the partition (or of its
conjugate). The generic situation is the following: if $L$ denotes the
typical length scale of the Young diagram, then these fluctuations are
of order $L^{1/3}$. Their rescaled distribution converges to the
so-called Airy ensemble and, in particular, the rescaled marginal law
of the first part 
converges to the Tracy--Widom
distribution. The peculiarity of our \emph{multicritical Schur
  measures} is that they display a \emph{different} edge behaviour. By
letting the Schur measure parameters tend to infinity in a specific
fine-tuned way, we observe fluctuations of order $L^{1/(2m+1)}$ with
$m$ an arbitrary positive integer. The Airy ensemble and the
Tracy--Widom distribution are then replaced by their higher order
analogues. For $m=1$ we recover the generic situation.

\begin{figure}
{\def\svgwidth{0.35\textwidth}\scriptsize 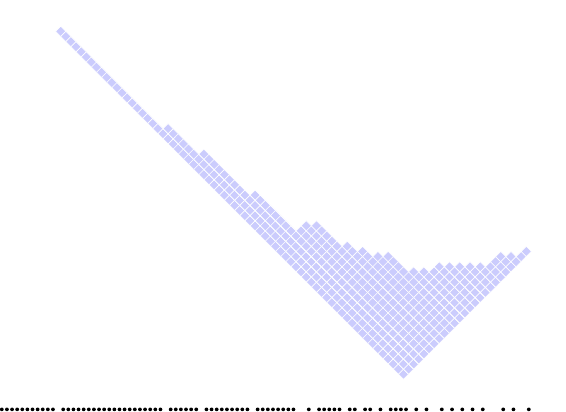}\hfill{\def\svgwidth{0.325\textwidth} \scriptsize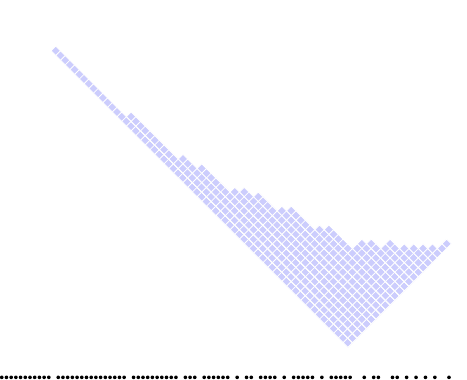}\hfill{\def\svgwidth{0.274\textwidth} \scriptsize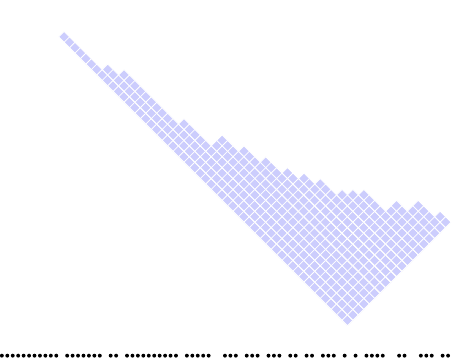}
\caption{Partitions sampled according to certain multicritical measures, namely the minimal measures $\P^{\aa,m}_\theta$ introduced in Section~\ref{sec:explicit} for $m=2,3,4$ and $\theta=20$ (sampling is done via a Metropolis--Hastings algorithm). Their Young diagrams are drawn in the Russian convention. The profile of each partition is the piecewise linear function shown in black; for comparison we display in blue the (appropriately rescaled) limit shape obtained for $\theta \to \infty$.  The fermionic state indexed by each partition, as detailed in Section~\ref{sec:lattice}, is illustrated by the dots below the diagram: these indicate filled sites, corresponding to points where the profile has slope $-1$.} \label{fig:intro_young_diagrams}
\end{figure}

These results were first announced in~\cite{FPSAC_preprint_BBW}. Independent work by Kimura and Zahabi~\cite{Kimura_Zahabi_2021} arguing that the same multicritical edge phenomena may be found for Schur measures appeared on the arXiv shortly after our extended abstract~\cite{FPSAC_preprint_BBW} did. The authors considered the semi-classical analysis of multivariate Bessel functions related to Schur measures, and presented results that are consistent with ours. Our approach is somewhat more direct, as we use asymptotic analysis of gap probabilities to prove rigorously the appearance of multicritical edge behaviour, and we provide explicit instances of multicritical measures. Let us mention the more recent related works~\cite{Kimura_Zahabi_2-2021b,Kimura2022} which further discuss the connection with gauge theory, and introduce higher-order analogues of the ``cusp'' Pearcey kernel.

\paragraph*{Outline} The remainder of this section presents our main results. In Section~\ref{sec:multicriticalmeasures}, we give a general definition of multicritical Schur measures, and we state our main two theorems, the first describing the universal asymptotic edge fluctuations in the first part, and the second giving a formula for its non-universal limit shape. In Section~\ref{sec:explicit}, we give two explicit examples of these measures for each order of multicriticality $m$, along with their explicit limit shapes. In Section~\ref{sec:intro_unitary}, we use an expression for the distribution of the first part of a Schur random partition as an integral over unitary matrices, to write a definition of multicritical unitary matrix models which recovers precisely the models found by Periwal and Shevitz. In Section~\ref{sec:latticegen}, we construct quantum mechanical lattice fermion models corresponding to the multicritical Schur measures, and discuss their asymptotic behaviour heuristically. In Section~\ref{sec:mc_proofs}, we prove the main theorems, using determinantal point processes. In Section~\ref{sec:matrixintegrals}, we discuss multicritical unitary matrix models: we review their connection with multicritical Schur measures and present some heuristic asymptotic analysis, then finally discuss possible further random matrix connections. We relegate some background information and extensions to the appendix: Appendix~\ref{sec:rem} revisits Okounkov's determinantal point process formulation of Schur measures and Appendices~\ref{sec:mc_cylindric} and~\ref{sec:ccgairy} present generalisations of the multicritical Schur measures with asymptotic edge fluctuations governed by distributions arising respectively from positive temperature fermions and from general solutions of Painlevé II equations. 

\subsection{Multicritical Schur measures: definition and main results} \label{sec:multicriticalmeasures}
 An integer partition, or \emph{partition} for short, is a
finite non-increasing sequence $\lambda$ of positive integers
$\lambda = (\lambda_1 \geq \lambda_2 \geq \ldots \geq \lambda_{\ell(\lambda)})$,
where the $\lambda_i$ are called \emph{parts} and $\ell(\lambda)$ is called the \emph{length} of the partition. For $i>\ell(\lambda)$ we set $\lambda_i=0$ for convenience. The \emph{size} of $\lambda$ is the sum of its parts $|\lambda| = \lambda_1+\ldots+\lambda_{\ell(\lambda)}$, and its  \emph{conjugate partition} $\lambda'$ is the partition with parts $\lambda_j' = \#\{i: \lambda_i \geq j\}$. Note that $\lambda_1' = \ell(\lambda)$ and $(\lambda')' = \lambda$. The \emph{Young diagram} associated with a partition $\lambda$ is an arrangement of $|\lambda|$ square boxes in left-aligned rows, with $\lambda_i$ boxes in the $i$-th row. It is drawn in the ``French convention'' by ordering the rows from the bottom up (i.e., the longest rows are at the bottom), and in the ``Russian convention'' by rotating the diagram in the French convention counter-clockwise by 45$^\circ$ (see Figure~\ref{fig:intro_young_diagrams} or Figure~\ref{fig:ydtomaya} for examples in this convention). 

Let $\lambda $ be a partition and let $t:=(t_1,t_2,\ldots)$ be a sequence of parameters. The \emph{Schur function} $s_\lambda[t]$ evaluated at the \emph{Miwa times} $t$ is defined as
\begin{equation} \label{eq:jt1} s_\lambda[t] = \det_{1\leq i,j \leq
    \ell(\lambda)} h_{\lambda_i-i+j}[t]
 \end{equation} 
 where $h_k[t]$ is given by the generating function
 \begin{equation} \label{eq:hgenfn}
 \sum_{k} h_k[t] z^k = \exp \left[ \sum_{r \geq 1} t_r z^r\right]
\end{equation}
(we have $h_k[t] = 0$ for $k<0$ and $h_0[t]=1$). When we take
$t_r=\frac1r \sum_i x_i^r$ for some set of formal variables $x_i$,
$h_k(x_1,x_2,\ldots)$ is nothing but the complete homogeneous
symmetric function of degree $k$, and the definition~\eqref{eq:jt1} we
use for the Schur function is the Jacobi--Trudi identity; see
e.g.~\cite[Sections I.2 and I.3]{Macdonald_1998}. The \emph{Schur
  measure}, defined by Okounkov~\cite{Okounkov_2001}, assigns to a
partition $\lambda$ a weight of the form
$e^{-\sum_{r\geq 1} rt_r t'_r} s_\lambda[t] s_\lambda[t']$, where $t$
and $t'$ are two sequences of Miwa times. Upon imposing appropriate
conditions on $t,t'$, it is a probability distribution over the set of
all partitions, which is normalized by the Cauchy identity, see
e.g.~\cite[Section I.4]{Macdonald_1998}.  For instance, we may take
$t_r$ and $t_r'$ to be complex conjugate to one another for all $r$,
with $\sum_{r\geq 1} r|t_r|^2 < \infty$: in this case we say that the
Schur measure is \emph{Hermitian}.

\begin{defn}[Multicritical measures] \label{def:msm}
For each positive integer $m$,  an order $m$ \emph{multicritical Schur measure} is a Hermitian Schur measure $\P_\theta^m (\lambda):= e^{-\theta^2\sum_r r {\gamma_r}^2} s_\lambda[\theta \gamma]^2$, for real Miwa times
$\theta \gamma :=(\theta \gamma_1, \theta \gamma_2, \ldots)$ where $\theta$ is a positive parameter and $\gamma = (\gamma_1,\gamma_2, \ldots)$ is a sequence of real numbers with finite support satisfying the vanishing conditions
\begin{align} 
  \sum_{r\geq 1} r^{2p+1} \gamma_r =0 \quad \text{ for } p=1, 2, \ldots, m-1  \label{eq:multicriticalitycondition}
\end{align}
(there are no vanishing conditions for $m=1$), and 
\begin{align} 
 \sum_{r\geq 1} r^{2m+1} \gamma_r  \neq 0, \qquad \sum_{r\geq 1} r^2 \gamma_r \sin r \phi \geq 0 \quad \text{ for all }\phi\in[0,\pi]. \label{eq:singlefermiseacondition}
\end{align}
The positive constants 
\begin{equation} \label{eq:coefficientsmcdef}
b :=2 \sum_{r\geq 1} r \gamma_r, \quad \tilde{b} := 2\sum_{r\geq 1}(-1)^{r+1}r\gamma_r,\quad d := \frac{2(-1)^{m+1}}{(2m)!}\sum_{r\geq 1} r^{2m+1} \gamma_r
\end{equation} are respectively called the right edge, left edge and (right) fluctuation coefficients associated with the measure. 
\end{defn}

See Figure~\ref{fig:intro_young_diagrams} for examples of partitions sampled under this measure. The expectation of the size of $\lambda$ under $\P_\theta^m$ can be
computed using the Cauchy identity
$\sum_{\lambda} s_\lambda[\theta \gamma]^2 = e^{\theta^2 \sum_{r} r
  {\gamma_r}^2}$ and the homogeneity property of Schur functions, to yield
\begin{equation}
\E_\theta^m(|\lambda|) = e^{-\theta^2\sum_r r {\gamma_r}^2} \sum_{\lambda}\frac{1}{2} \frac{d}{d q }s_\lambda[q \theta \gamma_1, q^2\theta \gamma_2, q^3\theta \gamma_3, \ldots]^2 \bigg\vert_{q=1}=  \theta^2\sum_r r^2 {\gamma_r}^2.
\end{equation}
As $|\lambda|$ is the area of the Young diagram of $\lambda$, the
parameter $\theta$ defines a typical length scale for the parts
$\lambda_i$, $\lambda_i'$.

The $m=1$ measures are the most generic, and they notably include the \emph{Poissonised Plancherel measure} $\P_\theta$ which arises naturally from the study of increasing subsequences in uniform random permutations (see e.g.~\cite{Romik_2015} for a thorough overview), and which is obtained by putting $\gamma_1 = \theta$ and all other $\gamma_r$ equal to zero. The order $m+1$ measures must satisfy one more linear constraint than the order $m$ ones. Of the two conditions~\eqref{eq:singlefermiseacondition}, the first sets the order of multicriticality along with~\eqref{eq:multicriticalitycondition}, while the second establishes a sign convention, giving $b >0$ and $d>0$, along with a more technical requirement used in our asymptotic analysis which corresponds to the quantum mechanical notion of having a ``single Fermi sea''. The single Fermi sea condition is discussed in Section~\ref{sec:fermions-asymptotics} below, and the case in which this condition is partially lifted was studied in the recent preprint~\cite{HW_splitfermi} by the third author. 

Let us remark that, aside from the special case of the Poissonised
Plancherel measure, Hermitian Schur measures differ from the more
typical combinatorial setting, in which the Schur measure is defined
from sets of parameters $t,t'$ which are chosen such that the Schur
functions $s_\lambda[t],s_\lambda[t']$ are both non-negative for each
$\lambda$ (i.e., $t,t'$ both correspond to Schur-non-negative
specialisations). While this latter setting often allows the Schur
measure to be extended to a probabilistic time-dependent
process~\cite{Okounkov_Reshetikhin_2003}, this does not seem possible
for our multicritical Schur measures for $m>1$.  We note however that
the Hermitian Schur measures arise naturally through a correspondence
with lattice fermions, presented in Section~\ref{sec:latticegen}. A
system of lattice fermions evolving in imaginary time, which
corresponds to an extension of an $m=2$ multicritical measure, was
considered by Bocini and Stéphan~\cite{Bocini_Stephan_2020}, but that model was described as ``non-probabilistic'' by the authors, as
it gave rise to negative Boltzmann weights at certain times.

In order to state our first main theorem, we 
introduce,  following~\cite{LDMS_2018}, the Fredholm determinant
\begin{equation} \label{eq:FdetFred}
  \begin{split}
    F_{2m+1}(s) &:= {\det} (1 - \mathcal{A}_{2m+1})_{L^2([s,\infty))} \\
    &\phantom{:}= \sum_{n=0}^\infty \frac{(-1)^n}{n!} \int_s^\infty \cdots \int_s^\infty \det_{1\leq i,j \leq n} \mathcal{A}_{2m+1}(x_i,x_j) dx_1 \cdots dx_n
\end{split}
\end{equation}
with $\mathcal{A}_{2m+1}$ the higher-order Airy kernel defined by
\begin{equation}
  \begin{split}
\mathcal{A}_{2m+1}(x, y) &= \int_0^\infty \Ai_{2m+1}(x+v) \Ai_{2m+1}(y+v)dv \\
&= \sum_{i=0}^{2m-1}\frac{(-1)^{m+i+1}\Ai_{2m+1}^{(i)}(x)  \Ai_{2m+1}^{(2m-1-i)}(y)}{x-y}
\end{split}\label{eq:genairykernel}
\end{equation}
and $\Ai_{2m+1}$ the higher-order Airy function\footnote{Our integration convention differs from~\cite[Equation 5]{LDMS_2018} which defines the same function. In their expression the integration is taken over a line to the left of the origin for even $m$, and is recovered from ours by the change of integration variable $\zeta \to -\zeta$. It also differs from~\cite[Equation~1.1]{Cafasso_Claeys_Girotti_2019}, where integration contours at an angle of $\frac{m\pi}{2m+1}$ are taken instead for faster convergence, but again both integrals define the same function.  }
\begin{equation}
\Ai_{2m+1}(x) := \frac{1}{2\pi i}\int_{1+i\R} \exp\bigg[ (-1)^{m-1}\frac{\zeta^{2m+1}}{2m+1}-x\zeta \bigg] d\zeta.   \label{eq:genairydef}
\end{equation}
In the second expression of~\eqref{eq:genairykernel}, we use the notation $f^{(n)}(x):=d^n f / dx^n$, and the $x=y$ case is recovered by L'Hôpital's rule i.e. evaluating the derivative of the numerator at $x=y$. Note that the higher-order Airy functions $\Ai_{2m+1}$ decay to zero at  positive infinity, and that $\mathcal{A}_{2m+1}$ has finite trace on any $L^2([t,\infty))$ where $t$ is finite. In the $m=1$ case, we have $\Ai_3 = \Ai$ and $F_3 $ is the Tracy--Widom distribution for the Gaussian unitary ensemble, $\fg$.

Definition~\ref{def:msm} amounts to tuning Hermitian Schur measures to
have the following edge behaviour:

\begin{thm}[Asymptotic edge fluctuations of multicritical
  measures] \label{thm:multicritical} Let $\lh$ be a random partition
  distributed by an order $m$ multicritical measure
  $\P^m_{\theta}(\lambda)$ with right edge position and fluctuation
  coefficients $b,d$. Then, we have
\begin{equation}  \label{eq:edgefluctuationsf}
  \lim_{\theta \to \infty}\P_\theta^m \bigg[\frac{\lh_1 - b \theta }{(d\theta)^{\frac{1}{2m+1}} } \leq s \bigg] = F_{2m+1}(s).
\end{equation}
\end{thm}

The proof of this theorem is given in
Section~\ref{sec:mcedge_fluctuations}. For the Poissonised Plancherel
measure $\P_\theta$, Theorem~\ref{thm:multicritical} reduces to the
(Poissonised version of the) celebrated theorem of Baik, Deift and
Johansson~\cite{Baik_Deift_Johansson_1999}, with $b=2$ and $d=1$. For
$m>1$, the theorem defines higher-order analogues of the TW-GUE
distribution associated with critical fluctuation exponents $1/(2m+1)$.
Note that the constants $b,d$ are not universal for all order $m$
multicritical measures, but that these exponents are.

The distributions $F_{2m+1}$ were related to higher order integrable equations of the Painlevé II hierarchy first in~\cite[Appendix G]{LDMS_2018} using an approach similar to Tracy and Widom's~\cite{Tracy_Widom_1993}, and then by rigorous Riemann--Hilbert analysis in~\cite{Cafasso_Claeys_Girotti_2019}. These authors showed that the order $2m$ equation of the Painlevé II hierarchy has a solution $q_m$ such that
\begin{equation} \label{eq:fgenpainleve}
    F_{2m+1}(s) = \exp \bigg[ - \int_s^\infty (x-s) q_m^2 ((-1)^{m+1}x) \, dx\bigg]
\end{equation}
where $F_{2m+1}(s)$ denotes the Fredholm determinant defined at~\eqref{eq:FdetFred}, and this solution satisfies
\begin{equation} \label{eq:piisolatinfty}
  q_m((-1)^{m+1} s )=
  \begin{cases}
    \left(\frac{|s| m!^2}{(2m)!}\right)^{1/2m} (1 + o(1)) & \text{as $s\to -\infty$,} \\
    O\left(e^{-C s^{\frac{2m+1}{2m}}}\right) & \text{as $s\to +\infty$,}\\
  \end{cases}
\end{equation}  
for some constant $C>0$. The $m=1$ case is just the Painlevé transcendent expression for $\fg$ given by Tracy and Widom~\cite{Tracy_Widom_1993}. In~\cite{Cafasso_Claeys_Girotti_2019}, the authors also found solutions to the generalised Painlevé II hierarchy in terms of the Fredholm determinants of a more general class of higher-order Airy kernels; we discuss generalised multicritical Schur measures with corresponding edge behaviours in Appendix~\ref{sec:ccgairy}. The distributions $F_{2m+1}$ have positive temperature extensions analogous to Johansson's extension of $\fg$~\cite{Johansson_2007}, as shown in~\cite[Appendix E]{LDMS_2018}, and in Appendix~\ref{sec:mc_cylindric} we define laws on partitions with the same kind of asymptotic  fluctuations by way of the periodic Schur process.

We now consider the ``macroscopic'' shape of multicritical random partitions. The shape of a partition $\lambda$ may be described by its   \emph{rescaled profile} $\psi_{\lambda,\theta}$, which is defined from 
 the implicit relations
\begin{equation} \label{eq:profilephis}
v = v(u) := \tfrac{1}{\theta} \lambda_{\lfloor  \theta u \rfloor+1}, \; u \in (0, \infty) \quad\text{and}\quad u = u(v) := \tfrac{1}{\theta} \lambda'_{\lfloor  \theta v \rfloor+1}, \;  v \in (0, \infty)
\end{equation}
(with $\lfloor \cdot \rfloor$ denoting the floor function) through a change of coordinates 
\begin{equation} \label{eq:changetohookcoords}
\psi_{\lambda,\theta}(x) = u+v, \qquad x = v-u. 
\end{equation}
This is the piecewise linear curve tracing the upper edge of the Young diagram of $\lambda$ drawn in the Russian convention, and we notably have $\psi_{\lambda}(x) = |x|$ for $x> \lambda_1/\theta$ and $x<-\ell(\lambda)/\theta$, and $\psi_{\lambda,\theta}(x) > |x|$ for all intermediate values of $x$; see Figure~\ref{fig:intro_young_diagrams} or Figure~\ref{fig:ydtomaya}. At a scale of $1/\theta$ as set by~\eqref{eq:profilephis}, we have the following limit shape phenomenon:
\begin{thm}[Limit shapes of multicritical measures] \label{thm:mclimitshape}
The rescaled profile $\psi_{\lambda,\theta}$ of a random partition $\lh$ under an order $m$ multicritical measure $\P_\theta^m$ has a deterministic limit curve: as $\theta \to \infty$, we have the convergence in probability
\begin{equation}
\sup_{x}| \psi_{\lambda,\theta}(x)- \Omega(x)| \xrightarrow{p} 0
\end{equation}
where $\Omega$ is the function depending on the sequence $\gamma$ given by
\begin{equation} \label{eq:multicritlimitprofile}
\Omega(x) = \begin{cases}x +2\tilde{b} -  \frac{2}{\pi} \int_{-\tilde{b}}^x \chi(v)dv \,, &\quad x\in [-\tilde{b},b] \\ 
|x|\, , &\quad x>b \text{ or } x<-\tilde{b}
\end{cases}
\end{equation}
with $\chi(x) \in [0,\pi ]$ determined implicitly by
\begin{equation}
  \label{eq:chi}
2\sum_r r \gamma_r \cos r \chi(x) = x , \quad x\in [-\tilde{b},b].
\end{equation}
At the right edge of the Young diagram,  the derivative of the limiting profile displays a universal critical exponent $1/2m$, with 
\begin{equation} \label{eq:limitshapevanishing}
\Omega'(x) \sim 1-\frac{2}{\pi} \bigg(\frac{b-x}{d}\bigg)^{\frac{1}{2m}} \qquad\text{as }x\to b^-.
\end{equation}
\end{thm}

The proof of this theorem is given in
Section~\ref{sec:mclimitshape}. Note that the
condition~\eqref{eq:singlefermiseacondition} ensures
that~\eqref{eq:chi} indeed admits a unique solution. In the special
case of the Poissonised Plancherel measure $\P_\theta$, this is a
Poissonised version of the limit shape theorem proven by Vershik and
Kerov~\cite{Vershik_Kerov_1977} and, independently, Logan and
Shepp~\cite{Logan_Shepp_1977}. Note
that~\eqref{eq:multicritlimitprofile} provides a general formula for
the limit shape $\Omega$, depending on the precise constants
specifying the measure, and only the behaviour at the edge of the
support of $\Omega(x)-|x|$ is universal.

The exponent $1/(2m)$ appearing in~\eqref{eq:limitshapevanishing} can
be related to the fluctuation exponent $1/(2m+1)$
of~\eqref{eq:edgefluctuationsf} by the following heuristic scaling
argument. From the limit shape theorem we expect that, for any $u>0$,
we have $\lambda_{\lfloor \theta u \rfloor} \sim \theta (x+u)$ as
$\theta \to \infty$, where $x$ is given implicitly by
$2u=\Omega(x)-x$. By~\eqref{eq:limitshapevanishing}, we have
$\Omega(x)-x \propto (b-x)^{\frac{2m+1}{2m}}$ for $x \to
b^-$. Inverting, we find $b-x \propto u^{\frac{2m}{2m+1}}$ for
$u \to 0$. Now, assuming that we may take $u=\theta^{-1}$ (the devil
hides there!), we deduce that
$\lambda_1-b\theta \propto \theta^{\frac{1}{2m+1}}$ consistently
with~\eqref{eq:edgefluctuationsf}.

Although we only state our main theorems for the right edge at $\lh_1$, analogous results for the second interface at $\lel$ can be extracted directly, because of the following:
\begin{prop}[Conjugate partition under a Schur measure] \label{prop:mcconjugation}
If $\lh$ is a random partition under a Schur measure $\P(\lambda) = e^{- \sum_r r t_rt_r'} s_\lambda[t]s_\lambda[t']$, then the law of its conjugate $\lh'$ is 
\begin{equation}
\P(\lambda') =  e^{- \sum_r r {t}_rt'_r} s_\lambda[\tilde{t}]s_\lambda[\tilde{t}'], \quad \text{where} \quad \tilde{t}_r = (-1)^{r-1}t_r, \:  \tilde{t}'_r = (-1)^{r-1}t'_r.
\end{equation}
\end{prop}
This follows from properties of the Schur function, which may equivalently be defined in terms of the conjugate partition by
\begin{equation} \label{eq:jt2}
s_\lambda[t] = \det_{1 \leq i,j \leq \ell(\lambda')} e_{\lambda_i' - i + j}[t]
\end{equation}
where $e_k[t]$ is defined by way of the generating function 
\begin{equation}
\sum_{k} e_k[t]z^k = \exp \left[  \sum_{r\geq 1} (-1)^{r+1}t_r z^r \right].
\end{equation}
The $e_k[t]$ are the elementary symmetric functions in terms of the Miwa times $t$, and~\eqref{eq:jt2} is the second Jacobi--Trudi identity (see~\cite[Section I.3]{Macdonald_1998}).

\subsection{Examples: minimal multicritical measures} \label{sec:explicit}
To give concrete examples, let us introduce two canonical ways to construct a multicritical measure for each order $m$. The first construction  consists of fixing the sequence $\gamma$ by allowing only the first $m$ coefficients $\gamma_r$ to be non-zero: as the vanishing conditions~\eqref{eq:multicriticalitycondition} form a linear system of $m-1$ independent equations, they determine $\gamma_1,\ldots,\gamma_m$ uniquely up to normalisation, which we fix by setting $\gamma_1 = 1$ in each case (so that, in particular, we recover the Poissonised Plancherel measure $\P_\theta$ for $m=1$). It turns out that the coefficients chosen this way also satisfy the extra condition~\eqref{eq:singlefermiseacondition} (see Proposition~\ref{prop:minimal}), and we find the following family of measures:

\begin{defn}[Minimal multicritical measures] \label{def:minimalmsm}
The order $m$ \emph{minimal multicritical measure} is $\P_{\theta}^{\aa,m}(\lambda) =  e^{-\theta^2 \sum_r r {\gamma_r}^2} s_\lambda[\theta \gamma]^2$ where
\begin{equation} \label{eq:min_parameters_a}
\gamma_r = \begin{cases} \tfrac{(-1)^{r+1}}{r}\binom{2m}{m+r} / \binom{2m}{m-1}\, , \qquad &r= 1,2,\ldots,m \\
0, \qquad &r > m. \end{cases}
\end{equation}
Its edge and fluctuation coefficients are
\begin{equation}
b = \frac{m+1}{m}, \quad \tilde{b}= 4^m\binom{2m}{m}^{-1} - \frac{m+1}{m},\quad d = \binom{2m}{m-1}.
\end{equation}
\end{defn}

From Theorem~\ref{thm:mclimitshape}, we have explicit limit shapes for these measures:

\refstepcounter{thm}\label{corr:minls}
\edef\tempthmno{\thethm}
\begingroup
\renewcommand\thethm{\tempthmno~of Theorem~\ref{thm:mclimitshape}}
\begin{corr}[Limit shapes of minimal multicritical measures] \label{corr:asymmmin}
The rescaled profile $\psi_{\lh,\theta}$ of a random partition under $\P_\theta^{\aa,m}$ converges in probability to
\begin{equation} \label{eq:alimitprofile}
\Omega^{\aa,m}(x) = \begin{cases}x +2\tilde{b} - \frac{2}{\pi}\int_{-\tilde{b}}^x \arccos \big[ 1 - \tfrac{1}{2}\binom{2m}{m-1}^{\frac{1}{m}}(b-v \big)^{\frac{1}{m}}\big]dv, \; &x\in [-\tilde{b},b] \\ 
|x|, \hfill x>b \text{ and}\hspace*{-0.5em} &x<-\tilde{b}
\end{cases}
\end{equation}
in the supremum norm. 
\end{corr}
\endgroup
\addtocounter{thm}{-1}

\begin{figure} 
\begin{center}
{\def\svgwidth{0.8\textwidth} \small 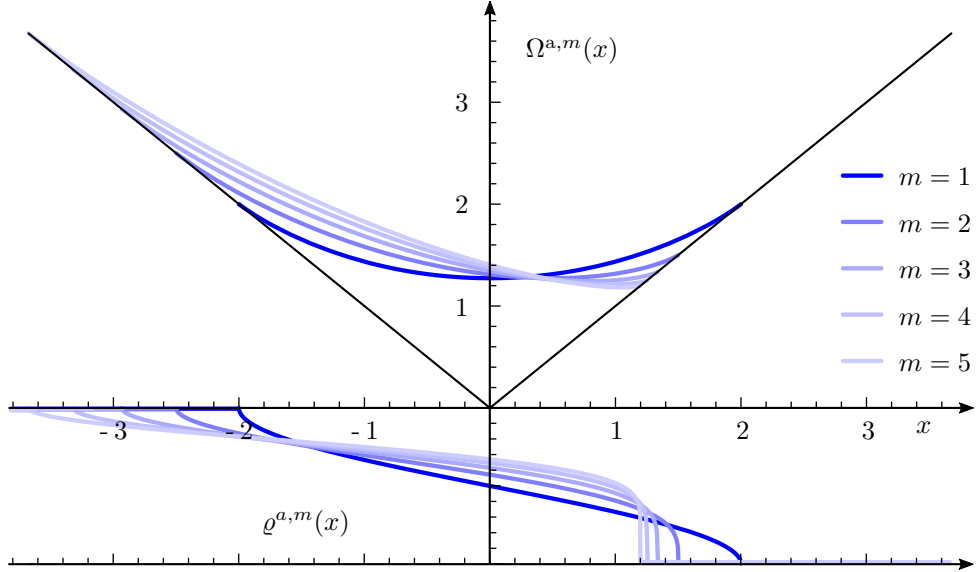}
\end{center}
\caption[Limit curves of the minimal multicritical measures]{Limit curves $\Omega^{\aa,m}$ for partitions under the minimal multicritical measures $\P^{\aa,m}_\theta$ as $\theta \to \infty$ (see Corollary~\ref{corr:asymmmin}). The limiting densities $\varrho^{\aa,m}$ in the corresponding fermion models (discussed in Section~\ref{sec:fermions-asymptotics}) are shown below; they are related to the limit curves by $\Omega'(x) = 1  - 2\varrho(x)$.} \label{fig:asymlimitshape}
\end{figure}

These limit shapes are shown for the first few $m$ in Figure~\ref{fig:asymlimitshape}; for $m=1$, the curve~\eqref{eq:alimitprofile} is precisely the Vershik--Kerov--Logan--Shepp limit shape curve. Notice that for $m>1$ there is no multicriticality at the left edge: by Proposition~\ref{prop:mcconjugation}, the conjugate partition is distributed according to a Hermitian Schur measure with Miwa times $\theta (-1)^{r-1} \gamma_r$, for which multicriticality conditions are no longer satisfied. We see ``generic'' behaviour on the left edge, with the asymptotic fluctuations of $\lel$ governed by the TW-GUE fluctuations.

We can alternatively ensure symmetry under conjugation: thanks to Proposition~\ref{prop:mcconjugation},  each \emph{even-indexed} coefficient should vanish. Our second construction consists of letting only the first $m$ \emph{odd-indexed} coefficients $\gamma_r$ be non-zero, and again fixing $\gamma_1 =1$ (so again we have $\P_\theta$ for $m=1$), to yield the following family of measures:

\begin{defn}[Symmetric minimal multicritical measures] \label{def:symminimalmsm}
The order $m$ \emph{symmetric minimal multicritical measure} is $\P_{\theta}^{\ss,m}(\lambda) = \P_{\theta}^{\ss,m}(\lambda')=  e^{-\theta^2 \sum_r r {\gamma_r}^2} s_\lambda[\theta \gamma]^2$ where
\begin{equation} \label{eq:min_parameters_s}
\gamma_{2r-1} =\begin{cases}\tfrac{(-1)^{r+1}}{(2r-1)^2}\binom{2m-1}{m-r} / \binom{2m-1}{m-1}\, , \qquad &r = 1,2,\ldots,m \,  \\
0 \qquad & r > m\end{cases}
\end{equation}
and $\gamma_{2r} = 0$ for each positive integer $r$. Its edge and fluctuation coefficients are
\begin{equation}
b = \tilde{b} = \frac{2^{4 m-1} (m!)^4}{m ((2 m)!)^2}, \qquad d = \frac{(2m-2)!!}{(2m-1)!!}.
\end{equation}
\end{defn}

The symmetric limit shape for this measure can again be found from Theorem~\ref{thm:mclimitshape}, and are shown for the first few $m$ in Figure~\ref{fig:symlimitshape} (note that here again we have the VKLS curve at $m=1$):

\refstepcounter{thm} \label{corr:symminls}
\edef\tempthmno{\thethm}
\begingroup
\renewcommand\thethm{\tempthmno~of Theorem~\ref{thm:mclimitshape}}
\begin{corr}[Limit shapes of symmetric minimal multicritical measures] 
The rescaled profile $\psi_{\lh,\theta}$ of a random partition under $\P_\theta^{\ss,m}$ converges in probability to
\begin{equation} \label{eq:slimitprofile}
\Omega^{\ss,m}(x) = \begin{cases}x + 2b - \frac{2}{\pi}\int_{-b}^x \chi(v)dv \,, \quad &x\in [-b,b] \\ 
|x|\, , \quad &|x|>b
\end{cases}
\end{equation}
in the supremum norm, where $\chi(x)$ satisfies
\begin{equation}
\int_0^{\chi(x)} \sin^{2m-1} \phi d\phi = \frac{(-1)^{m+1}}{2^{2m-1}} \binom{2m-1}{m} x, \qquad x \in[-b,b]. 
\end{equation}
\end{corr}
\endgroup
\addtocounter{thm}{-1}

\begin{figure}
\begin{center}
{\def\svgwidth{0.8\textwidth} \small 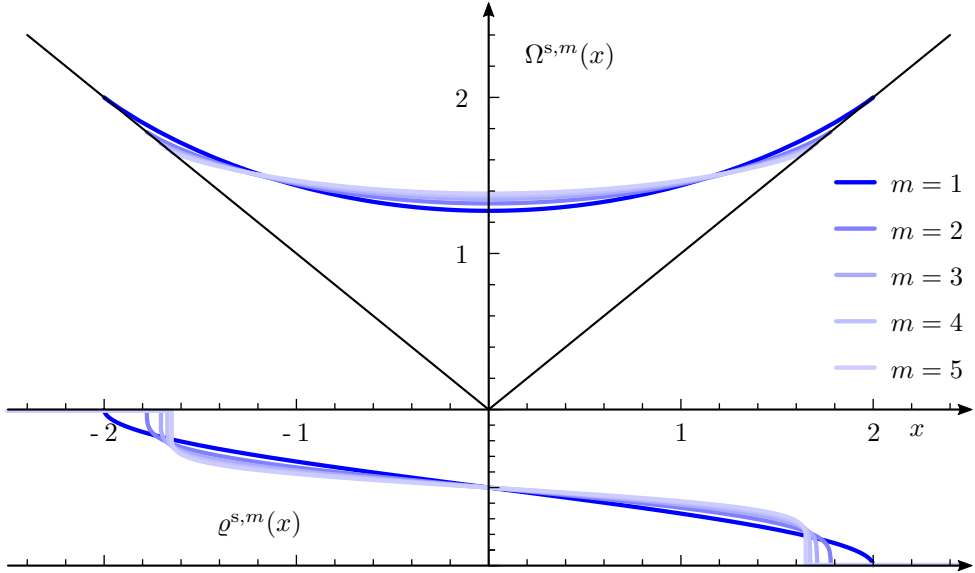}
\end{center}
\caption{Limit curves $\Omega^{\ss,m}$ of partitions under the minimal multicritical measures $\P^{\ss,m}_\theta$ as $\theta \to \infty$ (see Corollary~\ref{corr:symminls}), and corresponding limiting fermion densities $\varrho^{\ss,m}$. Note the symmetry under $x \mapsto -x$.  }  \label{fig:symlimitshape}
\end{figure}

\subsection{Connection with unitary matrix models} \label{sec:intro_unitary}

The multicritical measures on partitions defined above are in direct correspondence with probability densities on unitary matrices, which are multicritical in their own sense. The correspondence comes from an exact expression for the cumulative distribution of the first part of a partition distributed by a Schur measure in terms of an integral over a unitary group, which may be found from identities proven by Baik and Rains~\cite{Baik_Rains_2001} and by Borodin and Okounkov~\cite{Borodin_Okounkov_2000} (we give a self contained proof in Section~\ref{sec:matrixintegrals}). It is as follows:

\begin{thm}[Edge distributions under Schur measures and unitary matrix integrals] \label{thm:genunitary}
Let $\lambda$ be a random partition under a Schur measure $\P(\lambda) = e^{-\sum_r r t_r t'_r} s_\lambda[t] s_\lambda[t']$ for some sequences of Miwa times $t,t'$. Then, for any positive integer $\ell$, we have
\begin{equation} \label{eq:unitintegral}
e^{\sum_r r t_r t'_r}\P(\lambda_1\leq\ell) =  \det_{1 \leq i, j \leq \ell}  f_{j-i} = \int_{\mathcal{U}(\ell)}  \exp \bigg[  \tr \sum_r (-1)^{r-1}(t_r U^r+t'_r U^{-r}) \bigg] \mathcal{D} U
\end{equation}
where the $f_n$ appearing in the determinant are given by
\begin{equation} \label{eq:symbol1}
\sum_{n \in \Z} f_n z^n = \exp \bigg[ \sum_{r\geq 1}(-1)^{r-1}( t_r z^r  + t'_r z^{-r})\bigg]
\end{equation}
and where $\mathcal{D}U$ denotes the Haar measure on the unitary group $\mathcal{U}(\ell)$.
\end{thm}
As $\ell \to \infty$, the first equality  recovers a form of the strong Szegő theorem~\cite{Simon_2005_1}:
  \begin{equation}
    \lim_{\ell \to \infty} \det_{1 \leq i,j \leq \ell} f_{j-i} =  \exp \bigg[ \sum_{r \geq 1} rt_rt'_r\bigg].
  \end{equation}

To make sense of the Haar measure $\mathcal{D}U$ and the integral in~\eqref{eq:unitintegral}, we use the Weyl integration formula (see e.g.~\cite[Chapter~1]{Meckes_2019}): we can perform a change of variables to the  eigenvalues $u_1,\ldots,u_\ell$, each of which lies on the unit circle, and recover (for a given function $f$)
\begin{equation} \label{eq:weyl}
  \int_{\mathcal{U}(\ell)} e^{\tr f(U)} \mathcal{D} U = \frac{1}{(2\pi i)^\ell \ell !} \oint_{c_1} \hspace{-0.3em}\cdots  \oint_{c_1} \prod_{i=1}^\ell e^{f(u_i)}  \prod_{i<j}|u_i - u_j|^2 \frac{du_1}{ u_1} \cdots\frac{du_\ell}{u_\ell}
    \end{equation}  
    where $c_1:|u| = 1$ denotes the unit circle.
    This is equivalently the expectation of $ e^{\tr f(U)}$ in the circular unitary ensemble, and the joint probability density of the eigenvalues can be read from this expression. 
     Note that Proposition~\ref{prop:mcconjugation} once again gives an expression  analogous to~\eqref{eq:unitintegral} for the cumulative distribution of the length $\ell(\lh)$. 

    In the Hermitian case of a Schur measure with $t_r'=t_r^\ast$, the integrand on the right hand side of~\eqref{eq:unitintegral} is non-negative for all $U$, and the distribution $\P(\lambda_1 \leq \ell)$ may be interpreted as the normalisation (or \emph{partition function}) of a random matrix ensemble. Further restricting to the cases of the multicritical Schur measures defined above, Theorem~\ref{thm:genunitary} leads to a definition for a random unitary matrix analogue for a multicritical random partition, as follows:  
\begin{defn}[Multicritical unitary matrix models]
For $\theta>0$ and a sequence $\gamma = (\gamma_1,\gamma_2,\ldots)$ satisfying~\eqref{eq:multicriticalitycondition} and~\eqref{eq:singlefermiseacondition} such that $ e^{-\theta^2\sum_r r\gamma_r^2} s_\lambda[\theta \gamma]^2 =\P_\theta^m(\lambda)$ is an order $m$ multicritical measure, the ensemble of $\ell \times \ell$ random unitary matrices with density
\begin{equation}
  \label{eq:pmUdef}
p^m_{\theta,\ell}(U) := \frac{1}{Z_\ell} e^{\theta \sum_{r \geq 1}(-1)^{r+1} \gamma_r (U^r + U^{- r})},
\end{equation}
with respect to the Haar measure $\mathcal{D}U$ normalised by the partition function
\begin{equation}
Z_\ell := e^{\theta^2 \sum_r r \gamma_r^2}\P_\theta^m(\lambda_1 \leq \ell) = \int_{\mathcal{U}(\ell)} e^{ \theta \sum_{r }(-1)^{r+1} \gamma_r  (U^r + U^{- r})}\mathcal{D}U, 
\end{equation}
is called an \emph{order $m$ multicritical unitary matrix model}.
\end{defn}

It turns out that the linear relations~\eqref{eq:multicriticalitycondition} for the coefficients $\gamma_r$ correspond to natural ``multicriticality'' conditions  for unitary matrix models too, in the sense that they give rise to non-generic vanishing exponents for the limiting eigenvalue density, in analogy with Kazakov's multicritical Hermitian matrix models~\cite{Kazakov_1989}. Letting $\xi = \{\xi_1,\ldots,\xi_\ell\}$ where $-\pi < \xi_j \leq \pi$ denote the arguments  of the eigenvalues $e^{i\xi_j}$ of a random $\ell \times \ell$ unitary matrix, we define the limiting density function $\varrho$ on $(-\pi,\pi]$ such that in the limit $\ell \to \infty$ we have
\begin{equation} \label{eq:evdensitydef}
 \P(\xi \cap [\beta_1,\beta_2]) = \int_{\beta_1}^{\beta_2} \varrho(\alpha) d\alpha, \qquad -\pi<\beta_1 < \beta_2 \leq \pi.
 \end{equation} 
Adapting a computation by Gross and Witten from~\cite{Gross_Witten_1980}, we show that: \textit{ if $U$ is a random $\ell \times \ell$  unitary matrix subject to an order $m$ multicritical probability density $\P(U)\mathcal{D}U=p^m_{\theta,\ell}(U) \mathcal{D}U$, then in a critical regime where $\theta:= \ell/b$, the limiting eigenvalue density $\varrho(\alpha)$  satisfies
\begin{equation} \label{eq:mcunitaryedge}
\varrho(\alpha) \sim \frac{1}{2\pi} \frac{d}{b}(\pi - \alpha)^{2m}, \qquad \alpha \to \pi^-
\end{equation}
where $b,d$ are the constants defined at~\eqref{eq:coefficientsmcdef}.} This behaviour 
is ``inverse'' to the behaviour at the edge of the limit shape under a multicritical measure $\P_\theta^m$ as described in Theorem~\ref{thm:mclimitshape} (we do not, however, prove the result for the limiting eigenvalue density in full rigour). 

The unitary matrix model corresponding to the Poissonised Plancherel measure is well established: the density $
p_{\theta,\ell}(U) := {Z_\ell}^{-1} e^{\theta\tr (U+U^{-1})}$
defines a model of lattice gauge theory shown to exhibit a third order phase transition by Gross, Witten~\cite{Gross_Witten_1980} and, independently, Wadia~\cite{Wadia_1980}, see also~\cite{Johansson_1998}. These authors showed that in a regime where $\theta := \ell/x$, the free energy $\mathcal{F}:= \lim_{\ell \to \infty}\ell^{-2} \log Z_\ell$ has a discontinuity in its third derivative at the critical point $x = 2$; they also show that at this critical point a break appears in the support of the limiting eigenvalue density near $-1$. Analogous ``multicritical''  models were found by Periwal and Shevitz~\cite{Periwal_Shevitz_1990,Periwal_Shevitz_1990_2} using fine-tuned polynomial potentials.
Their approach gives precisely the densities $p^{\mathrm{a},m}_{\theta,\ell} $ corresponding to our minimal multicritical measures $\P_\theta^{\mathrm{a},m}$: indeed, taking the $\gamma_r$ as in~\eqref{eq:min_parameters_a}, the polynomial $\sum_{r\geq 1} (-1)^{r+1} \gamma_r z^r$ involved in the density~\eqref{eq:pmUdef} matches, up to an overall factor of $m/(m+1)$ and up to a constant term, the polynomial $V_k(z)$ found on~\cite[pp.\ 736--737]{Periwal_Shevitz_1990_2} for $k=m$.  In~\cite{Periwal_Shevitz_1990}, the authors also found the critical edge behaviour~\eqref{eq:mcunitaryedge} of the limiting eigenvalue density in their models for $m$ from 1 to 3. 

Combining Theorems~\ref{thm:multicritical} and~\ref{thm:genunitary} with~\cite[Theorem~1.1]{Cafasso_Claeys_Girotti_2019} gives a rigorous proof of  Periwal and Shevitz's main result, which asserts namely that, in a regime where $\ell \sim b \theta  + (d\theta)^\frac{1}{2m+1}s$, the second derivative of $\ell^{-2} \log Z_\ell$ is given by the square of a solution $q_m(s)$ of the $m$th Painlevé II equation as $\theta \to \infty$. Moreover the proof is actually valid for the full class of unitary matrix models defined by the densities $p^{m}_{\theta,\ell}(U)$. From the behaviour of that solution at infinity, we can heuristically describe the phase transition exhibited by these models in the $\theta := \ell/x$ regime. Following arguments in~\cite[Section~5.3]{Kimura_Zahabi_2-2021b}, we observe a discontinuity in the third derivative of $\mathcal{F}$ at $x=b$, but the scaling exponent of $\mathcal{F}$ in $x$ generalises from 3 to $2+ 1/m$ above $b$. In particular, by integrating~\eqref{eq:piisolatinfty} one can approximate $\mathcal{F}$ for $x$ close to $b$ as $\ell \to \infty$ with
\begin{equation}
\frac{1}{\ell^2}\log Z_\ell = \begin{cases}
\mathcal{F}_c + C d^{-\frac{1}{m}}|x-b|^{2 + \frac{1}{m}} + O(\ell^{-2})\qquad& \text{as } x \to b^- \\
\mathcal{F}_c + O(e^{-c\ell})& \text{as } x \to b^+
\end{cases}
\label{eq:Zellasy}
\end{equation}
for constants $\mathcal{F}_c, C$ and $c$. The same phase transition was observed in~\cite{Periwal_Shevitz_1990_2}, and predicted for momentum space flat trap models in~\cite{LDMS_2018}.  Note that our derivation of~\eqref{eq:Zellasy} is not rigorous as Theorem~\ref{thm:multicritical} only holds for fixed $s$, and here we extrapolate that it remains true when we take $s$ of order $\theta^{\frac{2m}{2m+1}}$, upon taking the appropriate asymptotics for $F_{2m+1}(s)$.

Recently, Chouteau and Tarricone~\cite{Chouteau_Tarricone_2022} proved that the partition functions $Z_\ell$ for densities $p^m_{\theta,\ell}$ (or indeed $\P^m_\theta(\lambda_1 \leq \ell)$ for integer $\ell$) satisfy recurrence equations which form a discrete analogue to the Painlevé II hierarchy. In the case of the minimal densities $p^{\mathrm{a},m}_{\theta,\ell}$ (and distributions $\P^{\mathrm{a},m}_\theta(\lambda_1 \leq \ell)$), they showed that the recurrence relation reduces to the order $2m$ Painlevé II equation as $\theta \to \infty$ with $\ell \sim b\theta + (d\theta)^\frac{1}{2m+1}s$, giving yet another path to Periwal and Shevitz's result.

\section{Formulation in terms of lattice fermions}
\label{sec:latticegen}

In this section we reformulate the probability laws on partitions and
results introduced above in more physical terms. The Schur measures
form \emph{determinantal point processes}, via a bijection between
partitions and certain infinite sets~\cite{Okounkov_2001}. Physically,
these sets can be interpreted as basis states for a quantum-mechanical
system of fermions on a unidimensional lattice. In this language,
Schur measures map to ground states of certain free Hamiltonians
(by free we mean that the fermions are non-interacting or, more
precisely, that they interact with one another only via the
Pauli exclusion principle). In Section~\ref{sec:lattice} we construct
these free fermion models in the second quantisation formalism, then
relate them back to Schur measures. In
Section~\ref{sec:fermions-asymptotics}, we informally discuss
asymptotic regimes for these models corresponding to the ones of
Theorems~\ref{thm:multicritical} and~\ref{thm:mclimitshape}, and
identify criteria for asymptotic edge behaviour coinciding with that
of momentum space models of fermions in ``flat traps'' on a line
previously studied, and dubbed ``multicritical'', by Le Doussal,
Majumdar and Schehr~\cite{LDMS_2018}.

\subsection{From lattice free fermions to Hermitian Schur measures} \label{sec:lattice}

We consider a system of free fermions on a unidimensional
lattice. For later convenience we label the lattice sites by
half-integers $\pm \tfrac12, \pm \tfrac32, \ldots$. We 
introduce, for any site $k$,
the creation operator $c_k^\dagger$ and the annihilation operator
$c_k$. These operators satisfy the canonical anticommutation relations
\begin{equation} \label{eq:anticomm}
\{ c_k, c^\dagger_\ell\} = \delta_{k\ell} \, , \qquad \{c_k,c_\ell\} = \{c^\dagger_k,c^\dagger_\ell\} = 0,
\end{equation}
with $\{\cdot,\cdot\}$ denoting the anticommutator and $\delta_{k\ell}$ denoting the indicator function for $k=\ell$. We denote by
$\ket{\emptyset}$ the domain-wall state with every positive site empty
and every negative site occupied. In other words, we have
$c_k \ket{\emptyset} = c_{-k}^\dagger \ket{\emptyset} = 0$ for all
$k>0$. The fermions are placed in a linear potential which, in
dimensionless units, corresponds to a second-quantized Hamiltonian of
the form
\begin{equation}
  \mathcal{H}_0 := \sum_{k} k \normord{c_k^\dagger c_k}.
\end{equation}
Here, $\normord{\cdot}$ denotes the normal ordering
\begin{equation}
  \normord{c_i^\dagger c_j} \ := c_i^\dagger c_j - \bra{\emptyset} c_i^\dagger c_j \ket{\emptyset}
\end{equation}
with respect to the domain-wall state $\ket{\emptyset}$,
which is clearly the ground state of $\mathcal{H}_0$.
We now modify the model by adding kinetic hopping terms of the form
\begin{equation}
  \label{eq:ardef}
  a_r := \sum_{k} \normord{c^\dagger_k c_{k+r}}
\end{equation}
for each integer $r$. 
More precisely, we choose a collection of complex parameters $t_r$,
$r \geq 1$ 
and we introduce the unitary operator
\begin{equation}
  \mathcal{U}_t := 
  e^{\sum_{r \geq 1} (t_r a_r^\dagger - t_r^* a_r)}
\end{equation}
and the modified Hamiltonian
\begin{equation}
  \label{eq:Hunit}
  \mathcal{H}_t = \mathcal{U}_t \mathcal{H}_0 \mathcal{U}_t^{-1}.
\end{equation}
For simplicity, we will consider only the polynomial case where the $t_r$ have finite support. 
Using the commutation relations $[a_r,a_s^\dagger]=r \delta_{r,s}$ and
$[\mathcal{H}_0,a_r^\dagger]=ra_r^\dagger$ that follow from the
canonical anticommutation relations~\eqref{eq:anticomm}, we obtain
that $\mathcal{H}_t$ reads explicitly
\begin{equation}
  \label{eq:Hform}
  \mathcal{H}_t = \mathcal{H}_0 - \sum_{r \geq 1} r(t_r^* a_r + t_r a_r^\dagger) +
  \sum_{r \geq 1} r^2 |t_r|^2,
\end{equation}
i.e. $\mathcal{H}_t$ consists of a linear combination of the linear
potential $\mathcal{H}_0$ and of finite-range hopping operators, plus a
scalar term ensuring that the spectra of $\mathcal{H}_t$ and
$\mathcal{H}_0$ are equal.  By~\eqref{eq:Hunit}, the ground state of
$\mathcal{H}_t$ is given by
\begin{equation}
  \label{eq:Hgs}
  \ket{ \mathrm{g.s.}_t} := \mathcal{U}_t \ket{\emptyset} = e^{-\sum_{r \geq 1} r|t_r|^2/2} e^{\sum_{r \geq 1} t_r a_r^\dagger} \ket{\emptyset}.
\end{equation}
Here, we obtain the right-hand side by performing a normal ordering of
the operators $a_r$ and $a_r^\dagger$, noting that
$a_r \ket{\emptyset} = 0$.

It is instructive to reinterpret this discussion in the language of
the quantum mechanics of harmonic oscillators. In terms of the bosonic
operators $a_r$, the fermionic linear potential becomes
$\mathcal{H}_0= \sum_{r \geq 1} a_r^\dagger a_r + a_0^2/2$, i.e. it
corresponds to a collection of harmonic oscillators, up to the square
of the charge operator $a_0$. Then, the unitary operator $\mathcal{U}_t$
corresponds to a translation in position space, momentum space, or a
combination thereof. This creates the linear terms in the shifted
Hamiltonian $\mathcal{H}_t$. The translated ground state 
$\ket{\mathrm{g.s.}_t}$ is nothing but a \emph{coherent state}.

We now relate these considerations with the Schur measure. To a
partition of length $\ell$, we associate the fermionic state
$\ket{\lambda}$ obtained from the domain-wall state $\ket{\emptyset}$
by moving for each $i=1,\ldots,\ell$ the fermion initially at
position $-i+\frac12$ to the right by $\lambda_i$ sites, namely
\begin{equation} \label{eq:partitionstate}
  \ket{\lambda} := c^\dagger_{\lambda_1-\frac12} c^\dagger_{\lambda_2-\frac32} \cdots c^\dagger_{\lambda_\ell-\ell+\frac12}
  c_{-\ell+\frac12} \cdots c_{-\frac32} c_{-\frac12} \ket{\emptyset}.
\end{equation}
See Figures~\ref{fig:intro_young_diagrams} or~\ref{fig:ydtomaya} for examples of states associated to partitions.  The state $\ket{\lambda}$ is an
eigenstate of $\mathcal{H}_0$, with eigenvalue
$|\lambda|=\lambda_1+\lambda_2+\cdots+\lambda_\ell$, i.e. the {size} of the partition $\lambda$. Then, as shown in
Appendix~\ref{sec:rem} (see Lemma~\ref{lem:schurfromdpp}), the ground state
of $\mathcal{H}_t$ decomposes as
\begin{equation}
  \ket{\mathrm{g.s.}_t} = e^{-\sum_{r \geq 1} r |t_r|^2/2}
  \sum_\lambda s_\lambda[t_1,t_2,\ldots] \ket{\lambda}
\end{equation}
where $s_\lambda[t_1,t_2,\ldots]$ is the Schur function~\eqref{eq:jt1} evaluated at
the Miwa times $t_1,t_2,\ldots$. In other words, if
we could prepare the quantum state $\ket{\mathrm{g.s.}}_t$ and simultaneously measure the
occupation numbers of all sites of the lattice, then the probability
of observing the eigenstate $\ket{\lambda}$ would be equal to
\begin{equation}
  \label{eq:SM}
  \left\vert \braket{\lambda}{\mathrm{g.s.}_t} \right\vert^2
  = e^{-\sum_{r \geq 1} r|t_r|^2} s_\lambda[t_1,t_2,\ldots] s_\lambda[t^*_1,t^*_2,\ldots].
\end{equation}
Although a measurement of this kind is not physically meaningful, we recognize that the right hand side of~\eqref{eq:SM} is a well defined probability measure on integer partitions, namely a Hermitian Schur measure as defined in Section~\ref{sec:multicriticalmeasures}. 
The physical meaning of a general Schur measure where the Miwa times are not complex conjugate to one another is more
elusive. In the Hermitian case, we find an exact mapping
between a quantum model of fermions and a probabilistic model of
random partitions. 

Let us now consider a finite number of sites $k_1,\ldots,k_n$ (assumed
distinct). The probability that they are all simultaneously occupied
is the \emph{correlation function}
\begin{equation} \label{eq:correlation1}
  \rho_n(k_1,\ldots,k_n) := 
  \langle
  c_{k_1}^\dagger c_{k_1} \cdots c_{k_n}^\dagger c_{k_n} 
  \rangle_{t}
\end{equation}
where we use $\langle \cdot \rangle_{t}$ to denote an average $\bra{\mathrm{g.s.}_t} \cdot \ket{\mathrm{g.s.}_t}$ with respect to the ground state.
As reviewed in Appendix~\ref{sec:rem} (see Lemma~\ref{lem:dpp}), this correlation function is
given by the $n \times n$ determinant
\begin{equation}
  \rho_n(k_1,\ldots,k_n) = \det_{1 \leq i,j \leq n} K(k_i,k_j)
\end{equation}
where the \emph{correlation kernel} (or propagator) $K(\cdot,\cdot)$
is explicitly given by 
\begin{equation} \label{eq:hermitiankernel}
  K(k,\ell) =
  \langle c_k^\dagger c_\ell \rangle_{t} = 
  \sum_{m=0}^{\infty} J_{k+m+\frac12}(t_1,t_2,\ldots) J_{\ell+m+\frac12}(t^*_1,t^*_2,\ldots) 
\end{equation}
where $J_m(t_1,t_2,\ldots)$ is the \emph{multivariate Bessel function}
\begin{equation} \label{eq:mvbesselfn}
J_m(t_1,t_2,\ldots) := \frac{1}{2i\pi} \oint \frac{dz}{z^{m+1}} e^{\sum_{r \geq 1} (t_r z^r - t_r^* z^{-r})}
\end{equation}
(it reduces to the classical Bessel function $J_m(2\theta )$ when $t_1=\theta$ and all other $t_r$ are zero). Note that $K(k,\ell) = K(\ell,k)^*$, i.e.\
the kernel is Hermitian.

Let us finally mention that the discussion of this section, and in
particular the correspondence between fermions and bosons discussed
above, plays a key role in describing the solutions of the Kadomtsev--Petviashvili (KP) hierarchy of integrable differential equations, see
e.g.~\cite{djm}.

\subsection{Continuum limit and multicriticality}\label{sec:fermions-asymptotics}

\paragraph*{Motivation from momenta of trapped fermions on a line} Let us turn to an informal discussion of the  behaviour of lattice fermion models defined above in a continuum limit. The particular asymptotic regimes we consider are motivated by the \emph{multicriticality} phenomena found in~\cite{LDMS_2018}, where the authors considered the following model in the first quantisation formalism: $N$ non-interacting fermions in continuous unidimensional space are each subject to a ``flat trap'' single particle Hamiltonian (written in terms of a dimensionless position space coordinate)
\begin{equation}
H = - \frac{1}{2}\frac{d^2}{dx^2} + x^{2m}
\end{equation}
for integer $m\geq 1$ (recovering the harmonic oscillator at $m=1$). 
The potential confines the particles to a region around the origin, in terms of the Fermi energy $E_F$ there is a right hand side edge at position $\xedge = {E_F}^{1/2m}$. Looking at particles in a small window around  $\xedge$, scaling with $\xedge^{-1/3}$, the potential they experience may be approximated by a linear one. The behaviour in this window is universal for fermions at the edge of confining traps: as argued by Eisler~\cite{Eisler_2013}, and recently proven via rigorous semi-classical analysis by Deleporte and Lambert~\cite{Deleporte_Lambert_2021},  as $N\to \infty$, the fluctuations in the position ${x}_{\max}$  of the rightmost fermion around $\xedge$ are at a scale of $\xedge^{-1/3}$ and governed by the TW-GUE distribution. 

 If instead we write the flat trap Hamiltonian in momentum coordinates $p$,  it reads\footnote{The coordinates used in~\cite{LDMS_2018} have dimensions, with momentum space Hamiltonian $H = (-1)^{m}\hbar^{2m} g \tfrac{d^{2m}}{dp^{2m}} + \tfrac{1}{2M} p^2$ for a coupling $g$ and particle mass $M$.}
\begin{equation} \label{eq:flattrapham}
H =  (-1)^{m} \frac{d^{2m}}{dp^{2m}} + \frac{1}{2}p^{2}.
\end{equation} 
There is an edge in momentum space too, at $\pedge = {2E_F}^{1/2}$. Looking near the edge in  coordinates $\tilde{p} = (p-\pedge)/\kappa$, in a critical scaling regime\footnote{To use the conventions of~\cite{LDMS_2018}, $\kappa$ should be replaced with $p_N = \hbar \big(\frac{Mg}{\hbar \pedge}\big)^{\frac{1}{2m+1}}$.} $\kappa =\pedge^{-1/(2m+1)}$ we may linearise the quadratic kinetic energy to obtain the edge Hamiltonian
\begin{equation} \label{eq:edgeham}
H_{\mathrm{edge}} := (-1)^m \frac{d^{2m}}{d\tilde{p}^{2m}} + \tilde{p} 
\end{equation}
which satisfies $H_{\mathrm{edge}} = \pedge^{-\frac{2m}{2m+1}} H + O(\pedge^{-\frac{2m}{2m+1}} (p-\pedge)^{2})$. The square integrable eigenfunctions of this operator are given by the higher-order Airy functions defined at~\eqref{eq:genairydef}, as we have
\begin{equation} \label{eq:genairyevrel}
H_{\mathrm{edge}} \Ai_{2m+1}(x+v) = \bigg( (-1)^m \frac{d^{2m}}{dx^{2m}} + x\bigg) \Ai_{2m+1}(x+v) = -v\Ai_{2m+1}(x+v). 
\end{equation}
The edge Hamiltonian has an unbounded linear spectrum. 
In~\cite{LDMS_2018}, the authors found that in a system of $N$ fermions under the order $m$ flat trap Hamiltonian~\eqref{eq:flattrapham},  the fluctuations in the maximum fermion momentum ${p}_{\max}$ are asymptotically governed by $F_{2m+1}(s)$ as $N \to \infty$, with
\begin{equation}
\P\bigg[ \frac{{p}_{\max} - \pedge}{\pedge^{-\frac{1}{2m+1}}} < s\bigg] \to F_{2m+1}(s) := \det (1-\mathcal{A}_{2m+1})_{L^2([s,\infty))}.
\end{equation}

\paragraph*{Parametrised lattice fermion models } We now propose a discrete counterpart of the above model. Starting from the lattice fermion models of Section~\ref{sec:lattice}, we introduce a parameter $\theta>0$ and consider  Hamiltonians 
\begin{align}
\mathcal{H}_{\theta \gamma} &=  \sum_{r \geq 1} \bigg[ a_{-r} a_r - \theta r \gamma_r (a_r + a_{-r}) + \theta^2 \gamma_r^2 r^2 \bigg]+ a_0^2/2  \notag \\
&= \sum_{k}\bigg[ k  \normord{c^\dagger_kc_k}- \sum_{r \geq 1}\theta r \gamma_r  \big(c^\dagger_kc_{k+r} + c^\dagger_kc_{k-r} \big) \bigg] + \theta^2 \sum_{r \geq 1 } r^2 \gamma_r^2 
\end{align}
where $\gamma = (\gamma_1,\gamma_2,\ldots)$ is a real finite sequence with $\gamma_1>0$. From~\eqref{eq:SM}, the distribution of the positions of ground state fermions under $\mathcal{H}_{\theta\gamma}$ can be expressed in terms of the Schur measure 
$\P(\lambda) = e^{-\theta^2\sum_r r\gamma_r^2}s_\lambda[\theta \gamma]^2$, and we consider $\mathcal{H}_{\theta\gamma}$ as $\theta$ grows large, first in a regime corresponding to that of Theorem~\ref{thm:mclimitshape} then in one corresponding to Theorem~\ref{thm:multicritical}.

\paragraph*{The bulk} First, consider a macroscopic scale of $\theta$, and more precisely consider lattice positions scaling as  $k \sim x\theta $ for finite $x$. In this regime we consider the \emph{ground state limiting density profile}
\begin{equation}
  \varrho(x) = \lim_{\theta \to \infty} \langle c^\dagger_{x\theta } c_{x\theta } \rangle_{\theta\gamma}
\end{equation} 
where $\langle \cdot \rangle_{\theta \gamma}$ denotes the expectation on the ground state of $\mathcal{H}_{\theta \gamma}$; then the probability of finding a particle in $[x,x+dx]$ is $\varrho(x)dx$. Following for instance~\cite[Section 1]{Allegra_2016} and~\cite{Bocini_Stephan_2020}, let $\hat{c}^\dagger(\xi) = \sum_{k} e^{i k \xi} c^\dagger_k$ be the Fourier transform of the fermionic creation operator; then, under the assumptions of a  \emph{local density approximation} (see e.g.~\cite{Stephan_2019}) where $\delta,\delta'$ are at a scale much smaller than the system size but much bigger than the typical gap between particles, the propagator at that scale experiences the potential as a fixed Fermi energy, which limits the Fourier frequencies to $\xi \in [-\chi,\chi]$ for some $\chi:= \chi(x) $ between $0$ and $\pi$, once we make the simplifying assumption that the Fourier frequencies have compact support for all $x$. Hence we have
\begin{equation}
  \langle c^\dagger_{x\theta + \delta } c_{x\theta +\delta' } \rangle_{\theta\gamma} \approx  \int_{-\chi}^{\chi}\frac{1}{2\pi} e^{i\xi (\delta - \delta')} d\xi= \frac{\sin \chi (\delta - \delta')}{\pi (\delta - \delta')},
 \end{equation}
 where ``$\approx$'' indicates a local density approximation for  large $\theta$; in the limit as $\delta \to \delta'$, we find $\varrho(x) = \chi(x)/\pi$.

Under these assumptions, we consider a local estimate of $\mathcal{H}_{\theta \gamma}$ experienced at sites in a small region around $x\theta$ for large $\theta$; with the linear potential approximated as constant, a Fourier transformation gives
\begin{align}
\mathcal{H}_{\gamma\theta} \approx \int_{-\pi}^\pi \big(x - \sum_{r\geq 1}2 r \gamma_r \cos r \xi\big) c^\dagger(\xi)c(\xi)d\xi.
\end{align}
The ground state of this Hamiltonian corresponds to the  frequencies $\xi \in [-\chi,\chi]$ where 
\begin{equation} \label{eq:heuristicbulk}
x  - \sum_{r\geq 1}2 r \gamma_r \cos r \chi  = 0,
\end{equation} which in turn gives an explicit formula for the limiting density, as $\varrho(x) = \chi/x$. In terms of the function $\Omega(x)$ defined at~\eqref{eq:multicritlimitprofile} as the limiting profile of the corresponding Schur measure, we have $\varrho(x) = \frac{1}{2} - \frac{1}{2}\Omega'(x)$. The compact support assumption for the Fourier frequencies means there is a unique  $\chi\geq 0 $ satisfying~\eqref{eq:heuristicbulk}, and for this to be true the sequence $\gamma$ must satisfy $\sum_{r} r^2 \gamma_r \sin r\phi \geq 0$ for all $\phi \in [0,\pi]$, which is the condition~\eqref{eq:singlefermiseacondition}. One immediate consequence is that~\eqref{eq:heuristicbulk} has a solution only for
\begin{equation}
2\sum_{r\geq 1} (-1)^r r \gamma_r \leq x \leq  2\sum_{r\geq 1} r \gamma_r,
\end{equation}
or concisely for $x\in[-\tilde{b},b]$ in terms of the constants defined at~\eqref{eq:coefficientsmcdef}; to the left of this region, for $x < -\tilde{b}$, we have $\varrho(x) = 1$  and the right, $x > b$, we have $\varrho(x) = 0$. We revisit this formally in Section~\ref{sec:mclimitshape}.

\paragraph*{The edge} Let us turn our attention to the fluctuations around the right edge, and consider the Hamiltonian in a microscopic critical scaling regime. In particular, we take an ansatz $k \sim b \theta + x(d\theta)^{1/(2m+1)}$ for the critical regime, where $m$ is a positive integer, $x$ is a finite parameter and $d$ is a constant. Then, writing $\tilde{c}^\dagger_x := c^\dagger_{k}$, we treat the kinetic hopping terms in this regime just by Taylor expanding, with 
\begin{equation}
c_{k+r}^\dagger = \tilde{c}^\dagger_{x + r (d\theta)^{-1/(2m+1)}} = \sum_{n=0}^{\infty} r^n (d\theta)^{-\frac{n}{2m+1}} \frac{d^n}{dx^n}\tilde{c}^\dagger_{x}.
\end{equation}
In the expanded Hamiltonian, all of the odd derivatives cancel each other out. If we require that the first $m-1$ even derivatives are cancelled out but the $m$-th one has a non-zero coefficient, the sequence $\gamma$ must satisfy $\sum_r r^{2p+1} \gamma_r = 0$ for $p = 1, \ldots, m-1$ and $\sum_r r^{2m+1} \gamma_r \neq 0$, i.e. the multicriticality conditions~\eqref{eq:multicriticalitycondition}; fixing $d$ to the fluctuation coefficient in~\eqref{eq:coefficientsmcdef}, we have
\begin{align} 
\mathcal{H}_{\theta\gamma} = (d\theta)^{\frac{1}{2m+1}} \int_{\mathbb{R}}  \normord{ \tilde{c}^\dagger_x \bigg[x+ (-1)^m \frac{d^{2m}}{dx^{2m}} \bigg] \tilde{c}_x } dx + O(\theta^{-\frac{2m+2}{2m+1}}), \label{eq:continuummulticrit}
\end{align}
As $\theta \to \infty$, we see that $(d\theta)^{-\frac{1}{2m+1}} \mathcal{H}_{\theta\gamma} $ should coincide precisely with the Hamiltonian $H_{\mathrm{edge}}$ of~\eqref{eq:edgeham}. This heuristic is a meaningful one:   one would expect as a consequence that the multivariate Bessel wave functions of the lattice model coincide with the higher-order Airy functions of the flat trap edge potential in this asymptotic regime, and that the kernels (or ground state propagators) would be asymptotically equivalent in turn; see~\cite{Kimura_Zahabi_2021}. We confirm this rigorously in Section~\ref{sec:mcedge_fluctuations}.

\section{Asymptotic analysis of multicritical Schur measures} \label{sec:mc_proofs}
In this section we prove our main results, Theorems~\ref{thm:multicritical} and~\ref{thm:mclimitshape}. Our strategy exploits the known connection between Schur measures and determinantal point processes (DPP). In Section~\ref{sec:kernel_regimes} we write down an integral expression for the DPP correlation kernel corresponding to a multicritical measure $\P_\theta^m$, and review Okounkov's general framework for its asymptotic analysis. In Sections~\ref{sec:mclimitshape} and~\ref{sec:mcedge_fluctuations} we analyse in detail the limiting DPPs corresponding to the bulk and edge behaviour respectively, from which we prove each theorem. 

\subsection{Determinantal point process formulation and the higher-order Bessel kernels} \label{sec:kernel_regimes}

In order to write explicit expressions for marginal statistics  of partitions under Schur measures, we apply the lattice fermion formulation presented in Section~\ref{sec:lattice}. To each partition $\lambda$ we associate the infinite set of half-integers
\begin{equation} \label{eq:halfintegerset}
S(\lambda) = \{\lambda_i - i + \tfrac{1}{2}, i \in \Z_{>0}\},
\end{equation}
corresponding to the occupied sites in the state $\ket{\lambda}$ defined at~\eqref{eq:partitionstate}. If $\lambda$ is distributed according to a multicritical Schur measure $\P_\theta^m(\lambda)$, the $n$-point correlation function $\rho_n(k_1,\ldots,k_n)$ as defined at~\eqref{eq:correlation1} is 
\begin{equation} \label{eq:mcdpp}
\rho_n(k_1,\ldots,k_n) = \P_\theta^m( \{k_1,\ldots,k_n\}\subset S(\lambda)) = \det_{1\leq i,j \leq n} \mathcal{J}_{\theta}^m(k,\ell)
\end{equation}
in terms of the correlation kernel 
\begin{equation}
\mathcal{J}_{\theta}^m(k,\ell) := \sum_i J_{k+i+ \frac12}(\theta\gamma) J_{\ell+i+ \frac12}(\theta\gamma).
\end{equation}
Note that this is the correlation kernel~\eqref{eq:mvbesselfn} where we take the Miwa times at their multicritical values.  We call $\mathcal{J}_{\theta}^m(k,\ell)$ an order $m$ Bessel kernel, and it has an equivalent contour integral expression 
\begin{align} 
\mathcal{J}_{\theta}^m(k,\ell) &= \frac{1}{(2\pi i)^2} \oiint \frac{\exp[\theta \sum_r \gamma_r(z^r - z^{-r})]}{\exp[\theta \sum_r \gamma_r(w^r - w^{-r})]} \frac{1}{z-w} \frac{dzdw}{z^{k+\frac{1}{2}} w^{-\ell+\frac{1}{2 }}}  \label{eq:jthetacontourint}
\end{align}
where the integral in $w$ is taken counter-clockwise along a contour $c_-$ enclosing the origin and the integral in $z$ is taken counter-clockwise over a contour $c_+$ enclosing $c_-$ (see Lemma~\ref{lem:dpp} in Appendix~\ref{sec:rem}). Equation~\eqref{eq:mcdpp} states that the set $S(\lambda)$ associated with a random partition $\lambda$ under $\P_\theta^m$ forms a determinantal point process (DPP; see e.g.~\cite{Hardy_Maida_2019} for a short introduction) with kernel $\mathcal{J}_{\theta}^m$. This is simply a special case of~\cite[Theorem~1]{Okounkov_2001}, which  we restate and prove in Appendix~\ref{sec:rem} for convenience (see Theorem~\ref{thm:schurdpp}).

 The asymptotic statistics described by Theorems~\ref{thm:multicritical} and~\ref{thm:mclimitshape} are extracted from the large $\theta$ limits of $\mathcal{J}_{\theta}^m$ in appropriate regimes. On the one hand, the cumulative distribution of $\lh_1$ is equal to a gap probability on $S(\lh)$. By the inclusion-exclusion principle, we have
\begin{align}
\P_\theta^m(\lh_1 < k+ \tfrac12) &=  
 \sum_{n=0}^\infty \frac{(-1)^n}{n!} \sum_{j_1 = k+\frac12}^{\infty} \cdots \sum_{j_n = k+\frac12}^{\infty}\rho_n(j_1,\ldots,j_n) =\det(1- \mathcal{J}_{\theta}^m)_{\ell^{2}([k,\infty))} \label{eq:discretegapprob}
\end{align}
where the final expression is a discrete Fredholm determinant, 
since $\P_\theta^m(\lh_1 < k+ \tfrac12)$ is just the probability of finding no element greater than $k-\tfrac12$ in $S(\lambda)$.
On the other hand, noting that at integer values of $\theta x $ we have  
\begin{equation} \label{eq:profileconfiguration}
\psi_{\lambda;\theta}(x) = x + \frac{2}{\theta} \cdot \# \{k \in S|k>\theta x\}, \qquad x \in \tfrac{1}{\theta} \Z,
\end{equation} we can see that the one-point function
\begin{equation}
\P_\theta^m(k \in S(\lh)) = \rho_1(k) = \mathcal{J}_{\theta}^m(k,k) 
\end{equation} 
gives the expectation of the rescaled profile of $\lh$ at points a distance $1/\theta$ apart, with 
\begin{equation} \label{eq:expprofile}
\E\big(\psi_{\lh,\theta}(x)\big) = x + \frac{2}{\theta} \sum_{k =  x \theta + \frac12}^\infty \mathcal{J}_{\theta}^m(k,k) , \qquad x \in \frac{1}{\theta} \Z. 
\end{equation}

\paragraph*{Action notation } In order to analyse the asymptotics of
the double contour integral \eqref{eq:jthetacontourint} it is
convenient to introduce some notation. We define the \emph{potential}
as
\begin{equation}
V(z) := \sum_{r \geq 1} \gamma_r z^r, 
\end{equation}
(it is always a polynomial in the cases we consider) and the \emph{action} as
\begin{equation} \label{eq:actiondef} S(z;x) := \sum_{r \geq 1}
  \gamma_r z^r - \sum_{r \geq 1} \gamma_r z^{-r} - x \log z = V(z) -
  V(z^{-1})- x \log z.
\end{equation}
We will be interested in the large $\theta $ behaviour of $\mathcal{J}_{\theta}^m(k,\ell)$ at points $k = x\theta + k'$, $\ell = x\theta + \ell'$ where $x$ is finite and $k',\ell'$ are sublinear in $\theta$. Then, we have
\begin{equation}
  \label{eq:intaction}
\mathcal{J}_{\theta}^m(k,\ell) = \frac{1}{(2\pi i)^2} \oiint e^{\theta [S(z;x) - S(w;x)]}\frac{dzdw}{z^{k'+\frac12}w^{-\ell'+\frac12}(z-w)}.
\end{equation}
We will analyse this integral in the large $\theta$ limit by the
method of steepest descent, following the approach outlined
in~\cite{Okounkov_2003} which applies to double contour integrals of
this specific form. We start by taking as integration contours two
circles of radius close to $1$, namely $c_+:\,|z| = 1+\delta$ and
$c_-:\,|w| = 1- \delta$ for a small $\delta>0$. As we will discuss in
the following subsections, the asymptotics is then governed by the
saddle points of the function $z \mapsto S(z;x)$ on the unit
circle. For now, let us just discuss whether such saddle points exist.

Let us introduce the even periodic function
\begin{equation}
  D(\phi):= \sum_r 2 r \gamma_r \cos r \phi = \frac{d}{d\log z} \left( V(z)-V(z^{-1})\right) \bigg|_{z = e^{i\phi}}.
\end{equation}
The saddle points of the action on the unit circle then correspond to
real solutions of the equation $D(\phi)=x$. Now, by the second
requirement~\eqref{eq:singlefermiseacondition} in the definition of a
multicritical Schur measure, we have
\begin{equation} \label{eq:dDnonpositive}
 \frac{d}{d \phi} D(\phi) \leq 0 \qquad \text{for all } \phi \in [0,\pi]
\end{equation}
and $\phi \mapsto D(\phi)$ is decreasing over $[0,\pi]$, with range
$[-\tilde{b},b]$ given by \eqref{eq:coefficientsmcdef}. We may
therefore distinguish three generic regimes, discussed in details in
Section~\ref{sec:mclimitshape}:
\begin{enumerate}[{\normalfont\sffamily\bfseries (i)}]
\item \emph{The empty region:} For $x>b$, $z \mapsto S(z;x)$ has no
  saddle points on the unit circle. We will find that, keeping the
  integration contours as they are, the integral~\eqref{eq:intaction}
  decays exponentially fast to zero.
\item \emph{The bulk:} For $-\tilde{b} < x < b$, $z \mapsto S(z;x)$
  has two saddle points on the unit circle, namely
  $z_\pm = e^{\pm i \chi}$ where $D(\chi) = x$ and $0<\chi<\pi$. Upon
  doing an appropriate deformation of the contours, we will find that
  the integral~\eqref{eq:intaction} has a finite limit depending on
  $\chi$.
\item \emph{The frozen region:} For $x<-\tilde{b}$, $z \mapsto S(z;x)$
  again has no saddle points on the unit circle, but we will have to
  interchange the contours of integration for $z$ and $w$ in order to
  obtain an exponentially decaying integral. We will then find that
  the kernel $\mathcal{J}_{\theta}^m(k,\ell)$ tends exponentially fast
  to $\delta_{k,\ell}$.
\end{enumerate}

In the cases $x=b$ and $x=-\tilde{b}$, the action has just one saddle
point on the unit circle, at $z=1$ and $z=-1$ respectively. This
saddle point is of multiplicity at least two, since it is formed at
the coalescence of (at least) the two bulk saddle points
$e^{\pm i \chi}$ discussed above, for $x \to b^-$ and
$x \to (-\tilde{b})^+$, respectively.

\paragraph*{Multicritical actions and minimal measures}
In fact, the multicriticality conditions of Definition~\ref{def:msm}
determine, for $x=b$, the multiplicity of the saddle point at
$z=1$. Indeed, we observe that, for any $q=1,\ldots,2m$, we have
\begin{equation}
  \label{eq:multicritaction}
  \left. \frac{d^q}{d(\log z)^q} S(z;b)\right\vert_{z=1} = 0.
\end{equation}
This follows for $q=1$ from the
definition~\eqref{eq:coefficientsmcdef} of $b$, for $q$ even from the
symmetry $S(z;x)=-S(z^{-1};x)$ of the action, and for $q$ odd between
$3$ and $2m-1$ from the linear
conditions~\eqref{eq:multicriticalitycondition}. The saddle point has
multiplicity exactly $2m$ since, by~\eqref{eq:singlefermiseacondition}
and \eqref{eq:coefficientsmcdef}, we have
\begin{equation}
  \label{eq:actionnonvanishing}
  (-1)^{m+1} (2m)! d = \left. \frac{d^{2m+1}}{d(\log z)^{2m+1}} S(z;b)\right\vert_{z=1} \neq 0.
\end{equation}
In view of exponential factor $e^{\theta S(z;b)}$ appearing
in~\eqref{eq:intaction}, it is natural to expect that the integral is
dominated by the region in which $z-1$ is of order
$\theta^{-1/(2m+1)}$, and similarly for $w$. This will be discussed in
detail in Section~\ref{sec:mcedge_fluctuations}.

As for the case $x=-\tilde{b}$, we have at $z=-1$ a saddle point of
multiplicity $2\tilde{m}$, where $\tilde{m}$ is the order of
multicriticality of the conjugate measure given by
Proposition~\ref{prop:mcconjugation}. We conclude the present section by proving the following:
\begin{prop}[Multicriticality of the minimal measures] \label{prop:minimal}
The measures $\P^{\aa,m}_\theta$ and  $\P^{\ss,m}_\theta$ are order $m$ multicritical.
\end{prop}

\begin{proof}
From the coefficients~\eqref{eq:min_parameters_a}, one can recognise that the log derivative of the action associated with $\P^{\aa,m}_\theta$ at the right edge $b$ is a binomial series, which sums to 
\begin{equation} \label{eq:asymmetricaction}
\frac{d}{d\log z}S(z;b) = (-1)^{m+1}\binom{2m}{m-1}^{-1} (z^{1/2} - z^{-1/2})^{2m}. 
\end{equation}
It is immediately clear that the first $2m-1$ derivatives of this function disappear at $z=1$, and it follows that the action satisfies~\eqref{eq:multicritaction}; similarly, it follows from the fact that $(2m)$-th derivative of this function does not vanish  at $z=1$ that the action satisfies~\eqref{eq:actionnonvanishing}. Evaluating the log derivative at $z = e^{i\phi}$ we have
\begin{equation}
D(\phi) - b = -4^m \binom{2m}{m-1}^{-1}  \sin^{2m} \frac{\phi}{2}
\end{equation}
and its derivative is
\begin{equation}
\frac{d}{d\phi}D(\phi) = - 4^m m \binom{2m}{m-1}^{-1} \sin^{2m-1} \frac{\phi}{2} \cos \frac{\phi}{2}
\end{equation}
which is non-positive for $\phi \in [0,\pi]$ as required. 

From the coefficients~\eqref{eq:min_parameters_s}, the second log derivative of the action associated with $\P^{\ss,m}_\theta$ is a binomial series, summing to
\begin{equation} \label{eq:symmetricaction}
\frac{d^2}{d \log z^2}S(z;x) = (-1)^{m+1}\binom{2m-1}{m-1}^{-1}  (z - z^{-1})^{2m-1}
\end{equation}
for any $x$. Since its first $2m-2$ derivatives vanish at $z=1$ and its $(2m-1)$-th derivative is non-zero, $S(z;x)$ satisfies~\eqref{eq:multicritaction} and~\eqref{eq:actionnonvanishing}. Inserting $z = e^{i\phi}$ we have
\begin{equation}
\frac{d}{d\phi} D(\phi) = i \frac{d^2}{d \log z^2} S(z;x)\big\vert_{z=e^{-\phi}} = - \frac{4^m}{2} \binom{2m-1}{m-1}^{-1} \sin^{2m-1} \phi
\end{equation}
which is non-positive for $\phi \in [0,\pi]$ as wanted. Hence, each family of  minimal measures is meets the requirements of Definition~\ref{def:msm}. 
\end{proof}
 
The limit shapes for $\P^{\aa,m}$ and $\P^{\ss,m}$ given in Corollaries~\ref{corr:asymmmin} and~\ref{corr:symminls} respectively are obtained by inverting the corresponding functions $D(\phi)$ given above.

\subsection{Limit shapes} \label{sec:mclimitshape}
To prove Theorem~\ref{thm:mclimitshape}, we start with a general limit shape result:
\begin{lem}
\label{lem:convergencels}
Let $\lh$ be a random partition under any Hermitian Schur measure with a single positive parameter $\theta$, such that $\P^\gamma_\theta(\lambda) = e^{-\theta^2\sum_{r}r |\gamma|^2}s_\lambda[\theta \gamma] s_\lambda[\theta {\gamma}^*]$ for some non-zero sequence of complex coefficients $\gamma$. Then, if there exists a curve $\Omega$ such that $\Omega(x)-|x|$ has finite support and such that for all $x$ we have $\E(\psi_{\lh,\theta}(x)) \to \Omega(x)$ as $\theta \to \infty$, we also have
\begin{equation}
\sup_{x}| \psi_{\lh,\theta}(x)- \Omega(x)| \xrightarrow{p} 0. 
\end{equation}
\end{lem}
\begin{proof}
Let $K_\theta$ denote the Hermitian kernel given at~\eqref{eq:hermitiankernel} with $t = \theta \gamma$. Then,  $S(\lh)$ forms a DPP with kernel $K_\theta$. Setting
\begin{equation}
{N}(n) := \#\{k\in S(\lh)|k>n\}),
\end{equation}
the expectation and variance of ${N}(n)$ may be expressed in terms of $K_\theta$ as
\begin{equation}
\E({N}(n)) =\tr_{(n,\infty)} K_\theta, \qquad \Var({N}(n)) = \tr_{(n,\infty)} (K_{\theta} - K_{\theta}^2) .
\end{equation}
Since $K_{\theta}$ is Hermitian, we have $\tr_{(n,\infty)} K_{\theta}^2 \geq 0$ and hence
\begin{equation}
\Var({N}(n)) \leq \E({N}(n)).
\end{equation} 
Now, considering a regime where $n = x\theta$, we set $\tilde{{N}} (x\theta) = {N}(x\theta)/\theta$, so 
\begin{equation}
\Var (\tilde{N}(x\theta) ) = \theta^{-2}\Var ({N}(x\theta)) \leq \theta^{-1} \E(\tilde{{N}} (x\theta)). 
\end{equation}
Suppose that there exists a fixed function $\Omega$ such that we have a limit
\begin{equation}
\Omega(x) = \lim_{\theta \to \infty} \E(\psi_{\lh,\theta}(x)) = x + 2 \lim_{\theta \to \infty} \E(\tilde{{N}} (x\theta))
\end{equation}
(we recall the expression~\eqref{eq:profileconfiguration} for the profile). Then, we have
\begin{equation}
\Var(\psi_{\lh,\theta}(x)) \leq  \theta^{-1} 2 \E(\tilde{{N}} (x\theta)) \xrightarrow{\theta \to \infty} 0
\end{equation} 
and we have pointwise convergence in probability for $\psi_{\lh,\theta}(x)$.

This in turn implies the convergence of the supremum norm.  Let $I \subset \R$ be a bounded interval and let $I_\eps$ be the set $I \cap \eps \Z$; then,  by the 1-Lipschitz property of $\psi_{\lh,\theta}$ we have for each $\eps>0$
\begin{equation}
\P\left( \sup_{x\in I} |\psi_{\lh,\theta}(x)- \Omega(x)| > \eps\right) \leq \P\left( \sup_{x\in I_\eps} |\psi_{\lh,\theta}(x)- \Omega(x)| > \frac{\eps}{2}\right).
\end{equation}
On right hand side, the supremum is over a finite set, so the convergence to zero at each $x \in I_\eps$ implies the convergence of the supremum to zero, in turn implying that the supremum norm over all $I$ converges to zero in probability, and in particular we have convergence over the support of $\Omega(x)-|x|$. To extend to all $\R$, we reapply the 1-Lipschitz property: choose $a>0$ such that $[-a,a]$ contains the support of $\Omega(x)-|x|$, then
\begin{equation}
\sup_{x\in(a,\infty)}| \psi_{\lh,\theta}(x)- \Omega(x)| = \sup_{x\in(a,\infty)}| \psi_{\lh,\theta}(x)- |x| | \leq \psi_{\lh,\theta}(a)- a 
\end{equation}
 and the final term converges to zero in probability, completing the proof. 
\end{proof}

With that, we need only find the limiting expectation of the rescaled profile under a multicritical measure. First, we have a limiting kernel and one-point function:
\begin{lem}\label{lem:density}
For finite integers $s,t$, as $\theta \to \infty$ we have
\begin{equation} \label{eq:bulkkernel}
\mathcal{J}_\theta^m(\lfloor x\theta \rfloor + s-\tfrac12,\lfloor x\theta \rfloor + t - \tfrac12) \to  \begin{cases} \delta_{st},  &x < -\tilde{b} \\
\frac{\sin \chi(x)(s-t)}{\pi(s-t)}, \qquad & x \in [-\tilde{b} ,b] \\
0, &x>b
\end{cases}
\end{equation}
where $\chi(x)$ is the unique non-negative solution of $\sum_{r} 2 r \gamma_r \cos r \chi(x) = x $, uniformly for $x$ in compact subsets of $\R$ and $s,t$ in compacts of $\Z$.
If $\lh$ is a random partition under $\P^m_\theta(\lambda) = e^{-\theta^2\sum_r r \gamma_r^2} s_\lambda[\theta \gamma]^2$,
as $\theta \to \infty$ we have
\begin{equation}
\P(\lfloor x\theta \rfloor -\tfrac12 \in S(\lh)) = \rho_1(\lfloor x\theta \rfloor -\tfrac12) \to \varrho(x) = \begin{cases} 1,  &x < -\tilde{b} \\
\frac{\chi(x)}{\pi},  \qquad & x \in [-\tilde{b} ,b] \\
0, &x>b.
\end{cases}
\end{equation}
\end{lem}

We state the limiting kernel itself because it is the universal aspect of the asymptotic bulk behaviour: while the limiting density profile $\varrho$ depends on the specific choice of coefficients $\gamma$, the \emph{discrete sine kernel} on the right of~\eqref{eq:bulkkernel} for points a finite distance apart in the bulk, and even more universal than the edge behaviour of Theorem~\ref{thm:multicritical} since it does not depend on the order of multicriticality $m$.

\begin{proof}
The expression for the limiting density follows directly from the $s = t$ case of the limiting kernel, so we only need to find the limit as $\theta \to \infty$ of
\begin{align}
\mathcal{J}_{\theta}^m(\lfloor x\theta \rfloor + s-\tfrac12,\lfloor x\theta \rfloor + t - \tfrac12) &= \frac{1}{(2\pi i)^2} \oiint_{c_+,c_-} \frac{e^{\theta [S(z;x) - S(w;x)]}dzdw}{z^{s}w^{t}(z-w)}
\end{align}
in each of the three ``regions'' corresponding to ranges for $x$ previously mentioned. Starting from contours $c_+$ for $z$ passing just outside the unit circle and $c_-$ for $w$ passing just inside it, we deform them to some
\begin{equation} \label{eq:primecontours}
c'_\pm: R_\pm e^{i\phi}, \qquad \phi \in [-\pi,\pi]
\end{equation} where each $R_{\pm} := R_{\pm}(\phi)$ may depend on the angle $\phi$ but is everywhere close to $1$. We will look at
\begin{align}
\Re[S(z;x)-S(w;x)]\bigg\vert_{\substack{z = R_+e^{i\phi_+} \\ w =R_-e^{i\phi_-}}} &= (R_+ - 1)[ D(\phi_+) - x] - (R_- - 1) [D(\phi_-) -x)\notag \\
&\qquad + O\left((R_+ - 1)^2+(R_- - 1)^2\right)\label{eq:realaction}
\end{align}
in order to identify the appropriate contours, which are illustrated in Figures~\ref{fig:contoursgenasym} and~\ref{fig:contourssym}. 

\begin{figure} 
\begin{center}
{\def\svgwidth{0.8\textwidth} \small 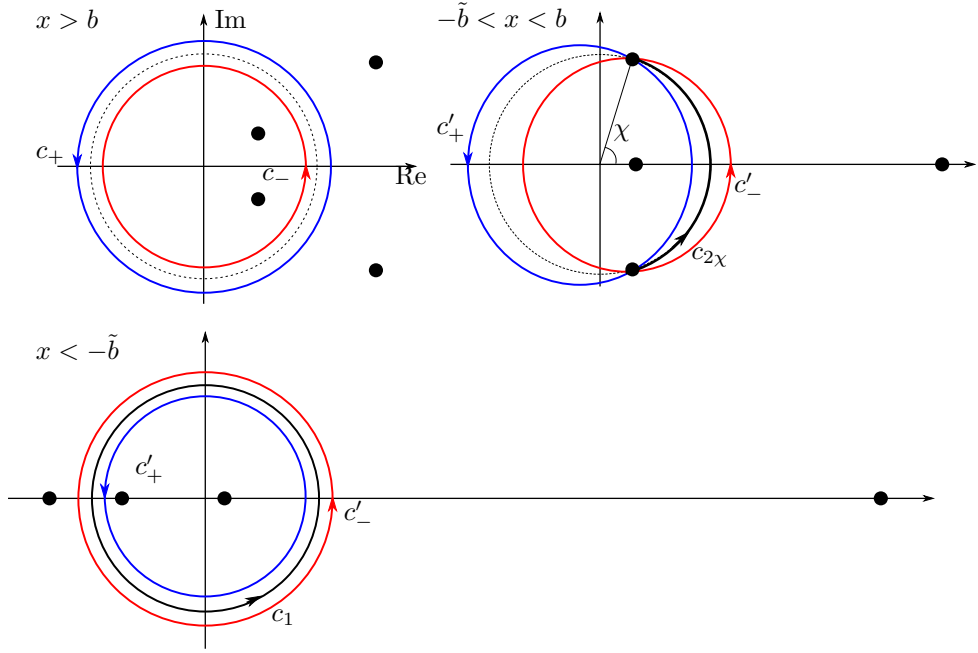} 
\end{center}
\caption{Saddle points (shown as black dots) of the action $S(z;x)$ for the $m=2$ minimal measure $\P^{\aa,2}_\theta$, at  $x = 1.6$ in the empty region (top left), at $x = 1$ in the bulk (top right) and at $x = -2.8$ in the frozen region (bottom). The contours $c_+, c_+'$ (shown in blue) pass through regions where $\Re(S(z;x))<0$, whereas the contours $c_-,c_-'$ (shown in red) pass through regions where $\Re(S(z;x))>0$.   } \label{fig:contoursgenasym}
\end{figure}

\paragraph*{(i) The empty region} For $x > b$, we have $D(\phi) < x $ for all $\phi \in [-\pi,\pi]$.
Setting $R_+>1$ and $R_-<1$ for all $\phi$ in $c'_\pm$ as defined in~\eqref{eq:primecontours} we have, for $z \in c'_+$ and $w\in c'_-$, 
\begin{align}
\Re[S(z;x)-S(w;x)] <0 
\end{align}
for $R_\pm$ sufficiently small (note that the contours do not need to pass through saddle points to find the required decay). In deforming $c_\pm$ to $c'_\pm$ the contours do not cross one another, so there is no $z=w$ pole to consider; hence for all finite $s,t$ (and indeed for all $s,t = o(\theta)$) we have exponential decay of the kernel which in turn implies dominated convergence, so that
\begin{align}
\lim_{\theta \to \infty}\mathcal{J}_\theta^m(x\theta + s,x\theta + t)  =  \frac{1}{(2\pi i)^2}  \lim _{\theta \to \infty}\oiint_{c_+',c_-'}  \frac{ e^{\theta [S(z,x) - S(w,x)]}dzdw}{z^{s+\frac{1}{2}}w^{t+\frac{1}{2}}(z-w)} = 0. 
\end{align}
Similarly, $\varrho(x) = 0$ for all $x>b$.

\paragraph*{(ii) The bulk} For $x\in (-\tilde{b},b)$, there is a unique  $\chi \in (0,\pi)$ such that $D(\chi) = x$. By the monotonicity condition~\eqref{eq:dDnonpositive}, we have $D(\phi) > x $ for $|\phi| < \chi $ and $D(\phi) < x $ for $|\phi| > \chi$.
Hence, $c'_\pm$ are proper saddle point contours on which $\Re[S(z;x)-S(w;x)]\leq 0$ is maximal and equal to  $0$ at $z = w = e^{\pm i \chi}$ 
if we set, respectively, $R_+ <1$ and $R_- >1$ for $|\phi| < \chi $,  and  
$R_+ >1$,  $R_- <1$ for $|\phi| > \chi $, all
 sufficiently close to 1. Deforming each of $c_\pm$  to $c'_\pm$  involves pulling them across one another either side of the unit circle along the arc $c_{2\chi}: z = e^{i \phi},\; \phi \in [-\chi, \chi]$. From the exchange, the integral in $z$ picks up a residue of 1 from the $z-w$ pole  for all $w = e^{i \phi}$ along $c_{2\chi}$, and we have 
 \begin{equation} \label{eq:twocontourintegrals}
\mathcal{J}_\theta^m(x\theta + s,x\theta + t) =\frac{1}{2 \pi i }\int_{c_{2\chi}} \frac{d w}{w^{s-t +1 }} +  \frac{1}{(2 \pi i )^2}   \oiint_{c_+',c_-'} \frac{ e^{\theta (S(z;x) - S(w;x))} dz dw}{z^{s+\frac12}w^{t+\frac12}(z-w)}  
\end{equation}
As this is a saddle point approximation we can easily estimate the rate of decay:
for all finite $s,t$, the integral on $c_\pm'$ is ${O}(\theta^{-1/2})$, since a  change of variables to $z = e^{\pm i \chi} + i\theta^{-1/2} \zeta$ and $w =  e^{\pm i \chi} + i\theta^{-1/2} \omega $ shows that, in terms of $f(\zeta,\omega)= (\zeta^2-\omega^2)S^{\prime\prime}(z_+;x)/2$  this integral is equal to 
\begin{align}
 \frac{\theta^{-\frac{1}{2}}}{(2 \pi)^2} \iint_{-\theta^{\frac{1}{2}}\pi}^{\theta^{\frac{1}{2}}\pi}  \frac{e^{\Re f(\zeta,\omega) } \sin (\Im f(\zeta,\omega)) }{e^{i\chi(s+1)} e^{-i\chi(t+1)} }\frac{d\zeta d\omega}{\zeta - \omega} + {O}(e^{-\theta}).
\end{align}
This is sufficient to see that only the integral on  $c_{2\chi}$  contributes to the limit, to give 
\begin{equation}
\lim_{\theta \to \infty} \mathcal{J}_\theta^m(x\theta + s,x\theta + t) = \frac{1}{2 \pi} \int_{-\chi}^\chi e^{-i\phi(s-t)}d \phi = \begin{cases}\frac{\sin\chi (s-t)}{\pi (s-t)}, \qquad &s \neq t \\
\frac{\chi}{\pi}, \qquad &s=t.
 \end{cases}
\end{equation} 
It follows that $\varrho(x) = \chi/\pi$ where $\chi$ is the non-negative solution of $D(\chi) = x$ for $x \in [\tilde{b},b]$.

\begin{figure} 
\begin{center}
{\def\svgwidth{0.95\textwidth} \small 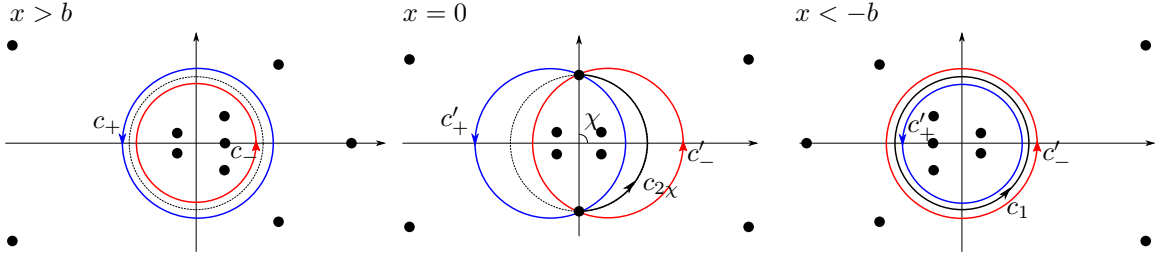} 
\end{center}
\caption{Saddle points of the action and contours for the $m=3$ symmetric minimal measure $\P^{\ss,m}_\theta$, at $x=2$ in the empty region (left), at $x=0$ in the bulk (centre) and at $x=-2$ in the frozen region (right). Here, exchanging $x$ with $-x$ reflects the saddle points about the imaginary axis. } \label{fig:contourssym}
\end{figure}

\paragraph*{(iii) The frozen region} For $x < -\tilde{b}$, we have $D(\phi)>x$ for all $\phi$. Hence, from~\eqref{eq:realaction}, by setting $R_+<1$ and $R_->1$ sufficiently close to $1$ for all $\phi$, we have $\Re[S(z;x)-S(w;x)]<0$ for $z$ on $c_+'$, $w$ on $c_-'$. Now deforming $c_\pm$ to $c_\pm'$ involves passing them across one another along the whole unit circle $c_{1}: |z| = 1$. We have
\begin{equation} \label{eq:twocontourintegrals2}
\mathcal{J}_\theta^m(x\theta + s,x\theta + t) =\frac{1}{2 \pi i }\oint_{c_1} \frac{d w}{w^{s-t +1 }} +  \frac{1}{(2 \pi i )^2}   \oiint_{c_+',c_-'} \frac{ e^{\theta (S(z;x) - S(w;x))}  dz dw }{z^{s+\frac12}w^{t+\frac12}(z-w)}  .
\end{equation} The integral on $c_\pm'$ decays to zero exponentially fast as $\theta \to \infty$, and the residue on $c_1$ gives
\begin{equation}
 \lim_{\theta \to \infty } \mathcal{J}_\theta^m(x\theta + s,x\theta + t) = \delta_{st}.
 \end{equation} 
 It follows that $\varrho(x) = 1$ for $x < - \tilde{b}$. 

 Putting the three regions together, the proof is complete. 
\end{proof}

With these ingredients we can finally prove the limit shape theorem.

\begin{proof}[{ Proof of Theorem~\ref{thm:mclimitshape}}]
By Lemma~\ref{lem:convergencels}, it is sufficient to find the limit of the expectation for $\E(\psi_{\lh,\theta})$ to have convergence in probability; by~\eqref{eq:expprofile} this is 
\begin{equation}
\Omega(x) := \lim_{\theta \to \infty} \E(\psi_{\lh,\theta}(x)) = x + 2\int_x^\infty \varrho(x')dx' 
\end{equation}
in terms of the limiting density $\varrho$ given in the previous Lemma~\ref{lem:density}. 
Since $\Omega(b) = b$ and $\Omega(-\tilde{b}) = \tilde{b}$, we can write this as the finite integral
\begin{equation} \label{eq:curvefromdensity}
\Omega(x) = \begin{cases}x + 2\tilde{b} + \frac{2}{\pi} \int_{-\tilde{b}}^x \chi(v) dv , \quad & x \in [-\tilde{b},b] \\ 
|x|\, , \quad &x>b \text{ and } x<-\tilde{b}
\end{cases}
\end{equation}
as required.

Now consider the vanishing of $\varrho(x)$ as $x \to b$. Noting that $\chi(b) = 0$, we develop $\chi(b-\eps)$ around zero when $\eps>0$ is small. Expanding $D(\chi)$ for $\chi$ small and applying the multicriticality condition~\eqref{eq:multicriticalitycondition} we find
\begin{equation}
b - d \chi^{2m} + O(\chi^{2m+2}) = b - \eps \, .
\end{equation}
So $\chi (b-\eps) \sim (\eps/d)^{1/2m} $ as $\eps \to 0$, and as $x \to b$ we have 
\begin{equation}
\chi(x) \sim \bigg( \frac{b-x}{d}\bigg)^{\frac{1}{2m}}.
\end{equation}
Then, from~\eqref{eq:curvefromdensity}, we recover the edge vanishing behaviour~\eqref{eq:limitshapevanishing} as required. 
\end{proof}

\subsection{Asymptotic edge fluctuations} \label{sec:mcedge_fluctuations}

We turn our attention to proving Theorem~\ref{thm:multicritical}. Again we start with the limiting kernel, now in the critical scaling regime near the right edge:

\begin{lem}
\label{lem:edgekernel}
As $\theta \to \infty$, we have
\begin{align}
(d\theta)^{\frac{1}{2m+1}} \mathcal{J}^m_\theta &(\lfloor b \theta + x (d\theta)^{\frac{1}{2m+1}}\rfloor - \tfrac12,\lfloor b \theta + y (d\theta)^{\frac{1}{2m+1}}\rfloor - \tfrac12) \label{eq:kernelconvergence}\\
&\to \mathcal{A}_{2m+1}(x,y) =\frac{1}{(2\pi i)^2} \int_{i\R -1} \int_{i\R +1} \frac{\exp\big[(-1)^{m+1}\frac{\zeta^{2m+1}}{2m+1}- x\zeta\big]}{\exp[(-1)^{m+1}\frac{\omega^{2n+1}}{2n+1} - y \omega]} \frac{d\zeta d\omega}{\zeta - \omega}  \notag
\end{align}
uniformly for $x,y$ in compact subsets of $\R$. 
\end{lem}

\begin{figure}
{\def\svgwidth{0.32\textwidth}\small 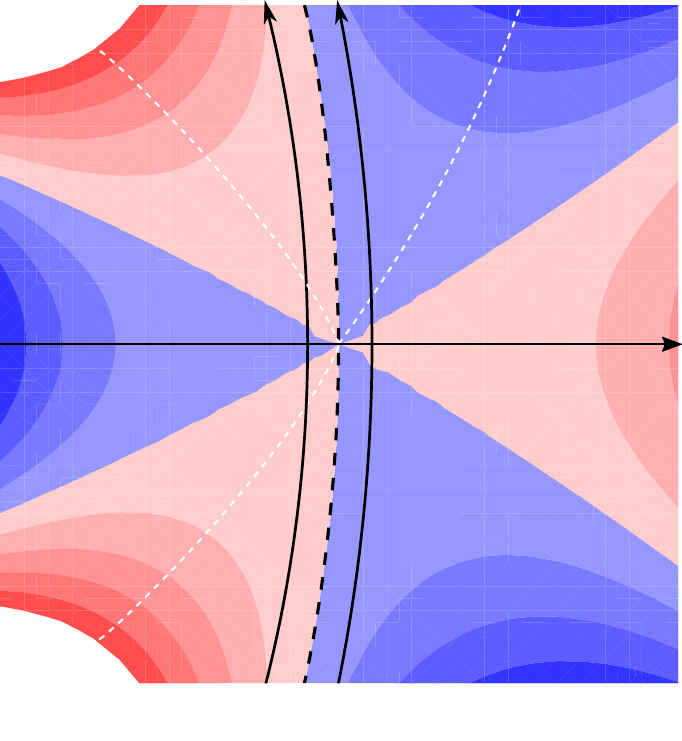}\hfill{\def\svgwidth{0.32\textwidth} \small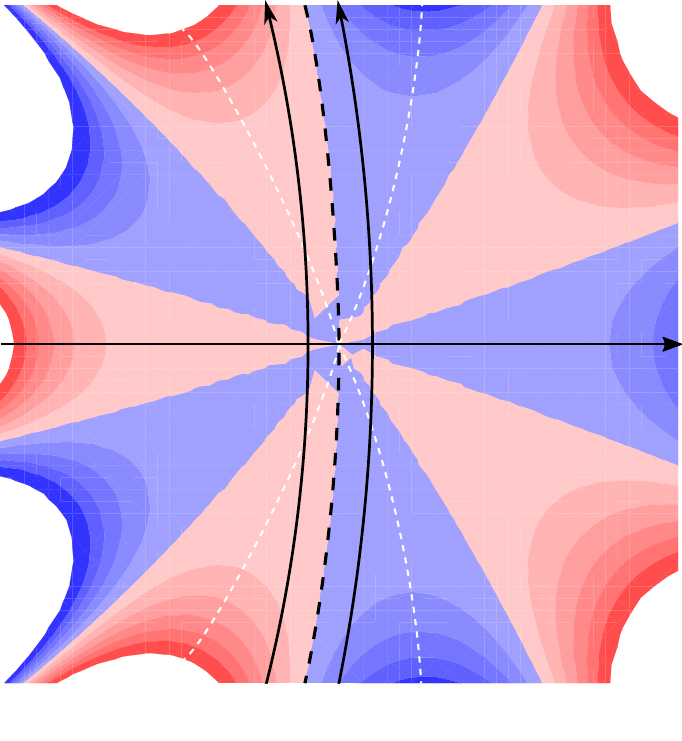}\hfill{\def\svgwidth{0.32\textwidth} \small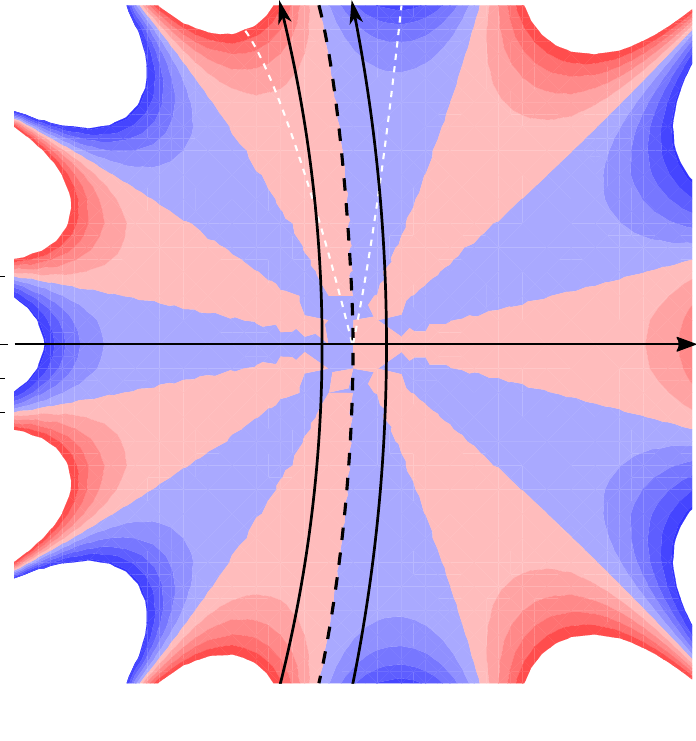}
\caption{The vicinity of the order $2m$ saddle point of the action $S(z;b)$ for order $m=1,2,3$ multicritical measures, and our choice of integration contours $c_+,c_-$ (shown in black) for this regime. The colours indicate the values of $\Re(S(z;b))$, with blue where $\Re(S(z;b))<0$ and red where $\Re(S(z;b)) >0$. The steepest ascent and descent curves near the unit circle are indicated with dashed white lines; they leave the point $z=1$ at angles of $\frac{m \pi}{2m+1}$ from the real axis. We plot this for $\P_\theta^{\aa,m}$, but this picture is universal close enough to $z=1$. } \label{fig:multicritscs}
\end{figure}

Here we justify the critical scaling regime directly by the proof, but we refer to Appendix~\ref{sec:ccgairy} for an informal derivation of a similar scaling regime.
\begin{proof}
In terms of the action, the kernel in the edge regime is
\begin{align}
\mathcal{J}_\theta^m &\big(\lfloor b \theta + x (d\theta)^{\frac{1}{2m+1}}\rfloor - \tfrac12,\lfloor b \theta + y (d\theta)^{\frac{1}{2m+1}}\rfloor - \tfrac12 \big) \notag  \\
&=\frac{1}{(2\pi i)^2} \oiint_{c_+,c_-} \frac{\exp[\theta  S(z;b) - x (d\theta)^{\frac{1}{2m+1}} \log z]}{\exp[\theta S(w;b) - y (d\theta)^{\frac{1}{2m+1}} \log w ]}[1+o(1)] \frac{dz dw}{z-w}  \label{eq:prelimitkernel}
\end{align}
where the little-$o$ accounts for the difference between the continuous coordinates and their integer parts, and is uniform in $x$ and $y$. 
The action $S(z;b)$ has an order $2m$ saddle point at $z=1$. Let us take the integral over contours which only approach this point as $\theta$ tends to infinity, namely the circles\footnote{This choice of contour differs from the one presented in~\cite{Okounkov_2003} even at $m=1$; rather, we adapt the contours used in~\cite{Betea_Bouttier_2019}. 
One can alternatively adapt the contours in~\cite{Okounkov_2003} to ones passing through $1$ at angles of $m\pi/(2m+1)$ from the real axis; asymptotically this recovers the integration contours used in~\cite{Cafasso_Claeys_Girotti_2019} to define the higher-order Airy function (see Figure~\ref{fig:multicritscs}; of course, this does not change the value of the integral).}
\begin{equation}
c_+: |z| = \exp\big[(d\theta)^{-\frac{1}{2m+1}}\big], \qquad c_-: |w| =  \exp\big[- (d\theta)^{-\frac{1}{2m+1}}\big].
\end{equation}
 Note that these contours do not cross, so we do not encounter the $z=w$ pole. They are illustrated in Figure~\ref{fig:multicritscs}.

As before, let us parametrise $c_+$  by $\phi \in [-\pi,\pi]$. Then we have, uniformly in $\phi$,
\begin{equation} \label{eq:realpartapprox}
\Re\big(S\big(e^{(d\theta)^{-\frac{1}{2m+1}}}e^{i \phi};b\big)\big) =  (D(\phi)-b) (d\theta)^{-\frac{1}{2m+1}} + {O}\big(\theta^{-\frac{3}{2m+1}}\big) .
\end{equation}
The dominant term is maximal at $\phi = 0$, and a similar parametrisation of $c_-$ shows that when $\theta$ is large, the dominant term of $\Re(S)$ is minimal at $\phi=0$. We now show that the integrals on $c_+$ and $c_-$ are each dominated by a suitably chosen region around their intersections with the positive real axis. 

\paragraph*{Central region and tails} We fix a number $\eps \in (0,\tfrac{2}{(2m+1)(2m+3)})$ (the choice of the upper bound will be justified in what follows), and define the \emph{central region} as
\begin{equation}
I := \{(z,w) \in c_+ \times c_- : |\arg z|, |\arg w| <(d \theta)^{-\frac{1}{2m+1} + \eps}\}.
\end{equation} 
The complementary region $I^c= c_+ \times c_- \setminus I$ is called the \emph{tail region}; let us first bound its contribution to the integral in~\eqref{eq:prelimitkernel}. 
By~\eqref{eq:dDnonpositive} and by the multicriticality condition~\eqref{eq:multicriticalitycondition}, there is a $C>1$ such that
\begin{equation}
 D(\phi) - b \leq  - \frac{\phi^{2m}}{C}\qquad  \text{for all }\phi \in [-\pi,\pi].
 \end{equation}
By \eqref{eq:realpartapprox}, we see that for any points in $I^c$, we have a uniform bound on the integrand of the kernel, with
\begin{equation}
 e^{\theta(S(z;b) - S(w;b))} = O\left(e^{-{\theta}^{2m \eps}/{C}}\right).
 \end{equation} 
 As the domain of integration is bounded, and as $1/(z-w) = O(\theta^{\frac{1}{2m+1}})$, we conclude that the contribution from the tail region  $I^c$ to the integral is exponentially small in $\theta$. 

 We now estimate the contribution from the central region $I$. For this, we make the change of variables
\begin{align}
z = \exp\big[\zeta (d\theta)^{-\frac{1}{2m+1}}\big], \; w =  \exp\big[\omega (d\theta)^{-\frac{1}{2m+1}}\big], \quad  & \zeta \in  i[-(d\theta)^\eps,(d\theta)^\eps] + 1, \\ &\omega \in  i[-(d\theta)^\eps,(d\theta)^\eps] -1 . \notag
\end{align}
Then, recalling the multicriticality
conditions~\eqref{eq:multicritaction}
and~\eqref{eq:actionnonvanishing}, a Taylor expansion of the action
yields
\begin{align}
S(e^{\zeta (d\theta)^{-\frac{1}{2m+1}}};b) = S(1,b) + \frac{(-1)^{m+1}}{\theta}\frac{\zeta^{2m+1}}{2m+1} + O\left(\theta^{(2m+3)\eps-\frac{2m+3}{2m+1}}\right)
\end{align}
uniformly in  $\zeta$.
 The integrand of $\mathcal{J}_\theta^m$ has an exponentially decaying upper bound, as uniformly on $c_+$ and $ c_-$ we have
\begin{equation} \label{eq:dominatedconv}
 \left\vert\frac{\exp[\theta  S(z;b) - x (d\theta)^{\frac{1}{2m+1}} \log z ]}{\exp[\theta S(w;b) - y (d\theta)^{\frac{1}{2m+1}} \log w ]} \right\vert\leq C_1 \exp\big[-C_2(d\theta)^{\frac{1}{2m+1}} (x+y)\big]
 \end{equation} 
 for constants $C_1,C_2$, so by dominated convergence the limit of its integral converges to the integral of its limit.
 Since $x \log z = x\zeta (d\theta)^{1/(2m+1)}$, $y \log w =y \omega (d\theta)^{1/(2m+1)}$  and $z-w = (d\theta)^{-1/(2m+1)}(\zeta-\omega)+ O(\theta^{-2/(2m+1)})$  and  by the bound on the tails contribution, we have a uniform approximation of the kernel 
\begin{align}
\mathcal{J}_\theta^m &\big(\lfloor b \theta + x (d\theta)^{\frac{1}{2m+1}}\rfloor - \tfrac12,\lfloor b \theta + y (d\theta)^{\frac{1}{2m+1}}\rfloor - \tfrac12 \big) \\
&=(d\theta)^{-\frac{1}{2m+1}} \frac{1}{(2\pi i)^2} \int_{iI_{\theta} -1} \int_{iI_{\theta} +1} \frac{\exp\big[(-1)^{m+1}\frac{\zeta^{2m+1}}{2m+1}- x\zeta\big]}{\exp[(-1)^{m+1}\frac{\omega^{2n+1}}{2n+1} - y \omega]} [1+o(1)]\frac{d\zeta d\omega}{\zeta - \omega}  \notag 
\end{align}
where $I_\theta$ is the interval $ [-(d\theta)^\eps,(d\theta)^\eps]$. The $o(1)$ term accounts for the error of $\theta O(\theta^{(2m+3)\eps-\frac{2m+3}{2m+1}})$ from the Taylor approximation of $S$, which is indeed $o(1)$ as we chose $\eps < \frac{2}{(2m+1)(2m+3)}$, as well as $o(1)$ errors from the discretisation and the first order approximation of $z-w$ by $\zeta-\omega$. As $\theta \to \infty$, we have $I_\theta \to \R$, and the convergence~\eqref{eq:kernelconvergence} follows immediately, as required. 
\end{proof}
With this, we can finally prove our main result.
\begin{proof}[{ Proof of Theorem~\ref{thm:multicritical}}]
If $\lh$ is a random partition under $\P^m_\theta$ the probability of that there is no element of $S(\lh)$ greater than some half-integer $\ell_s := \lfloor b\theta + s(d\theta)^{1/(2m+1)}\rfloor -\tfrac12$ is a discrete Fredholm determinant of the form~\eqref{eq:discretegapprob}, namely
\begin{align}
\mathcal{F}_\theta^m(\ell_s) &= \det(1- \mathcal{J}_\theta^m)_{l^2(\ell_s + \Z_{\geq 0})} = \P_\theta^m(\lh_1 < \lfloor b\theta + s(d\theta)^{1/(2m+1)}\rfloor)\notag \\
&= \sum_{n=0}^\infty \frac{(-1)^n}{n!} \sum_{k_1 =\ell_s}^\infty\cdots\sum_{k_n =\ell_s}^\infty \det_{1\leq i,j\leq n}\mathcal{J}_\theta^m(k_i,k_j) 
\end{align}
We need to show that it converges to the continuous Fredholm determinant $\det(1-\mathcal{A}_{2m+1})_{L^2([s,\infty)}$ of the kernel given in~\eqref{eq:kernelconvergence} as  $\theta \to \infty$. 

First, using a change of variables $k_{x_i} := \lfloor b\theta + x_i(d\theta)^{1/(2m+1)}\rfloor -\tfrac12$, we can write the discrete determinant on $L^2([s,\infty))$ as
\begin{equation}
\mathcal{F}_\theta^m(\ell_s) = \sum_{n=0}^\infty \frac{(-1)^n}{n!} \int_{s}^\infty\cdots\int_{s}^\infty\det_{1\leq i,j\leq n}\left[(d\theta)^{\frac{1}{2m+1}}\mathcal{J}_\theta^m(k_{x_i} ,k_{x_j} ) \right] dx_1 \cdots dx_n.
\end{equation}
By Lemma~\ref{lem:edgekernel}, for each $n$ in the sum we have the pointwise convergence of the integrand to $\det_{1\leq i,j \leq n}\mathcal{A}_{2m+1}(x_i,y_j)$. The convergence of the Fredholm determinant follows from an application of Hadamard's bound of the determinant by a product of column sums; then, we only need to show that the traces of $(d\theta)^{\frac{1}{2m+1}}\mathcal{J}_\theta^m$ converge to the traces of $\mathcal{A}_{2m+1}$. But we can apply the same exponential decay bound~\eqref{eq:dominatedconv} once again to bound $(d\theta)^{\frac{1}{2m+1}}\mathcal{J}_\theta^m$ itself on any interval that is bounded below, and by dominated convergence on such an interval we have the convergence of the discrete Fredholm determinant $\mathcal{F}_\theta^m(\ell_s)$ to the continuous one $F_{2m+1}(s) = \det(1-\mathcal{A}_{2m+1})_{L^2([s,\infty))}$.

It remains to show that the expression~\eqref{eq:kernelconvergence} for $\mathcal{A}_{2m+1}$ in Lemma~\ref{lem:edgekernel} is indeed equivalent to our original definition~\eqref{eq:genairykernel}. First, we insert $1/(\zeta-\omega) = \int_0^\infty e^{v(\zeta-\omega)}dv$ into~\eqref{eq:kernelconvergence} to write 
\begin{equation}
\mathcal{A}_{2m+1}(x,y) = \int_0^\infty \Ai_{2m+1}(x+v) \Ai_{2m+1}(y+v) dv.
\end{equation}
Following~\cite[Appendix D]{LDMS_2018}, we apply the eigenfunction relation~\eqref{eq:genairyevrel} to obtain 
\begin{align}
(&x-y) \mathcal{A}_{2m+1}(x,y) = \int_0^\infty [(x+v)-(y+v)]\Ai_{2m+1}(x+v) \Ai_{2m+1}(y+v) dv \\
&= (-1)^{m+1}\int_0^\infty (\Ai_{2m+1}^{(2m)}(x+v) \Ai_{2m+1}(y+v) - \Ai_{2m+1}(x+v) \Ai_{2m+1}^{(2m)}(y+v)) dv .\notag
\end{align}
Then we note that the integrand can be written
\begin{align}
\Ai_{2m+1}^{(2m)}(x+v) \Ai_{2m+1}&(y+v) - \Ai_{2m+1}(x+v) \Ai_{2m+1}^{(2m)}(y+v)\notag \\
&= \frac{\partial}{\partial v} \sum_{i=0}^{2m-1}(-1)^{i}\Ai_{2m+1}^{(i)}(x+v)  \Ai_{2m+1}^{(2m-1-i)}(y+v).
\end{align}
Inserting this back into the integral, only the $v=0$ boundary term contributes, recovering the ``Christoffel--Darboux type'' expression~\eqref{eq:genairykernel} and completing the proof.
\end{proof}

\section{Multicritical unitary matrix models} \label{sec:matrixintegrals}
In this final section we consider the $\ell \times \ell$ unitary matrix models with densities $p^m_{\theta,\ell}$. First we give a proof of Theorem~\ref{thm:genunitary}, which states that the partition functions of these models are given by the renormalised distributions $e^{\theta^2 \sum_r r \gamma_r^2} \P_\theta^m(\lambda_1 \leq \ell)$.  Then we informally study the behaviour of the limiting eigenvalue density as $\ell$ tends to infinity and as the parameter $\theta$ grows linearly in $\ell$, to find the edge behaviour given by~\eqref{eq:mcunitaryedge}.

\subsection{Exact mapping between Schur measures and unitary matrix integrals}
Now let us turn to proving Theorem~\ref{thm:genunitary}. By Proposition~\ref{prop:mcconjugation}, we can equivalently write this theorem as follows:

\refstepcounter{thm}  \label{corr:conjunitary}
\edef\tempthmno{\thethm}
\begingroup
\renewcommand\thethm{\tempthmno~of Theorem~\ref{thm:genunitary}}
\begin{corr}[Equivalent conjugate partition formulation of Theorem~\ref{thm:genunitary}]
Let $\lambda$ be a random partition under a Schur measure $\P(\lambda) = e^{-\sum_r r t_r t'_r} s_\lambda[t] s_\lambda[t']$ for some sequences of Miwa times $t,t'$. Then, for any positive integer $\ell$, we have
\begin{equation} \label{eq:lengthdistribution}
e^{\sum_r r t_r t'_r}\P(\ell(\lh)\leq \ell) =  \det_{1 \leq i, j \leq \ell}  g_{j-i} =  \int_{\mathcal{U}(\ell)}  e^{ \tr \, \sum_r (t_r U^r+t'_r U^{- r}) } \mathcal{D} U
\end{equation}
where the $g_n$ appearing in the determinant are given by
\begin{equation} \label{eq:symbol2}
\sum_{n \in \Z} g_n z^n = \exp \bigg[ \sum_{r\geq 1}( t_r z^r  + t'_r z^{-r})\bigg].
\end{equation}
and where $\mathcal{D}U$ denotes the Haar measure on the unitary group $\mathcal{U}(\ell)$.
\end{corr}
\endgroup
\addtocounter{thm}{-1}
We prove the result in this formulation, as it is somewhat simpler.

\begin{proof}[{ Proof of Corollary~\ref{corr:conjunitary} and Theorem~\ref{thm:genunitary} }]

The first equality of~\eqref{eq:lengthdistribution} may be written
  \begin{equation}
    \sum_{\lambda:\,\ell(\lambda) \leq \ell} s_\lambda[t] s_\lambda[t'] = \det_{1 \leq i,j \leq \ell} g_{j-i},
  \end{equation}
  which, with $g_n$ as defined in~\eqref{eq:symbol2}, was proven by Gessel in~\cite{Gessel_1990}.
  It follows from the Jacobi--Trudi formula~\eqref{eq:jt1} which for our purposes serves as a definition of the Schur measure, and which gives 
  \begin{equation} \label{eq:cbformofedge}
 \sum_{\lambda:\,\ell(\lambda) \leq \ell} s_\lambda[t] s_\lambda[t'] = \sum_{\lambda} \det_{1 \leq i,j \leq \ell} h_{\lambda_i - i + j}[t] \det_{1 \leq i,j \leq \ell} h_{\lambda_i - i + j}[t'] 
  \end{equation}
  where $h_i[t]$ again denotes the complete homogeneous symmetric functions with generating function $\sum_{i }h_i[t]z^i = e^{\sum_r t_r z^r}$.  
The expression~\eqref{eq:cbformofedge} is a sum of products of $\ell \times \ell $ minors of the non-square Toeplitz matrices 
\begin{align}
H = (H_{a,b})_{\substack{1 \leq a \leq \ell \\ 1 \leq b <\infty}}, \quad H_{a,b} = h_{b-a}[t] \quad\text{and} \quad  H' = (H'_{a,b})_{\substack{1 \leq a \leq \infty \\ 1 \leq b <\ell}}, \quad H'_{a,b} = h_{a-b}[t'].
\end{align}
The Cauchy--Binet identity (see e.g.~\cite[Chapter IV]{aitken1956} or~\cite{Forrester_2019}) gives, for any matrices $A,B$ such that $AB$ has dimension $\ell \times \ell$,
\begin{equation}
\sum_{\substack{L \subset \{1,2,\ldots\}\\ |L| = \ell}} \det A\vert_L \det B\vert_L=\det AB
\end{equation}
where $A\vert_L$ denotes the $\ell \times \ell$ submatrix of $A$ including the rows indexed by  $L$; noting that $H\cdot H'$ is $\ell \times \ell$, this gives
\begin{equation}
\sum_{\lambda:\,\ell(\lambda) \leq \ell} s_\lambda[t] s_\lambda[t'] = \sum_{\substack{L \subset (1,2,3,\ldots)\\ |L| = \ell}} \det H\vert_L \det H'\vert_L=\det H\cdot H'
\end{equation}
and  the entries of the final matrix product are
      \begin{equation}
      (H\cdot H')_{a,b} = \sum_i h_{i-a} [t] h_{i-b} [t'] = \sum_i h_{i-a+b} [t] h_{i} [t'] , \quad 1 \leq a,b \leq \ell
  \end{equation}
  (the sum over $i$ can run over all integers thanks to the convention $h_i = 0$ for $i<0$). 
  Thus $H\cdot H'$ is a Toeplitz matrix, and its symbol is (below we use $z^a = z^{k+a} z^{-k}$)
  \begin{align}
      \sum_n z^n  \sum_i h_{i+n} [t]  h_i [t']=  \sum_n  \sum_i z^{i+n}  h_{i+n} [t]  z^{-i} h_i [t'] =\exp \bigg[ \sum_{r\geq 1}( t_r z^r  + t'_r z^{-r})\bigg].
  \end{align}
  This is precisely the  Toeplitz determinant symbol generating the entries $g_n$ in the statement, proving the first equality. 

  To prove the second equality, or Heine's identity, we use the Cauchy--Binet identity in its continuous form; this is called the Andre\"ief identity, see e.g.~\cite{Forrester_2019}. For some space $R$ equipped with a measure $\mu_{\mathrm{r}}$ and integrable functions $\Phi_i,\Psi_i $ on $R$ for $1\leq i \leq \ell$, this identity gives
  \begin{equation} \label{eq:andreief}
      \int_R \hspace*{-0.3em} \cdots\hspace*{-0.3em} \int_R [\det \Phi_i (z_j)] \cdot [\det \Psi_i (z_j)] d \mu_{\mathrm{r}}(z_1) \cdots  d \mu_{\mathrm{r}}(z_\ell)  = \ell! \det [\int_R  \Phi_i (z) \Psi_j (z) d \mu_{\mathrm{r}}(z)] 
  \end{equation}
  where each determinant is over indices $1\leq i,j \leq \ell$.
We apply this to the unitary matrix integral on the right hand side, first writing it as an $\ell$-fold contour integral on the unit circle 
  \begin{equation}
  \frac{1}{(2\pi i)^\ell \ell !} \oint_{c_1} \hspace*{-0.3em}\cdots  \oint_{c_1} \prod_{i=1}^\ell e^{ \sum_{r\geq 1}( t_r u_i^r  + t'_r u_i^{-r})} \prod_{i<j}|u_i - u_j|^2 \frac{du_1}{ u_1} \cdots\frac{du_\ell}{u_\ell}
  \end{equation}
  and then as an $\ell$-fold integral over determinants: for the squared Vandermonde determinant, we have
  \begin{equation}
  \prod_{i<j} (u_i - u_j) ({u}^*_i - {u}^*_j) = \prod_{i<j} (u_i - u_j) (u_i^{-1} - u_j^{-1}) = \det_{1\leq i,j \leq \ell} u_j^{i-1} \cdot \det_{1\leq i,j \leq \ell} u_j^{1-i}
   \end{equation} 
   since $|u_i|=1$, and we split $\prod_i e^{ \sum_{r\geq 1}( t_r u_i^r  + t'_r u_i^{-r})}$ across each determinant to see that this integral is equal to the left hand side of the Andre\"ief identity expression~\eqref{eq:andreief} upon inserting
  \begin{equation} \label{eq:andreiefinsert}
    R= c_1, \; d \mu_{\mathrm{r}}(u) = \frac{du}{2 \pi i u}, \; \Phi_i(u) = u^{1-i} e^{\sum_{r\geq 1} t_r u^r}, \; \Psi_i(u) = u^{i-1} e^{\sum_{r\geq 1} t_r u^{-r}}
  \end{equation}
  (where $c_1$ denotes the unit circle).
 Then, the right hand side of the identity is
  \begin{equation}
  \det_{1\leq i,j\leq \ell} \oint_{c_1 } e^{\sum_{r\geq 1} (t_r u^r + t'_ru^{-r}) } \frac{du}{u^{i-j+1}}
  \end{equation}
  where the integral extracts precisely the Toeplitz matrix element $g_{i-j}$. This gives the second equality and completes the proof. 

  To prove the original equalities of Theorem~\ref{thm:genunitary} directly, we can proceed analogously from the dual Jacobi--Trudi formula~\eqref{eq:jt2}. We have
  \begin{equation}
  e^{\sum_r r t_r t'_r}\P(\lh_1 \leq \ell) = \sum_{\substack{\lambda}\,: \lambda_1 \leq \ell} \det_{1\leq i,j \leq \ell} e_{\lambda'_i-i+j}[t] \det_{1\leq i,j \leq \ell} e_{\lambda'_i-i+j}[t']
  \end{equation}
  where $e_i[t]$ are the elementary symmetric functions generated by $\sum_i e_i[t] z^i = e^{\sum_r (-1)^{r+1} t_r z^r}$.
Repeating the arguments above, we find 
  \begin{equation}
   e^{\sum_r r t_r t'_r} \P(\lh_1 \leq \ell) = \det_{1\leq a,b\leq \ell} \sum_i e_{i-a+b} [t] e_{i} [t'] 
  \end{equation}
and the symbol of the Toeplitz determinant is
\begin{equation}
\sum_{n} z^{n} \sum_i e_{i+n} [t] e_{i} [t']  = e^{\sum_r (-1)^{r+1} (t_r z^r + t'_r z^{-r})}
\end{equation}
which is the symbol in the statement; so, we have a determinant of $f_{i-j}$, proving the first equality.  This is the dual version of Gessel's theorem. 

Now once again we can start from the rightmost unitary matrix integral, and write it in the form of the left hand side of the Andre\"ief identity~\eqref{eq:andreief} with the same insertions~\eqref{eq:andreiefinsert}, except for
\begin{equation}
\Phi_i(u) = u^{1-i} e^{\sum_{r\geq 1} (-1)^{r+1}t_r u^r}, \; \Psi_i(u) = u^{i-1} e^{\sum_{r\geq 1}(-1)^{r+1} t_r u^{-r}};
\end{equation}
then, the right hand side of the identity gives us the determinant of a contour integral which extracts the matrix element $f_{i-j}$ from the symbol as required. Of course, the same equalities can be derived from Corollary~\ref{corr:conjunitary} by Proposition~\ref{prop:mcconjugation} directly.\end{proof}

\subsection{Asymptotic behaviour of multicritical unitary matrix models}

\begin{figure}
{\def\svgwidth{0.9\textwidth}\small 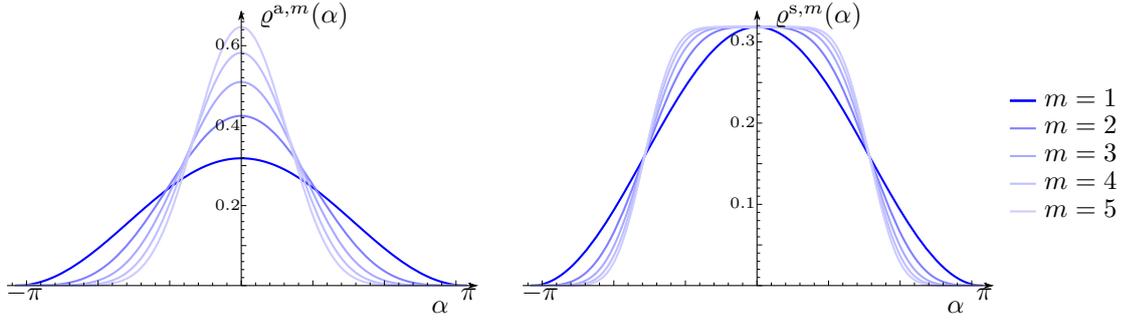}
\caption{Eigenvalue densities  for random $\ell\times \ell$ unitary matrices under the minimal multicritical probability densities $p^{\aa,m}_{\theta,\ell}$ (left) and $p^{\ss,m}_{\theta,\ell}$ (right) in the limit $\ell \to \infty $ in the critical regime $\theta = \ell/b$. 
The eigenvalues lie on the unit circle, and $\varrho^{\aa/\ss,m}(\alpha)$ denotes the limiting density at $e^{i\alpha}$, as defined at~\eqref{eq:evdensitydef}. See~\eqref{eq:evdensityformula} for a general expression. 
In the right hand figure, note the symmetry $\varrho^{\ss,m}(\pi - \alpha) = \frac{1}{\pi} - \varrho^{\ss,m}(\alpha)$.}
\label{fig:evdensities}
\end{figure}

In this section we consider the unitary matrix models exactly related to multicritical Schur measures $\P^m_\theta$ by Theorem~\ref{thm:genunitary}, and consider their probability density of eigenvalues in an asymptotic regime corresponding to the one in Theorem~\ref{thm:mclimitshape}. For a potential $V(z) = \sum_r \gamma_r z^r$ with coefficients satisfying~\eqref{eq:multicriticalitycondition} we look at an $\ell \times \ell$ random unitary matrix $U$ under the probability density 
\begin{equation}
p_{\theta,\ell}^m(U) = \frac{1}{Z_\ell}e^{-\theta \tr [V(-U)+ V(-U^{-1})]}
\end{equation}
with respect to the Haar measure, and set the coupling to $\theta := \ell / x$ for positive $x$. From~\eqref{eq:weyl}, the multicritical density induces the density 
\begin{equation}
 p^m_{\theta,\ell}( \alpha_1,\ldots,\alpha_\ell) = \frac{4^{\ell(\ell+1)/2}}{Z_\ell (2\pi)^\ell \ell!}e^{-\frac{\ell}{x} \sum_{j=1}^{\ell} [V(-e^{i\alpha_j})+V(-e^{-i\alpha_j}) ]} \prod_{j<k} \big|\sin \frac{\alpha_j-\alpha_k}{2}\big|^2
\end{equation}
on the ordered arguments $-\pi \leq \alpha_1 \leq \alpha_2 \leq \ldots \leq \alpha_\ell \leq \pi$ of the eigenvalues $e^{i\alpha_j}$ of $U$ with respect to $ d\alpha_1 \cdots d\alpha_\ell$ (note that, thanks to the inversion invariance, the sum of potentials can be written as a polynomial of cosines). 
Defining the non-decreasing function
\begin{equation}
\alpha(u):= \alpha_{\lfloor u \ell\rfloor}
\end{equation} 
in terms of the arguments of the eigenvalues, we  compute the limiting eigenvalue density as $\ell \to \infty$, following an approach from~\cite{BIPZ_1978} and generalising a calculation in~\cite{Gross_Witten_1980} by optimising the functional appearing exponentiated in the joint eigenvalue density, that is
\begin{equation}
-\frac{1}{x}\int_0^1 [ V(-e^{i\alpha(u)})+V(-e^{-i\alpha(u)})] du + \fint_{0}^1 \fint_{0}^1 \log \big|\sin \frac{\alpha(u)-\alpha(v)}{2}\big| du \, dv
\end{equation}
where  $\fint$ denotes the Cauchy principal part.
Now if the non-decreasing function $\alpha(u)$ encodes the limiting eigenvalue distribution, it is related to the limiting density by $\varrho(\alpha) = du/d\alpha$. The saddle point equation for the functional above is
\begin{equation} \label{eq:equillibriumformula}
 \frac{i}{x} \big[e^{i\alpha}V'(-e^{i\alpha})-e^{-i\alpha}V'(-e^{-i\alpha})\big]= \fint_{0}^1 \cot \frac{\alpha-\alpha(v)}{2} dv = \fint_{-\beta_c}^{\beta_c} \varrho(\beta) \cot \frac{\alpha-\beta}{2} d\beta 
\end{equation}
where the support $[-\beta_c,\beta_c]$ of $\varrho$ is also to be determined. Inserting $e^{i (\alpha+\pi)}$ for $-e^{i\alpha}$ we have 
\begin{align}
\fint_{-\beta_c}^{\beta_c} \varrho(\beta) \cot \frac{\alpha-\beta}{2} d\beta = -\frac{1}{x} \sum_{r\geq 1} 2 r \gamma_r \sin r (\alpha - \pi) . \label{eq:densitycond}
\end{align}
We  will approach the critical point from one side only\footnote{The approach from the subcritical $x<b$ side is much more subtle but is still feasible; we refer to the final equations of~\cite{Periwal_Shevitz_1990} for an explicit formula for the density and its support below criticality in any degree $4$ potential.}, and let  $x$ be sufficiently large that $\varrho$ is supported on $[-\pi,\pi]$. Then, following the steps of~\cite[Page 449]{Gross_Witten_1980}, we note that 
\begin{equation}
\fint_{-\pi}^{\pi} \cos(\beta-\pi) \cot \frac{\alpha-\beta}{2} d\beta = 2\pi \sin(\alpha-\pi)
\end{equation}
 and hence find the density
 \begin{equation} \label{eq:evdensityformula}
 \varrho(\alpha) = \frac{1}{2 \pi} \bigg[ 1 - \frac{1}{x} \sum_{r\geq 1} 2 r \gamma_r \cos r (\alpha - \pi) \bigg]
 \end{equation}
which satisfies~\eqref{eq:densitycond} and is normalised; in the notation of Section~\ref{sec:mc_proofs} we have $\varrho(\alpha) = \frac{1}{2 \pi} [1-\frac{1}{x}D(\alpha-\pi)]$. We plot some examples in Figure~\ref{fig:evdensities}. 
 Note that by the condition~\eqref{eq:singlefermiseacondition} on the coefficients $\gamma_r$, or more directly by the condition~\eqref{eq:dDnonpositive} on $D(\phi)$,  $\varrho(\alpha)$ has a unique minimum at $\alpha = \pi$ and we have $\varrho(x) >0$ for all $x> b$. At $x \to b$, we have the appearance of a single cut as $\varrho(\pi)\to 0$; developing in $\alpha - \pi$ close to zero, we employ the multicriticality conditions~\eqref{eq:multicriticalitycondition} and the definitions of $b,d$ once again to find that
\begin{equation}
\varrho(\alpha)  \sim \frac{1}{2 \pi } \frac{d}{b}(\alpha-\pi)^{2m}, \qquad \alpha \to \pi.
\end{equation}
The order of multicriticality $m$, as set by Definition~\ref{def:msm}, hence determines the vanishing exponent of $2m$. The same phenomenon was observed in~\cite{Periwal_Shevitz_1990}, and is analogous to the behaviour found in~\cite{Kazakov_1989} which inspired the use of the term ``multicritical'' in this context. 

\subsection{Perspectives on connections with Hermitian matrix models} 

Returning to our main result, Theorem~\ref{thm:multicritical}, the partition function $Z_\ell = e^{\theta^2 \sum_r r \gamma_r^2} \P_\theta^m(\lambda_1 \leq \ell)$ of an order $m$ multicritical unitary matrix model satisfies
\begin{equation}
\lim_{\theta \to \infty} e^{-\theta^2 \sum_r r \gamma_r^2}Z_{\lfloor b \theta + s (d\theta)^{1/(2m+1)}\rfloor} = F_{2m+1}(s).
\end{equation}
In the generic $m=1$ case, the distribution on the right hand side itself reveals an asymptotic connection with a random Hermitian matrix model. Recall that $F_3(s) := \fg(s)$ is the Tracy--Widom distribution for the GUE~\cite{Tracy_Widom_1993}: if $M$ is a random $N\times N$ Hermitian matrix  distributed by a probability density proportional to $e^{-\tr M^2/2}$ with respect to the Haar measure, its maximal eigenvalue $\xi_{\max}$ satisfies
\begin{equation}
\lim_{N \to \infty }\P\left(\frac{\xi_{\max} - (2N)^{1/2}}{2^{-1/2}N^{-1/6}} < s\right) = F_{3}(s)
\end{equation}
(and in fact, for any finite $k$, the largest $k$ eigenvalues rescaled as above converge in law to the same limiting ensemble as the first $k$ parts of a random partition under the Poissonised Plancherel measure $\P_\theta$ rescaled according to~\eqref{eq:edgefluctuationsf}~\cite{Okounkov_1999}).  It is natural to ask if it is possible to tune a potential $V(M)$ such that the Hermitian matrix density proportional to $e^{-\tr V(M)}$ exhibits multicritical edge behaviour. Different multicritical asymptotic statistics which are also related to the Painlev\'e II  hierarchy have been observed for random Hermitian matrices, for example by Claeys, Its and Krasovsky~\cite{Claeys_Its_Krasovsky_2010} who tuned even-degree potentials. We do not know if this could be related to our multicritical models. Another candidate is of course Kazakov's original multicritical Hermitian matrix models~\cite{Kazakov_1989}; it would be interesting to study their edge behaviour in depth, but we note that the vanishing exponents for the eigenvalue density in these models generalise in a different way from ours, implying we should not expect them to belong to the same universality class. In the $m=1$ case of the GUE, the connection with the fermion picture is explicit, as the joint distribution of eigenvalues is precisely equivalent to that of fermions in a unidimensional harmonic trap. It is unclear if it is possible to construct a Hermitian matrix model corresponding analogously to the flat trap potentials considered in~\cite{LDMS_2018} for $m>1$. 

  The multicritical unitary matrix models may present a path to finding related Hermitian ones-- we might note naively that if $M$ is Hermitian then $\exp(iM)$ and $(i-M)(i+M)$ are both unitary, one can pass from one picture to another, but the observables we are comparing on either side (partition functions and edge distributions) are not easily related. Let us discuss a connection between another Hermitian matrix model and the Plancherel measure, which is less well understood but which exhibits a connection with unitary matrix models similar to the one in Theorem~\ref{thm:genunitary}. Consider the \emph{Laguerre unitary ensemble} (LUE) of $N\times N$ matrices for given real $\theta > 0$ and integer $\ell > 0$, with measure 
   \begin{equation}
     \P_N(M) \mathcal{D}M = \frac{1}{Z_{\mathrm{LUE}}} e^{-\tr M} (\det M)^\ell \mathcal{D} M.
     \end{equation} The induced measure on ordered sets of eigenvalues $x_1 < x_2 < \dots < x_N$ is
  \begin{equation}
    \P_N(x_1, \dots, x_N)dx_1 \cdots dx_N = \frac{1}{Z_{\mathrm{e.v.}}}\prod_{1 \leq i < j \leq N} (x_i-x_j)^2 \prod_{1 \leq i \leq N} e^{-x_i} x_i^\ell dx_1 \cdots dx_N 
  \end{equation}
  (in either case, $Z_{\mathrm{LUE}}$ and $Z_{\mathrm{e.v.}}$ are normalisations). The eigenvalues in this model form a DPP with kernel given by Laguerre polynomials---see e.g.~\cite{Forrester_2010}. If we look at the lowest eigenvalue ${x}_1$ at the ``hard edge'' at 0 and rescale the eigenvalues to  $\tilde{{x}} = {{x}}_i/N$ and take $N \to \infty$, we obtain the \emph{continuous} Bessel ensemble DPP of Tracy and Widom~\cite{Tracy_Widom_1994_2}. Moreover, it can be proven (see e.g.~\cite{Borodin_Forrester_2003} and references therein) that the gap probability for the interval $(0, 4\theta^2)$ in the continuous Bessel ensemble equals a similar gap probability in the \emph{discrete} Bessel ensemble. In terms of a unitary matrix integral, we have \cite[Equation~(2.8)]{Borodin_Forrester_2003}:
  \begin{equation} \label{eq:laguerreequality}
    \lim_{N \to \infty} \P_N \left( \frac{x_1}{N} > 4 \theta^2 \right) = e^{-\theta^2} \int_{\mathcal{U}(\ell)} e^{\theta \tr (U+U^{-1})} \mathcal{D}U
  \end{equation}
  As we showed for the right hand side in Theorem~\ref{thm:genunitary}, both quantities above are Fredholm determinants so we have
  \begin{equation}
    \det (1-\mathsf{J}^{\ell})_{L^2(0, 4 \theta^2)} = \det (1-\mathcal{J}_\theta)_{l^2 (\ell+ \frac12 + \Z_{\geq 0})}
  \end{equation}
  where $\mathsf{J}^{\ell}$ is the continuous Bessel function of~\cite{Tracy_Widom_1994_2}, defined as
  \begin{equation}
     \mathsf{J}^{\ell}   (x, y) = \int_{0}^1 J_{\ell} (2 \sqrt{ux}) J_{\ell} (2 \sqrt{uy}) du
\end{equation}
and $\mathcal{J}_\theta$ is the usual discrete Bessel kernel defined by~\eqref{eq:kgenbessel} with $\gamma_1 =1$ and all other $\gamma_r$ equal to zero. 
It is possible that the equality~\eqref{eq:laguerreequality} is not a mere coincidence and so might have a multicritical extension, and hence define a ``multicritical Laguerre ensemble'' (although we do not know what a natural definition of multicriticality for such a model would be). Let us note that this formula was generalised in~\cite{Moriya_2019}.

\section{Conclusions and perspectives}

To summarise the results presented here, we have found discrete models belonging to the same universality classes as the models of trapped fermions of~\cite{LDMS_2018}, distinguished by non-generic  ``multicritical'' interface fluctuations. In particular, these fluctuations are asymptotically governed by higher-order analogues of the TW distribution, related with solutions of the Painlevé II hierarchy. This hierarchy also arises in the context of multicritical unitary matrix models~\cite{Periwal_Shevitz_1990,Periwal_Shevitz_1990_2}, and the multicritical measures on partitions that we introduce help to explain this connection, as they are in exact correspondence with both unitary matrix models and lattice fermion models. Recently, the same correspondence has led to further exact (that is, ``pre-limit'') relations for the multicritical measures by Chouteau and Tarricone, who found discrete analogues of the higher-order Painlevé II equations~\cite{Chouteau_Tarricone_2022}.

A peculiarity of the multicritical Schur measures is that they are not
defined in terms of Schur-positive specializations: this seems to
preclude the possibility of extending them into time-dependent
probabilistic Schur processes. At the combinatorial level, Schur
processes are deeply related with the Robinson--Schensted--Knuth (RSK)
correspondence and its variants, and it is unclear how our
multicritical measures fit in this picture. In very practical terms,
the RSK correspondence allows for very efficient sampling of Schur
processes~\cite{BBBCCV2018}: this method does not readily adapt to the
multicritical setting, and it is natural to ask whether there exists
any efficient algorithm to sample multicritical Schur measures,
besides the generic algorithms such as the one used to generate
Figure~\ref{fig:intro_young_diagrams}.

By relaxing some conditions in the definition of the multicritical Schur measure, we have identified two immediate directions in which to extend this work. Firstly, we can study non-integer orders of multicriticality $m$ by analytically extending the coefficients $\gamma_r$ defining the minimal measures $\P^{\aa/\ss, m}_\theta$, in analogy with the Ambj{\o}rn, Budd and Makeenko's generalisation of the multicritical Hermitian matrix models~\cite{Ambjorn_2016}. This presents an interesting analytic challenge as the support of the $\gamma_r$ becomes infinite. Secondly, we used a somewhat technical requirement to ensure that some Fourier frequencies associated with the multicritical Schur spanned a single interval. Removing this condition again changes the nature of asymptotic analysis, and appears to lead to new asymptotic edge fluctuations, as discussed in~\cite{HW_splitfermi}.

\section*{Acknowledgements}
We thank Saverio Bocini, Mattia Cafasso, Guillaume Chapuy, Thomas Chouteau, Tom Claeys, Valentin Féray, Taro Kimura, Arno Kuijlaars, Pierre Le Doussal, Alessandra Occelli, Grégory Schehr, Jean-Marie Stéphan, Sofia Tarricone and Ali Zahabi for support, conversations and feedback regarding this project. 

\begin{appendix}

\section{Reminders on Schur measures}
\label{sec:rem}

In this appendix we recall the following seminal result:
\begin{thm}[Determinantal point process associated with the Schur measure~\cite{Okounkov_2001}] \label{thm:schurdpp}
Fix two sequences $t= (t_1,t_2,\ldots)$ and $t'= (t'_1,t'_2,\ldots)$ such that $\P(\lambda) := e^{-\sum_r rt_r t_r'} s_\lambda[t]s_\lambda[t']$  is a Schur measure, and let $\lh$ be a random partition under that measure. Then, for each finite set  $\{k_1,\ldots,k_n\} \subset \Z+\tfrac12$, we have
\begin{equation} \label{eq:schurdpp}
\P(\{k_1,\ldots,k_n\} \subset S(\lh)) =\rho_n(k_1,\ldots, k_n) = \det_{1 \leq i,j \leq n} K(k_i,k_j)
\end{equation}
where 
\begin{equation} \label{eq:kgenbessel}
K(k,\ell) = \sum_{i=0}^\infty J_{k+i+1/2}(t,t') J_{\ell+i+1/2}(t,t')
\end{equation}
where $J_n(t,t')$ is the {multivariate Bessel function}
\begin{equation} \label{eq:mvbessel}
J_n(t,t') = \frac{1}{2\pi i} \oint \exp\bigg[\sum_r t_r z^r-\sum_rt'_r z^{-r}\bigg] \frac{dz}{z^{n+1}}.
\end{equation}
The kernel $K$ is generated by 
\begin{equation} \label{eq:kgenfn}
\sum_{k,\ell \in \Z +\frac12} z^{k}w^{-\ell} K(k,\ell) = \frac{\exp\big[\sum_r t_r z^r-\sum_rt'_r z^{-r}\big]}{\exp\big[\sum_r t_r w^r - \sum_rt'_r w^{-r} \big]}\frac{\sqrt{zw}}{z-w}, \qquad |w|<|z|.
\end{equation}
\end{thm}

This is summarised in the Hermitian case $t' = t^*$ in Section~\ref{sec:lattice}, by way of a lattice fermion model. Here we use the same anti-commuting operators and partition-indexed vectors to define a determinantal point process, but in a self-contained way without reference to quantum mechanics. Following~\cite[Appendix A]{Okounkov_2001}, we consider the space spanned by the vectors $\ket{S}$ indexed by sets of distinct half-integers $S\subset \Z+\frac12$ (making a change of notation from Section~\ref{sec:lattice}) such that both the set $S \setminus (\Z_{\leq 0} - \frac12)$ of positive half-integers in $S$ and the set $(\Z_{\leq 0} - \frac12)\setminus S$ of negative half-integers not in $S$ are finite. We equip this space with the inner product 
\begin{equation}
\braket{S}{T} = \delta_{S,T}.
\end{equation} 
We define the action of the creation and annihilation operators $c_k^\dagger, c_k$  on the vectors $\ket{S}$ by
\begin{align} \
c_k^\dagger\ket{S} = \begin{cases}
(-1)^{N_k}\ket{S \cup \{k\}}  &\text{if }k \notin S \\
0 & \text{if }k \in S
\end{cases},\: c_k\ket{S} = \begin{cases}
0 &\text{if }k \notin S \\
(-1)^{N_k}\ket{S \setminus \{k\}}  &\text{if }k \in S
\end{cases} \label{eq:fermioncs}
\end{align}
where $N_k:=|S\setminus (\Z_{<k}+\frac12 )|$ is the number of elements greater than $k$ in $S$. 
Hence, $c_k^\dagger$ and $c_k$ are adjoint with respect to $\braket{\cdot}{\cdot}$, and the orthonormalisation of the basis $\{\ket{S}\}$ ensures that they must satisfy the  canonical anti-commutation relations~\eqref{eq:anticomm}. 

\begin{figure}
  \centering
  \def\svgwidth{0.5\columnwidth}
  {\scriptsize 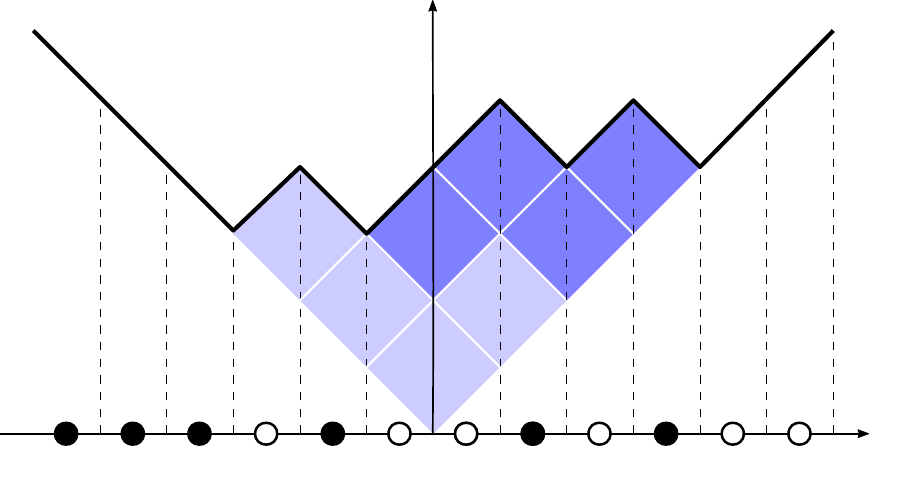}
  \caption{The Young diagram of the partition $\lambda = (4,3,1)$, with the corresponding fermion configuration $S(\lambda) = (\frac{7}{2},\frac{3}{2},-\frac{3}{2},-\frac{7}{2},-\frac{9}{2},-\frac{11}{2}, \ldots)$ shown below. The darker boxes form a ribbon of length $4$, and adding this ribbon to $\mu = (2,1,1)$ corresponds to moving the fermion at position $-\frac{1}{2}$ in $S(\mu) =(\frac{3}{2},-\frac{1}{2},-\frac{3}{2},-\frac{7}{2},-\frac{9}{2},-\frac{11}{2}, \ldots) $ to position $\frac{7}{2}$.}
  \label{fig:ydtomaya}
\end{figure}

In terms of the set $S(\lambda)$ defined at~\eqref{eq:halfintegerset}, the partition-indexed vectors already defined at~\eqref{eq:partitionstate} are $\ket{\lambda} := \ket{S(\lambda)}$, and in particular the vector corresponding to the empty partition (or domain wall state) is $\ket{\emptyset} := \ket{S(\emptyset)}$ is indexed by the negative half-integers $S(\emptyset) = \{-\tfrac12,-\tfrac32,-\tfrac52,\ldots\}$. For all $\lambda$ we have
\begin{equation}
|S(\lambda) \setminus (\Z_{\leq 0} - \tfrac12)| = | (\Z_{\leq 0} - \tfrac12) \setminus S(\lambda)|.
\end{equation}
The bosonic creation and annihilation operators $a_{\pm r}$ defined at~\eqref{eq:ardef} preserve $|S \setminus (\Z_{\leq 0} - \tfrac12)| - | (\Z_{\leq 0} - \tfrac12) \setminus S|$ when acting on a state $\ket{S}$, and their action on the state $\ket{\lambda}$ has a natural Young-diagrammatic interpretation: we have
\begin{equation}
a_{-r} \ket{\lambda} = \sum_{\mu = \lambda + \square^r } \ket{\mu}
\end{equation}
where the sum is taken over all partitions $\mu$ whose Young diagrams differ from that of $\lambda$ by the addition of a ``ribbon'' of length $r$, and in particular $(a_1)^n \ket{\emptyset} = \sum_{\lambda: |\lambda| = n} \ket{\lambda}$; see Figure~\ref{fig:ydtomaya}. 

\begin{proof}[{Proof of Theorem~\ref{thm:schurdpp}}]
We first write the Schur measure in terms of inner products on the vector space described above. Fix two  sequences $t = (t_1, t_2, \ldots)$ and $t' = (t_1',t_2',\ldots)$, and let
\begin{equation}
\Gamma_{\pm}(t) := \exp\bigg[ \sum_{r\geq 1} t_{r} a_{{\pm}r}\bigg]. 
\end{equation}
Note that $\Gamma_+(t) \ket{\emptyset} = \ket{\emptyset}$ and $\bra{\emptyset}\Gamma_-(t) = \bra{\emptyset}$. 

\begin{lem} \label{lem:schurfromdpp} For any partition $\lambda$, we have
\begin{equation}
\bra{\emptyset} \Gamma_+(t) \ket{\lambda} = s_\lambda[t]
\end{equation}
and the Schur measure may be written
\begin{equation} \label{eq:fockspaceschurmeasure}
\P(\lambda) = \frac{1}{Z} \bra{\emptyset} \Gamma_+(t) \ket{\lambda} \bra{\lambda} \Gamma_-(t') \ket{\emptyset}
\end{equation}
where the normalisation is $Z = \bra{\emptyset} \Gamma_+(t) \Gamma_-(t') \ket{\emptyset} = e^{\sum_r r t_r t_r'}$. 
\end{lem}

\begin{miniproof}
From the anti-commutation relations~\eqref{eq:anticomm}, we have
\begin{equation}
[a_r,c^\dagger(z)] = z^rc^\dagger(z),\qquad [a_r,c(w)] = -cw^r c(w),
\end{equation}
in terms of the generating functions 
\begin{equation}
c^\dagger(z):= \sum_k z^k c^\dagger_k  \quad \text{and} \quad  c(w):= \sum_\ell w^{-\ell} c_\ell.
\end{equation}
Then, from the formula 
\begin{equation}
e^{A}B = \sum_{n=0}^\infty \frac{1}{n!} \underbrace{[A,[A,\ldots[A}_{n\text{ times}},B]\ldots]] e^A
\end{equation}
we obtain
\begin{equation} \label{eq:gammacomms}
\Gamma_{\pm} (t)c^\dagger(z) = e^{\sum_r t_r z^{\pm r}}  c^\dagger(z) \Gamma_{\pm}(t), \qquad \Gamma_{\pm}(t) c(w) = e^{-\sum_r t_r z^{\pm r}}  c(w) \Gamma_{\pm}(t).
\end{equation}
Recalling that $e^{\sum_r t_r z^r} =:\sum_i h_i[t] z^i$ generates the complete homogeneous symmetric functions as defined in~\eqref{eq:hgenfn}, we extract coefficients to recover
\begin{align}
\Gamma_+(t) c^\dagger_k = \sum_{i=0}^\infty h_i[t]c^\dagger_{k-i} \Gamma_+(t) =: \hat{c}^\dagger_k \Gamma_+(t),  \qquad \Gamma_+(t) c_k = \sum_{m=0}^\infty h_i[t]c_{k+i} \Gamma_+(t) =: \hat{c}_k \Gamma_+(t),
\end{align}
 so $\hat{c}_k^\dagger$ and $\hat{c}_k$ are linear combinations of the $c^\dagger_k$ and $c_k$ respectively. Hence,
 we can apply Wick's lemma~\cite{Wick_1950} to obtain
\begin{align}
\bra{\emptyset} \Gamma_+(t) \ket{\lambda} &= \langle \Gamma_+(t) c^\dagger_{\lambda_1 - \frac{1}{2}} c_{- \frac{1}{2}}  c^\dagger_{\lambda_2 - \frac{3}{2}} c_{-\frac{3}{2}} \ldots  c^\dagger_{\lambda_{\ell(\lambda)}  - \ell(\lambda)  + \frac{1}{2}} c_{-\ell(\lambda) + \frac{1}{2}} \rangle \notag \\
&= \langle \hat{c}^\dagger_{\lambda_1 - \frac{1}{2}} \hat{c}_{- \frac{1}{2}}  \hat{c}^\dagger_{\lambda_2 - \frac{3}{2}} \hat{c}_{-\frac{3}{2}} \ldots  \hat{c}^\dagger_{\lambda_{\ell(\lambda)}  - \ell(\lambda)  + \frac{1}{2}} \hat{c}_{-\ell(\lambda) + \frac{1}{2}} \rangle \notag \\
&= \det_{1\leq i,j \leq \ell(\lambda)} \langle  \hat{c}^\dagger_{\lambda_i  - i  + \frac{1}{2}} \hat{c}_{-j + \frac{1}{2}} \rangle
\end{align}
Since the complete homogeneous functions satisfy $\sum_i h_{n-i}[t] h_{i}[t] = h_n[t]$, the matrix element is 
\begin{align}
\langle  \hat{c}^\dagger_{\lambda_i  - i  + \frac{1}{2}} \hat{c}_{-j + \frac{1}{2}} \rangle &= \sum_{m, n} h_m[t] h_n[t] \delta_{\lambda_i  - i  -m,n-j } \notag \\
&= \sum_{n} h_{\lambda_i  - i + j - n}[t] h_n[t] = h_{\lambda_i  - i + j}[t];
\end{align}
recalling the expression~\eqref{eq:jt1} for the Schur function, we have
\begin{equation} \label{eq:schurdecomp}
 \bra{\emptyset} \Gamma_+(t) \ket{\lambda} =  \det_{1 \leq i,j \leq \ell(\lambda)} h_{\lambda_i - i + j} = s_\lambda[t] 
\end{equation}
as required. Since we similarly have $\bra{\lambda} \Gamma_-(t') \ket{\emptyset}$, we have
\begin{equation}
\sum_{\lambda} s_\lambda[t] s_\lambda[t'] =  \sum_{\lambda}\bra{\emptyset} \Gamma_+(t)  \ket{\lambda} \bra{\lambda} \Gamma_-(t') \ket{\emptyset} = \bra{\emptyset} \Gamma_+(t)  \Gamma_-(t') \ket{\emptyset}
\end{equation}
as the sum of projections $\sum_\lambda \ket{\lambda} \bra{\lambda} $ is simply the identity. 

By application of the {Baker--Campbell--Hausdorff} formula $e^A e^B = e^{[A,B]}e^Be^A$ where $[A,[A,B]] = [B,[A,B]]=0$, we have
\begin{equation} \label{eq:BCHform}
\Gamma_+(t)\Gamma_-(t') = e^{\sum_r r t_rt_r'}  \Gamma_-(t')\Gamma_+(t)
\end{equation}
and hence the normalisation is $Z = \bra{\emptyset} \Gamma_+(t)  \Gamma_-(t') \ket{\emptyset} = e^{\sum_r r t_rt_r'} $, giving the expression for the Schur measure required. 
\end{miniproof}

Now, consider the random set of distinct half integers $S(\lambda)$ where $\lambda$ is distributed by the Schur measure. From the expression~\eqref{eq:fockspaceschurmeasure} for the Schur measure, the $n$-point correlation function on this set is\footnote{If we fix $t' = t^*$ this corresponds precisely to the  $n$-point correlation function for fermions in the lattice model described in Section~\ref{sec:lattice}, as $\mathcal{U} = \Gamma_+(t)\Gamma_-(t^*)^{-1}$. }
\begin{equation} 
\P(\{k_1,\ldots,k_n\} \subseteq S(\lambda) ) = \rho_n(k_1,\ldots,k_n) = \frac{1}{Z}\bra{\emptyset}  \Gamma_+(t) c_{k_1}^\dagger c_{k_1} \cdots c_{k_n}^\dagger c_{k_n} \Gamma_-(t')\ket{\emptyset} 
\end{equation}
for any finite set of half-integers $\{k_1,\ldots,k_n\}$. We will use the notation $\langle \cdot \rangle := \bra{\emptyset} \cdot \ket{\emptyset}$ for the expectation on the domain wall state.
\begin{lem} \label{lem:dpp}
We have
\begin{equation}
 \rho_n(k_1,\ldots,k_n) = \det_{1 \leq i,j \leq n} K(k_i,k_j)
\end{equation}
for a kernel
\begin{equation} \label{eq:proofpropagator}
K(k,\ell) = \bra{\emptyset} \Gamma_+(t) \Gamma_-(t')^{-1} c_k^\dagger c_\ell \Gamma_-(t')\Gamma_+(t)^{-1}\ket{\emptyset} 
\end{equation}
which is given by~\eqref{eq:kgenbessel}, and has generating function~\eqref{eq:kgenfn}. 
\end{lem}

\begin{miniproof} 
Setting
\begin{equation} \label{eq:ctildedef}
\tilde{c}_{k}^\dagger = \Gamma_+(t) \Gamma_-(t')^{-1}c^\dagger_k \Gamma_-(t')\Gamma_+(t)^{-1}, \quad \tilde{c}_{k} = \Gamma_+(t) \Gamma_-(t')^{-1}c_k \Gamma_-(t')\Gamma_+(t)^{-1} , 
\end{equation}
we have 
\begin{align}
\rho_n(k_1,\ldots,k_n) &= \frac{1}{Z}\langle \Gamma_+(t) \Gamma_-(t')\Gamma_+(t)^{-1} \tilde{c}_{k_1}^\dagger \tilde{c}_{k_1} \cdots\tilde{c}_{k_n}^\dagger \tilde{c}_{k_n} \Gamma_+(t) \Gamma_-(t')^{-1} \Gamma_-(t')\rangle \notag \\
&= \langle \tilde{c}_{k_1}^\dagger \tilde{c}_{k_1} \cdots\tilde{c}_{k_n}^\dagger \tilde{c}_{k_n} \rangle. 
\end{align}
Note that by~\eqref{eq:gammacomms}, the $\tilde{c}_k$ are linear combinations of the $c_k$. We can therefore  apply Wick's lemma to obtain
\begin{equation}
\rho_n(k_1,\ldots,k_n) = \det_{1\leq i,j\leq n} \langle \tilde{c}_{k_i}^\dagger \tilde{c}_{k_j} \rangle.
\end{equation}

The generating function of $K(k,\ell)$ is, from~\eqref{eq:gammacomms},
\begin{align}
\sum_{k,\ell} z^k w^{-\ell} K(k,\ell) &= \langle \Gamma_+(t) \Gamma_-(t')^{-1}c^\dagger(z) c(w) \Gamma_-(t')\Gamma_+(t)^{-1}\rangle \notag \\
&= e^{\sum_r t_r z^{r}- t_r z^{-r}}\langle c^\dagger(z) c(w)\rangle e^{\sum_r t_r' w^{-r}-\sum_r t_r w^r}. \label{eq:kgeninter}
\end{align}
To obtain an explicit expression, we evaluate the term $\langle c^\dagger(z) c(w) \rangle$ and get, for $|w|<|z|$,
\begin{equation}
\sum_{k,\ell \in \Z +\frac12} \frac{z^k}{w^\ell} \langle c^\dagger_k c_\ell \rangle = \sum_{k < 0 } \frac{z^k}{w^\ell} \delta_{k=\ell} =  \frac{\sqrt{zw}}{z-w},
\end{equation}
which gives~\eqref{eq:kgenfn} as required. To write $K(k,\ell)$ in terms of the multivariate Bessel functions defined at~\eqref{eq:mvbessel}, we manipulate the formal series in~\eqref{eq:kgeninter} further and get
\begin{align}
\sum_{k,\ell} z^k w^{-\ell} K(k,\ell) &= \sum_{k,\ell}z^k w^{-\ell}\sum_{m\in \Z}  z^{m} J_{m}(t,t') \sum_{n\in \Z} w^{-n} J_{n}(t,t') \ind_{k=\ell, k<0} \notag \\
&= \sum_{i=0 }^\infty \sum_{m,n \in \Z} z^{m-i-\frac12} J_{m}(t,t')  w^{i-n+ \frac12} J_{n}(t,t') \notag  \\
& = \sum_{k,\ell}  z^{k}w^{-\ell} \sum_{i=0}^\infty J_{k+i+\frac12}(t,t')  J_{\ell+i+\frac12}(t,t') .
\end{align}
This recovers~\eqref{eq:kgenbessel} as required.  \end{miniproof}

This concludes the proof. \end{proof}

\section{Cylindric multicritical Schur measures and positive temperature edge fluctuations} \label{sec:mc_cylindric}

In~\cite{LDMS_2018}, the authors found a direct generalisation of the higher-order TW-GUE distribution for the fluctuations in the largest momentum in a grand canonical ensemble of fermions in a 1D flat trap potential at positive temperature.
Here, we will construct a discrete model with the same asymptotic edge behaviour, as an instance of the \emph{periodic Schur process}~\cite{Borodin_2007}. Indeed, it was shown in~\cite{Betea_Bouttier_2019} that the periodic Schur process can be interpreted as a system of fermions at positive temperature (the discussion in Section~\ref{sec:lattice} corresponding to the zero temperature case).  In particular, the positive temperature generalization of the Poissonised Plancherel measure is a measure on pairs of partitions which gives rise to fluctuations governed by Johansson's positive temperature generalisation of the TW-GUE distribution~\cite{Johansson_2007} in a suitable asymptotic regime (see~\cite[Theorem~1.1]{Betea_Bouttier_2019}).

 We may similarly generalise the multicritical Schur measures to the positive temperature setting. Let $\lambda$ and $\mu $ be two partitions and let $t = (t_1,t_2,\ldots)$ be sequence of Miwa times. The \emph{skew Schur function} $s_{\lambda/\mu}[t]$ is defined via the Jacobi--Trudi identity as 
 \begin{equation}
 s_{\lambda/\mu}[t] = \det_{1 \leq i,j \leq \ell(\lambda)} h_{\lambda_i - i -\mu_j + j}[t]
 \end{equation}
 where $\sum_k h_k[t]z^k = \exp [ \sum_{r\geq 1} t_r z^r]$ as in~\eqref{eq:hgenfn}. Note that $s_{\lambda/\mu} = 0$ if $\lambda_i < \mu_i$ for some $i$. Then, we have the following definition:
\begin{defn}[Cylindric multicritical measure]
Let $\gamma = (\gamma_1,\gamma_2,\ldots)$ be a sequence of real numbers defining an order $m$ multicritical measure by the conditions of Definition~\ref{def:msm} with right edge and fluctuation coefficients $b,d$, let $\theta$ and $u$ be non-negative parameters with $u<1$. Then, the measure on pairs of partitions $(\lambda,\mu)$ 
\begin{equation}
\P^{m}_{u,\theta} (\lambda,\mu) = \frac{1}{Z}u^{|\mu|}s_{\lambda/\mu}[\theta \gamma]^2, \qquad Z = \frac{\exp[\frac{\theta^2}{1-u} \sum_{r} r^2 \gamma_r^2]}{\prod_{i\geq 1}(1-u^i)}
\end{equation} 
is called an order $m$ \emph{cylindric multicritical measure}. 
\end{defn}

From the partition function $Z$, we see that
\begin{equation}
\E^{m}_{u,\theta}(|\lh|) = \frac{\theta^2}{(1-u)^2} \sum_r r^2 \gamma_r^2 - u \frac{d}{du}\log (u;u)_{\infty}.
\end{equation}
Since $\log(u;u)_{\infty} \sim -\frac{\pi^2}{6(1-u)}$ as $u \to 1$, we see that the first term dominates for $\theta \to \infty$, whether $u$ is fixed or tends to 1. Hence, $\Theta := \theta/(1-u)$ asymptotically defines a natural length scale for the parts $\lh_i,\lh_i'$.

For the cylindric multicritical measures, we have the following positive temperature generalisation of Theorem~\ref{thm:multicritical}, which is also a multicritical generalisation of~\cite[Theorem~1.1]{Betea_Bouttier_2019}:

\begin{thm}[Asymptotic edge fluctuations of cylindric multicritical measures] \label{thm:cylindricmulticritical}
Let $(\lambda,\mu)$ be a random pair of partitions under a cylindric multicritical measure $P^{m}_{u,\theta} $ with right edge and fluctuation coefficients $b,d$. 
Then, for any $\alpha >0$, in the critical scaling regime $\theta \to \infty $, $u \to 1$ with $\theta(1-u)^{2m} \to \alpha^{2m+1}d >0$, we have
\begin{equation}
  \P_{u,\theta}^{m} \left(\frac{\lh_1 - b \Theta}{(d  \Theta)^{\frac{1}{2m+1}}} < s\right) \to F_{2m+1}^\alpha(s) := \det(1-\mathcal{A}_{2m+1}^\alpha)_{L^2([s,\infty))}
\end{equation}
with $\Theta :=\frac{\theta}{1-u} \sim \big(\frac{\theta}{\alpha}\big)^{\frac{2m+1}{2m}}  d^{-\frac{1}{2m}}$ and $F_{2m+1}^\alpha$ the Fredholm determinant of the higher-order $\alpha$-Airy integral kernel
\begin{equation}
\mathcal{A}_{2m+1}^\alpha(x,y) := \int_{-\infty}^\infty \frac{e^{\alpha v}}{1+ e^{\alpha v}} \Ai_{2m+1} (x+ v) \Ai_{2m+1}(y+v) dv. 
\end{equation}
\end{thm} 

Here again, $\alpha$ plays the role of a limiting inverse temperature, and in the limit $\alpha \to \infty$ we have $F_{2m+1}^\alpha \to F_{2m+1}$. We note that critical exponents are unchanged by the passage to finite temperature in this regime once we replace the large parameter $\theta$ with $\Theta$, which also tends to infinity. The Fredholm determinants $F_{2m+1}^\alpha $ have been related to an {integro-differential} generalisation of the Painlevé II hierarchy  by Krajenbrink~\cite{Krajenbrink_2020}, who  generalised an approach of Amir, Corwin and Quastel~\cite{Amir_Corwin_Quastel_2011} from the $m=1$ case, and by Bothner, Cafasso and Tarricone~\cite{Bothner_Cafasso_Tarricone_2021}, who used a rigorous Riemann--Hilbert approach.

\paragraph*{Determinantal point process in the grand canonical
  ensemble} Periodic Schur processes are in general not determinantal,
as first observed by Borodin~\cite{Borodin_2007}, who showed how to
remedy to this issue via a procedure called \emph{shift-mixing}. In
the language of fermions, this amounts to passing to the grand
canonical ensemble~\cite{Betea_Bouttier_2019}. Applying this procedure
to the cylindric multicritical measure $\P_{u,\theta}^{m}$, we find
that the shifted half-integer set
\begin{equation}
  \label{eq:Scdef}
  S_{{c}}(\lh) = \{ \lh_i - i + {c} + \tfrac12, i \in\Z_{\geq 1} \}
\end{equation}
is a DPP when $c$ is distributed according the \emph{discrete Gaussian distribution}
\begin{equation} \label{eq:claw}
\P(c) = \frac{t^cu^{c^2/2}}{\vartheta_3(t;u)}. 
\end{equation}
Here, $u$ is the same parameter as that of $\P_{u,\theta}^{m}$, but $t$ can be chosen arbitrarily (it is related with the fermionic chemical potential). The normalization $\vartheta_3(t;u) := \sum_{c\in \Z} t^c u^{c^2/2}$ is a Jacobi theta function.

By~\cite[Theorem~A]{Borodin_2007} or \cite[Theorem~3.1]{Betea_Bouttier_2019},
the correlation kernel of $S_{{c}}(\lh)$ reads explicitly
\begin{align} 
    \mathcal{J}_{u,t,\theta}^m(k,\ell) &=  \sum_{i \in \Z} \frac{t u^i}{1+t u^i}J_{k+i+\frac{1}{2}}(\Theta\gamma)J_{\ell+i+\frac{1}{2}}(\Theta\gamma) \label{eq:ft_kernel}\\
    &=\frac{1}{(2 \pi i)^2} \oiint_{c_+,c_-}  \frac{\exp[\Theta S(z,k/\Theta)]} {\exp[\Theta S(w,\ell/\Theta)]} \cdot \frac{\kappa(z, w) dzdw}{wz}, \quad  c_\pm: |z| =  u^{\mp 1/4},  \notag \\
\kappa(z, w) &= \sum_{i \in \Z+\frac{1}{2}} \frac{t u^i}{1+t u^i} \left(\frac{z}{w}\right)^i= \sqrt{\frac{w}{z}} \cdot \frac{(u;u)^2_{\infty}} {\vartheta_u(w/z)} \cdot \frac{\vartheta_3 (t z/w; u)}{\vartheta_3(t;u)}. \label{eq:kappacyl}
\end{align}
using the notation $\vartheta_u(x) := (x; u)_\infty (u/x; u)_\infty$ and reusing the action notation for the order $m$ multicritical measure defined at~\eqref{eq:actiondef}. The equivalence between the two forms of $\kappa$ is a special case of Ramanujan's ${}_1\Psi_1$ summation~\cite{Gasper_Rahman_2004}, and the choice of contours with $|w|<|z|$ ensures the sum converges. 
Note the similarity with the integral expression for the zero temperature kernel~\eqref{eq:kgenbessel}. The proof of this in~\cite{Betea_Bouttier_2019} adapts Okounkov's fermionic approach (see Theorem~\ref{thm:schurdpp}) to the positive temperature setting, the $\kappa(z,w)$ given in~\eqref{eq:kappacyl} is the corresponding generating function $\langle c^\dagger(z)c(w)\rangle_{u,t} = \sum_{k,\ell} z^kw^{-\ell}\langle c^\dagger_kc_\ell \rangle_{u,t}$ of propagators.

\paragraph*{The crossover regime} The asymptotic regime of Theorem~\ref{thm:cylindricmulticritical} is the one in which the ``thermal'' fluctuations coming from the factor of   $u^{|\mu|}$ match the order of magnitude of the ``quantum'' fluctuations coming from the skew Schur functions, so that $\alpha$ parametrises a crossover between regimes where either kind of fluctuation dominate. Heuristically, from the identification $u =e^{-1/T}$ where $T$ is the (dimensionless) temperature, the thermal fluctuations are of order of $T$, so comparing with scale of the fluctuations in the zero temperature case (i.e. the multicritical Schur measure)  we look for a regime in which
\begin{equation}
T \sim \Theta^{\frac{1}{2m+1}}.  
\end{equation}
Fixing a specific regime 
\begin{equation}
u:= \exp\big[-\alpha (d \Theta)^{-\frac{1}{2m+1}}\big], \qquad \theta := \alpha d^{-\frac{1}{2m+1}}\Theta^{\frac{2m}{2m+1}}
\end{equation} 
by this reasoning, it is straightforward to see that it is asymptotically equivalent to the crossover regime in the statement.

\begin{proof}[{Proof of Theorem~\ref{thm:cylindricmulticritical}}]
 Our proof follows that of~\cite{Betea_Bouttier_2019}, with some adaptations that correspond precisely to the arguments of Section~\ref{sec:mcedge_fluctuations} of this text. It consists of three steps.
  
 \paragraph*{(i) Shift-mixing}

 Let $(\lambda,\mu)$ be distributed according to $\P_{u,\theta}^{m}$, and $c$ distributed according to~\eqref{eq:claw} with $t=1$. Then, 
by the determinantal nature of the shift-mixed process~\eqref{eq:Scdef}, we have
\begin{equation}
\P(\lh_1 + {c} <  \ell_s ) = \det(1 - \mathcal{J}_{u,1,\theta}^m )_{l^2(\ell_s + \Z_{\geq 0})} 
\end{equation} 
with $\ell_s := \lfloor b\Theta + (d\Theta)^{\frac{1}{2m+1}}\rfloor - \frac{1}{2}$
and $\mathcal{J}_{u,1,\theta}^m$ the correlation kernel~\eqref{eq:ft_kernel} specialized at $t=1$. Our task is now to analyze the Fredholm determinant
in the asymptotic regime of the theorem, where $\theta \to \infty$,  $u \to 1$ with $\theta (1-u)^{2m} \to \alpha^{2m+1}$ (and $\Theta:= \theta/(1-u)$).

\paragraph*{(ii) Asymptotic analysis} Let us start from the  integrand of $ \mathcal{J}_{u,1,\theta}^m(k,\ell)$ in a regime where $k = \lfloor b\Theta + x(d\Theta)^{1/(2m+1)} \rfloor - \tfrac12$ and $\ell = \lfloor b\Theta + y(d\Theta)^{1/(2m+1)} \rfloor - \tfrac12$, which is (suppressing floor functions)
\begin{equation}
 \frac{ \exp \left[ \Theta S(z;b) -  x (d\Theta)^{\frac{1}{2m+1}}\log z\right]}{\exp \left[ \Theta S(w;b) -  y (d\Theta)^{\frac{1}{2m+1}}\log w\right]}\cdot \kappa(z, w).
\end{equation} 
 Since $\Theta \to \infty$ in our asymptotic regime, we can directly use $\Theta$ as a large parameter, and then for everything except for the function $\kappa(z,w)$, the steepest descent analysis follows precisely the arguments of Section~\ref{sec:mcedge_fluctuations} (with just a change from $\theta$ to $\Theta$). At $z=w=1$ there is an order $2m$ saddle point, and we  use the same change of variables 
\begin{align}
  z = \exp \left[ \zeta (d \Theta)^{-\frac{1}{2m+1}} \right], \quad w = \exp \left[ \omega (d \Theta)^{-\frac{1}{2m+1}} \right]. 
\end{align}
The arguments for the tails bound generalise. The contour $c_+$ of the integral in $z$ is circle on which
\begin{equation}
 |z| = u^{-1/4} = \exp[\Re(\zeta)(d \Theta)^{-{1}/{(2m+1)}}],
 \end{equation} 
 and as $u \to 1$ this is satisfied if and $\Re(\zeta) \sim (d\Theta)^{1/(2m+1)}/4 (1-u) \sim \alpha/4$, so the central region is asymptotically parametrised by  $\zeta \in  i \R +\alpha/4 $ and $\omega = i \R -\alpha/4 $.

At the same time, $\kappa$ has a reasonable asymptotic behaviour in the above regime and on the contours $c_\pm$. First, when $z, w$ are around around 1, observing that $z = u^{-\zeta/\alpha}, w = u^{-\omega/\alpha}$, we have
\begin{equation}
  \kappa(z, w) = \sum_{i \in \Z+\frac{1}{2}} \frac{(z/w)^i}{1+u^{-i}} \sim \alpha (d \Theta )^{-\frac{1}{2m+1}} \cdot \frac{\pi} {\sin \frac{\pi (\zeta - \omega)}{\alpha}} \quad \text{as } \Theta \to \infty.
\end{equation}
This follows by the same argument as that leading to~\cite[Equation (5.32)]{Betea_Bouttier_2019}: putting $u = e^{-r}$ and $z/w = e^{r/2+i\phi}$ for $\phi \in [-\pi, \pi]$, by the Poisson summation formula we have 
\begin{align}
\kappa(z,w) &= \sum_{k \in \Z + \frac{1}{2}} \frac{e^{i\phi k}}{2 \cosh \frac{rk}{2}} =  \sum_{n \in \Z } (-1)^n \frac{\pi}{r\cosh \frac{\pi(\phi - 2 \pi n)}{r}}
\end{align}
and on the contours $c_{\pm}$\footnote{Let us note that in this instance, we cannot readily switch to contours angled at $m \pi/2m+1$, due to the poles of $\kappa$ on the real line. } as $\Theta \to \infty$, $u \to 1$,
\begin{equation}
\kappa(z,w) \sim \frac{\pi}{r\cosh \frac{\pi \Im(\zeta - \omega)}{\alpha}} = \frac{\pi}{r\sin \frac{\pi (\zeta - \omega)}{\alpha}}
\end{equation}
The prefactor $(d \Theta )^{-\frac{1}{2m+1}}$ will be cancelled by part of the Jacobian for the change of variables $(z, w) \mapsto (\zeta, \omega)$. From the same Poisson summation formula, we see that outside of the central region around $z = w = 1$, $\kappa$ decays exponentially fast to 0, see~\cite[Lemma 5.5]{Betea_Bouttier_2019}.

Putting everything together and noting that the same exponential decay bounds imply dominated convergence,  as $\Theta \to \infty$ and $u \to 1$ we have
\begin{align}
    (d &\Theta)^{\frac{1}{2m+1}}  \mathcal{J}_{u,1,\theta}^m \left( \lfloor b\Theta + x(d\Theta)^{1/(2m+1)} \rfloor - \tfrac12, \lfloor b\Theta + y(d\Theta)^{1/(2m+1)} \rfloor - \tfrac12 \right) \notag\\
    & \to \frac{1}{(2\pi i)^2}\int_{ i \R+\frac{\alpha}{4}}  \int_{ i \R-\frac{\alpha}{4}} \frac{\exp\left[(-1)^{m+1}\frac{ \zeta^{2m+1}}{2m+1} - x \zeta\right]}  {\exp\left[(-1)^{m+1} \frac{\omega^{2m+1}}{2m+1} - y \omega\right]} \cdot \frac{\pi}{\alpha \sin \frac{\pi
    (\zeta-\omega)}{\alpha}} d\omega d\zeta. 
  \end{align}
Using the identity
\begin{equation} 
  \frac{\pi}{\alpha \sin \frac{\pi (\zeta-\omega)}{\alpha}} = \int_{-\infty}^\infty
\frac{e^{(\alpha+\omega-\zeta) v} d v}{1+e^{\alpha v}}
\end{equation}
valid for $0<\Re(\zeta-\omega)<\alpha$,
we see that the limiting kernel is equal to $\mathcal{A}^\alpha_{2m+1}(x, y)$.

The same exponential decay arguments for the integrand apply again to the integral, so the traces of $\mathcal{J}_{u,1,\theta}^m $ also converges uniformly to the traces of  $\mathcal{A}^\alpha_{2m+1}$ on any set that is bounded below. Since the Hadamard bound argument equally applies here, we have convergence of the Fredholm determinants too, with 
\begin{equation} \label{eq:shifteddist}
  \P\left[\frac{\lh_1 + {c} - b\Theta}{(d\Theta)^{\frac{1}{2m+1}}} <  s\right] \to \det (1- \mathcal{A}_{2m+1}^\alpha)_{L^2([s,\infty))}.
\end{equation}

\paragraph*{(iii) Shift removal} The limiting distribution~\eqref{eq:shifteddist} above is not quite what we wanted to prove due to the random shift ${c}$. Luckily it can be removed without affecting the result: indeed, by~\cite[Lemma~2.1]{Betea_Bouttier_2019}, ${c} / \Theta^{1/(2m+1)}$ converges to 0 in probability (recall that we set $t=1$ here).
\end{proof}

\section{Generalised higher-order Airy kernel} \label{sec:ccgairy}

In this appendix we extend the multicritical measures to have more general asymptotic edge distributions of a kind shown by Cafasso, Claeys and Girotti~\cite{Cafasso_Claeys_Girotti_2019} to encode Fredholm determinant solutions of the general Painlev\'e II hierarchy. 
The authors found that if we set
\begin{equation}
p_{\tau;2m+1}(x) :=  \frac{x^{2m+1}}{2m+1} + \sum_{i=1}^{m-1} \frac{\tau_i}{2i+1} x^{2i+1}
\end{equation}
for a given sequence of $m-1$  real constants $\tau = (\tau_1,\ldots,\tau_{m-1})$, then 
the Fredholm determinant 
\begin{equation}
F_{\tau;2m+1}(s) = \det(1-\mathcal{A}_{\tau;2m+1})_{L^2([s,\infty))}
\end{equation}
of the \emph{generalised higher-order Airy kernel}
\begin{equation}
  \label{eq:KCCG}
  \mathcal{A}_{\tau;2m+1}(x,y) = \frac1{(2\pi i)^2} \int_{i \R+1}  \int_{ i \R-1} 
  \frac{\exp[(-1)^{m+1}p_{\tau;2m+1}(\zeta)-x\zeta]}{\exp[(-1)^{m+1}p_{\tau;2m+1}(\omega)-y\omega]} \frac{d \zeta d \omega}{\zeta-\omega}.
\end{equation} 
is related to a solution $q_{\tau;m}(s)$ of the order $2m$ general Painlevé II hierarchy equation with coefficients $\tau_i$ by
\begin{equation}
F_{\tau;2m+1}(s) = \exp \bigg[ - \int_s^\infty (x-s) q_{\tau;m}^2 ((-1)^{m+1}x) \, dx\bigg].
\end{equation}
This relation generalises~\eqref{eq:fgenpainleve}, which corresponds to the case $\tau = (0,0,\ldots)$.

\paragraph{Generalised multicritical fermions and Schur measures} 
The generalised higher-order Airy functions
\begin{equation}
 \Ai_{\tau;2m+1}(x) = \frac{1}{2\pi i} \int_{i\R +1} \exp[(-1)^{m+1}p_{\tau;2m+1}(\zeta)-x\zeta] d\zeta,
 \end{equation} 
 making up the kernel $\mathcal{A}_{\tau;2m+1}$ satisfy the eigenfunction relations 
 \begin{equation}
 (-1)^{m+1} \left[ \frac{d^{2m}}{dx^{2m}} + \sum_{i=1}^{m-1} \tau_i \frac{d^{2i}}{dx^{2i}} \right] \Ai_{\tau;2m+1}(x) = -x \Ai_{\tau;2m+1}(x),
 \end{equation}
 generalising~\eqref{eq:genairyevrel}. One can adapt the flat trap models of~\cite{LDMS_2018} to recover momentum space edge Hamiltonians of the above form, generalising~\eqref{eq:edgeham}. This can be achieved for instance by considering trapping potentials of the form
\begin{equation}
V(x) = x^{2m} + \sum_i (-1)^i \tau_i \pedge^{\frac{2m - 2i}{2m+1}}x^{2i}
\end{equation} 
with the same scaling regime $\pedge \to \infty$ as that considered in Section~\ref{sec:fermions-asymptotics}; note that finer tuning is required than in the $\tau = (0,0,\ldots)$ case. We focus on a discrete construction, which coincides with the momentum space edge of such a model in a suitable continuum limit. Our main task is to identify the correct asymptotic regime. 

We again construct Hermitian Schur measures (and corresponding lattice fermion models) with a single real parameter $\theta$, but no longer require each Miwa time in the Schur function specialisation to grow linearly with $\theta$; once we consider combinations of Miwa times growing at different speeds, we can tune the speeds so that the integrand of the limiting edge kernel has a given odd polynomial in the exponential, from the same saddle point analysis of Section~\ref{sec:mcedge_fluctuations}. 

To be specific, we combine the coefficients $\gamma_r$ already used to define multicritical measures, to define generalised ones as follows (where we emphasise that the sequence of constants $\gamma$ is replaced with a $\theta$-dependant functions $\gamma^{\tau}(\theta)/\theta$):

\begin{defn}[Generalised multicritical measure] \label{def:genmmeasdef}
Fix a sequence of $m-1$ real constants $\tau = (\tau_1,\ldots,\tau_{m-1})$, and choose $m$ sequences of real coefficients $\gamma^{(1)}, \ldots, \gamma^{(m)}$ where $\gamma^{(i)}$ satisfies the conditions for an order $i$ multicritical measure and has right edge and fluctuation coefficients $b_i,d_i$. Then, for a positive parameter $\theta$, we define the sequence $\gamma^{\tau}(\theta)$ of Miwa times, with elements indexed $r \geq 1$
\begin{equation}
\gamma^{\tau}(\theta )_r = \theta \gamma_r^{(m)} + \sum_{i=1}^{m-1} \theta^{\frac{2i+1}{2m+1}} (-1)^{m-i} \frac{\tau_i}{d_i} \gamma_r^{(i)}
\end{equation}
and we define an order $m$ \emph{generalised multicritical measure }
\begin{equation}
\P^{\tau;m}_\theta(\lambda) = \frac{1}{Z} s_{\lambda}[\gamma^{\tau}(\theta )]^{2}, \qquad Z = e^{\sum_{r} r \gamma^{\tau}(\theta )_r^2}
\end{equation}
along with its edge position function 
\begin{equation} \label{eq:bigbtheta}
B(\theta) =  b_m \theta + \sum_{i=1}^{m-1} b_i (-1)^{m-i} \frac{\tau_i}{d_i} \theta^{\frac{2i+1}{2m+1}}.
\end{equation}
\end{defn}

This generalisation is defined so that we have the edge behaviour we would expect in analogy to Theorem~\ref{thm:multicritical}:

\begin{thm}[Edge fluctuations in generalised multicritical measures] \label{thm:genedgefluctuations}
If $\lh$ is a random partition under the generalised multicritical measure $\P^{\tau;m}_\theta(\lambda)$, then we have
  \begin{equation}
    \lim_{\theta \to \infty} \P^{\tau;m}_{\theta} \left[ \frac{\lh_1 - B(\theta) } {(d_m\theta)^{\frac{1}{2m+1}}} \leq s \right] = \det (1 - {\mathcal{A}}_{\tau;2m+1})_{L^2(s, \infty)} =: {F}_{\tau;2m+1}( s) .
  \end{equation}
\end{thm}

It is worth highlighting that the expected edge position $B(\theta)$ now has quite a non-trivial expansion: it has deterministic terms of orders $\theta$,
$\theta^{\frac{2n-1}{2n+1}}$, $\dots$, $\theta^{\frac{3}{2n+1}}$, and only at  order $\theta^{\frac{1}{2n+1}}$ do we encounter the fluctuations. The expected size is also more subtle: since we have $\E(|\lh|) = \sum_{r\geq 1} r^2 \gamma(\theta)_r^2 $, only the leading order term now scales with $\theta^2$.

\paragraph*{Tuning speeds and coefficients} The proof of Theorem~\ref{thm:genedgefluctuations} involves no new arguments than the ones of Section~\ref{sec:mcedge_fluctuations}, so we find it  more instructive to present an informal derivation of Definition~\ref{def:genmmeasdef}. To do so, let us define additional notation, putting
\begin{equation}
S^{(i)}(z;x) = \sum_{r\geq 1} \gamma_r^{(i)}\big(z^r - z^{-r}\big) - x\log z = V^{(i)}(z) - V^{(i)}(z^{-1}) - x\log z
\end{equation}
for the action and potential associated with the coefficients $\gamma^{(i)}$. Since each $\gamma^{(i)}$ defines an order $i$ multicritical measure with right edge and fluctuation coefficients $b_i, d_i$, we have, by~\eqref{eq:multicritaction} and~\eqref{eq:actionnonvanishing}, the following expansion of $S^{(i)}$ around $z=1$:
\begin{equation} \label{eq:actionsisaddle}
  S^{(i)}(z;b_i) = \frac{(-1)^{i+1}d_i}{2i+1} (z-1)^{2i+1} + O((z-1)^{2i+3}).
\end{equation}

Let us form a generalised potential, which now scales with $\theta$, 
\begin{equation}
 \mathsf{V}(z) = \sum_{i=1}^m f_i(\theta) V^{(i)}(z);
\end{equation}
we fix $f_m(\theta) = 1$ for convenience. Our goal is now to find suitable $f_i(\theta)$ so as to obtain the scaling regime of Theorem~\ref{thm:genedgefluctuations} and the limiting edge kernel $\mathcal{A}_{\tau;2m+1}$. We will just look at the integrand in the double contour integral representation in a region near the multicritical saddle point. The discrete kernel we start with is 
\begin{equation} \label{eq:K_def}
 \mathcal{J}_\theta^{\tau;m}(k, \ell) = \frac{1}{(2 \pi i )^2}\oiint_{c_+,c_-} \frac{ \exp[\theta ( \mathsf{V}(z)-\mathsf{V}(z^{-1})) ]} { \exp[\theta (\mathsf{V}(w)-\mathsf{V}(w^{-1}) )]} \frac{dz  dw}{z^{k+\frac12} w^{-\ell+\frac12} (z-w)}
\end{equation}
for $k,\ell \in \Z+\frac12$, with $c_+$ for the integration in $z$ passing just outside the unit circle and $c_-$ for $w$ passing just inside. 
Now we set
\begin{equation}
 \mathsf{S}(z;x) = \mathsf{V}(z)- \mathsf{V}(z^{-1}) - x\log z; \quad \mathsf{b}(\theta) := \sum_{i=1}^m f_i(\theta)b_i.
 \end{equation} 
  Then, if we rewrite the coordinates relative to $k=\mathsf{b}(\theta)+k'$, $\ell=\mathsf{b}(\theta)+\ell'$ the kernel may be written
\begin{equation} \label{eq:kernel}
\mathcal{J}_\theta^{\tau;m}(k, \ell) = \frac{1}{(2 \pi i )^2}\oiint_{c_+,c_-} \exp[\theta (\mathsf{S}(z;\mathsf{b}(\theta))-\mathsf{S}(w;\mathsf{b}(\theta)))]\frac{dz  dw}{z^{k'+1/2} w^{-\ell'+1/2} (z-w)}.
\end{equation}
 Since we have
 \begin{equation}
 \mathsf{S}(z;\mathsf{b}(\theta)) = \sum_{i=1}^m f_i(\theta) S^{(i)}(z;b_i),
 \end{equation}
near the order $2i$ saddle point for each $S^{(i)}$, we let $\eps$ be a small positive number that tends to zero as $\theta$ tends to infinity and consider a change of variables
\begin{equation}
  z = 1 + \zeta \eps, \qquad w = 1 + \omega \eps, \qquad k'=\frac{x}{\eps},
  \qquad \ell'=\frac{y}{\eps}
\end{equation}
(this simple setup is sufficient for our purposes; we will parametrise the contours explicitly once we have suitable $\eps$ and $f_i(\theta)$). Expanding in small $\eps$ and using~\eqref{eq:actionsisaddle}, the leading order approximation of the integrand is
\begin{equation}
  \label{eq:limint}
  \frac{1}{\eps(\zeta-\omega)} \exp\left[\sum_{i=1}^m \theta f_i(\theta) \frac{ (-1)^{i+1} d_i}{2i+1}  \eps^{2i+1}(\zeta^{2i+1} - \omega^{2i+1}) - x \zeta + y \omega + O(\theta \eps^{2m+3})\right] .
\end{equation}

It now becomes clear that in the generalised multicritical action, each $f_i(\theta)$ should scale as $\eps^{-2i-1}/\theta$. More precisely, to use our convention that $f_m(\theta)= 1$, we identify  $\eps = (d_m\theta)^{-1/(2m+1)}$ (which indeed tends to $0$) to be the appropriate scale; taking an action with 
\begin{equation}
f_i(\theta) := (-1)^{m-i} \frac{\tau_i}{d_i} \theta^{\frac{2i-2m}{2m+1}}, \quad i=1,\dots,m-1,
\end{equation}
the leading order term coincides precisely with the integrand of $\mathcal{A}_{\tau,2m+1}$.
At the level of the parametrised specialisations for the corresponding Schur measures, this gives corresponds precisely to Miwa times $
\gamma^\tau(\theta)_r $ corresponding the generalised multicritical measure $\P^{\tau;m}_\theta$. The function $\mathsf{b}(\theta)$ determining the edge scaling becomes $B(\theta)$ defined in~\eqref{eq:bigbtheta}.

\paragraph*{The  edge asymptotics} With $f_i(\theta), \eps$ now determined, let us briefly discuss the remaining analysis needed to prove Theorem~\ref{thm:genedgefluctuations}. From noting that the Jacobian for the change of variables from $z,w$ to $\zeta, \omega$ contributes a factor of $\eps^2$, we see that $ (d_m\theta)^{1/(2m+1)}\mathcal{J}_\theta^{\tau;m}$ is the relevant rescaled kernel. 

Comparing to the analysis of Section~\ref{sec:mcedge_fluctuations}, note that the tails bound and the exponential decay apply immediately to this case. The same contours can be reused along with the same dominated convergence arguments, to show firstly the uniform convergence
\begin{align}
  (d_m\theta)^{\frac{1}{2m+1}}\mathcal{J}_\theta^{\tau;m} (&\lfloor B(\theta) + x (d_m\theta)^{\frac{1}{2m+1}}\rfloor -\tfrac12 ,\lfloor B(\theta) + y (d_m\theta)^{\frac{1}{2m+1}}\rfloor -\tfrac12) \notag \\
  & \to  \mathcal{A}_{\tau;2m+1}(x,y)
\end{align}
as $\theta \to \infty$, and in turn the convergence of traces and finally of Fredholm determinants uniformly on sets bounded below, with 
\begin{align}
\lim_{\theta \to \infty}\P \left[\frac{\lh_1 - B(\theta)}{(d_m\theta)^{\frac{1}{2m+1}}} < s\right] &= \lim_{\theta \to \infty}\det(1-\mathcal{J}_\theta^{\tau;m})_{l^2(\Z_{\geq 0}+ \lfloor B(\theta) + s(d_m\theta)^{\frac{1}{2m+1}}\rfloor -\frac12 )}\notag\\
&= \det(1- \mathcal{A}_{\tau;2m+1})_{L^2([s,\infty))}
\end{align}
as required. 

Finally, let us note that the extensions presented in this appendix and in Appendix~\ref{sec:mc_cylindric} are completely compatible; we can directly construct analogous ``generalised cylindric multicritical measures'' using the Miwa time specialisations of Definition~\ref{def:genmmeasdef}. The distributions $F^\alpha_{2m+1}$ then generalise to Fredholm determinants of positive temperature kernels composed of the functions $\Ai_{\tau;2m+1}$. 

\end{appendix}

\bibliography{bib_multicritical.bib}

\end{document}

%% file: multicriticalsimm2.pdf_tex
\begingroup%
  \makeatletter%
  \providecommand\color[2][]{%
    \errmessage{(Inkscape) Color is used for the text in Inkscape, but the package 'color.sty' is not loaded}%
    \renewcommand\color[2][]{}%
  }%
  \providecommand\transparent[1]{%
    \errmessage{(Inkscape) Transparency is used (non-zero) for the text in Inkscape, but the package 'transparent.sty' is not loaded}%
    \renewcommand\transparent[1]{}%
  }%
  \providecommand\rotatebox[2]{#2}%
  \newcommand*\fsize{\dimexpr\f@size pt\relax}%
  \newcommand*\lineheight[1]{\fontsize{\fsize}{#1\fsize}\selectfont}%
  \ifx\svgwidth\undefined%
    \setlength{\unitlength}{271.35641056bp}%
    \ifx\svgscale\undefined%
      \relax%
    \else%
      \setlength{\unitlength}{\unitlength * \real{\svgscale}}%
    \fi%
  \else%
    \setlength{\unitlength}{\svgwidth}%
  \fi%
  \global\let\svgwidth\undefined%
  \global\let\svgscale\undefined%
  \makeatother%
  \begin{picture}(1,0.72727223)%
    \lineheight{1}%
    \setlength\tabcolsep{0pt}%
    \put(0,0){\includegraphics[width=\unitlength,page=1]{multicriticalsimm2.pdf}}%
    \put(0.09330682,0.06139015){\makebox(0,0)[lt]{\lineheight{1.25}\smash{\begin{tabular}[t]{l}$m=2$\end{tabular}}}}%
    \put(0,0){\includegraphics[width=\unitlength,page=2]{multicriticalsimm2.pdf}}%
  \end{picture}%
\endgroup%

%% file: multicriticalsimm3.pdf_tex
\begingroup%
  \makeatletter%
  \providecommand\color[2][]{%
    \errmessage{(Inkscape) Color is used for the text in Inkscape, but the package 'color.sty' is not loaded}%
    \renewcommand\color[2][]{}%
  }%
  \providecommand\transparent[1]{%
    \errmessage{(Inkscape) Transparency is used (non-zero) for the text in Inkscape, but the package 'transparent.sty' is not loaded}%
    \renewcommand\transparent[1]{}%
  }%
  \providecommand\rotatebox[2]{#2}%
  \newcommand*\fsize{\dimexpr\f@size pt\relax}%
  \newcommand*\lineheight[1]{\fontsize{\fsize}{#1\fsize}\selectfont}%
  \ifx\svgwidth\undefined%
    \setlength{\unitlength}{223.78894282bp}%
    \ifx\svgscale\undefined%
      \relax%
    \else%
      \setlength{\unitlength}{\unitlength * \real{\svgscale}}%
    \fi%
  \else%
    \setlength{\unitlength}{\svgwidth}%
  \fi%
  \global\let\svgwidth\undefined%
  \global\let\svgscale\undefined%
  \makeatother%
  \begin{picture}(1,0.81310775)%
    \lineheight{1}%
    \setlength\tabcolsep{0pt}%
    \put(0,0){\includegraphics[width=\unitlength,page=1]{multicriticalsimm3.pdf}}%
    \put(0.10366559,0.07598767){\makebox(0,0)[lt]{\lineheight{1.25}\smash{\begin{tabular}[t]{l}$m=3$\end{tabular}}}}%
    \put(0,0){\includegraphics[width=\unitlength,page=2]{multicriticalsimm3.pdf}}%
  \end{picture}%
\endgroup%

%% file: multicriticalsimm4.pdf_tex
\begingroup%
  \makeatletter%
  \providecommand\color[2][]{%
    \errmessage{(Inkscape) Color is used for the text in Inkscape, but the package 'color.sty' is not loaded}%
    \renewcommand\color[2][]{}%
  }%
  \providecommand\transparent[1]{%
    \errmessage{(Inkscape) Transparency is used (non-zero) for the text in Inkscape, but the package 'transparent.sty' is not loaded}%
    \renewcommand\transparent[1]{}%
  }%
  \providecommand\rotatebox[2]{#2}%
  \newcommand*\fsize{\dimexpr\f@size pt\relax}%
  \newcommand*\lineheight[1]{\fontsize{\fsize}{#1\fsize}\selectfont}%
  \ifx\svgwidth\undefined%
    \setlength{\unitlength}{228.27724123bp}%
    \ifx\svgscale\undefined%
      \relax%
    \else%
      \setlength{\unitlength}{\unitlength * \real{\svgscale}}%
    \fi%
  \else%
    \setlength{\unitlength}{\svgwidth}%
  \fi%
  \global\let\svgwidth\undefined%
  \global\let\svgscale\undefined%
  \makeatother%
  \begin{picture}(1,0.75176148)%
    \lineheight{1}%
    \setlength\tabcolsep{0pt}%
    \put(0,0){\includegraphics[width=\unitlength,page=1]{multicriticalsimm4.pdf}}%
    \put(0.10274363,0.07472627){\makebox(0,0)[lt]{\lineheight{1.25}\smash{\begin{tabular}[t]{l}$m=4$\end{tabular}}}}%
    \put(0,0){\includegraphics[width=\unitlength,page=2]{multicriticalsimm4.pdf}}%
  \end{picture}%
\endgroup%

%% file: asymetric.pdf_tex
\begingroup%
  \makeatletter%
  \providecommand\color[2][]{%
    \errmessage{(Inkscape) Color is used for the text in Inkscape, but the package 'color.sty' is not loaded}%
    \renewcommand\color[2][]{}%
  }%
  \providecommand\transparent[1]{%
    \errmessage{(Inkscape) Transparency is used (non-zero) for the text in Inkscape, but the package 'transparent.sty' is not loaded}%
    \renewcommand\transparent[1]{}%
  }%
  \providecommand\rotatebox[2]{#2}%
  \newcommand*\fsize{\dimexpr\f@size pt\relax}%
  \newcommand*\lineheight[1]{\fontsize{\fsize}{#1\fsize}\selectfont}%
  \ifx\svgwidth\undefined%
    \setlength{\unitlength}{361.09343289bp}%
    \ifx\svgscale\undefined%
      \relax%
    \else%
      \setlength{\unitlength}{\unitlength * \real{\svgscale}}%
    \fi%
  \else%
    \setlength{\unitlength}{\svgwidth}%
  \fi%
  \global\let\svgwidth\undefined%
  \global\let\svgscale\undefined%
  \makeatother%
  \begin{picture}(1,0.59343514)%
    \lineheight{1}%
    \setlength\tabcolsep{0pt}%
    \put(0,0){\includegraphics[width=\unitlength]{asymetric.pdf}}%
    \put(0.09493879,0.13796072){\color[rgb]{0,0,0}\makebox(0,0)[lt]{\lineheight{1.25}\smash{\begin{tabular}[t]{l}-\end{tabular}}}}%
    \put(0.10949419,0.13796072){\color[rgb]{0,0,0}\makebox(0,0)[lt]{\lineheight{1.25}\smash{\begin{tabular}[t]{l}3\end{tabular}}}}%
    \put(0.22493256,0.13796072){\color[rgb]{0,0,0}\makebox(0,0)[lt]{\lineheight{1.25}\smash{\begin{tabular}[t]{l}-\end{tabular}}}}%
    \put(0.23948796,0.13796072){\color[rgb]{0,0,0}\makebox(0,0)[lt]{\lineheight{1.25}\smash{\begin{tabular}[t]{l}2\end{tabular}}}}%
    \put(0.35492523,0.13796072){\color[rgb]{0,0,0}\makebox(0,0)[lt]{\lineheight{1.25}\smash{\begin{tabular}[t]{l}-\end{tabular}}}}%
    \put(0.36948063,0.13796072){\color[rgb]{0,0,0}\makebox(0,0)[lt]{\lineheight{1.25}\smash{\begin{tabular}[t]{l}1\end{tabular}}}}%
    \put(0.62219123,0.13796072){\color[rgb]{0,0,0}\makebox(0,0)[lt]{\lineheight{1.25}\smash{\begin{tabular}[t]{l}1\end{tabular}}}}%
    \put(0.75218528,0.13796072){\color[rgb]{0,0,0}\makebox(0,0)[lt]{\lineheight{1.25}\smash{\begin{tabular}[t]{l}2\end{tabular}}}}%
    \put(0.88217934,0.13796072){\color[rgb]{0,0,0}\makebox(0,0)[lt]{\lineheight{1.25}\smash{\begin{tabular}[t]{l}3\end{tabular}}}}%
    \put(0.462,0.26652301){\color[rgb]{0,0,0}\makebox(0,0)[lt]{\lineheight{1.25}\smash{\begin{tabular}[t]{l}1\end{tabular}}}}%
    \put(0.462,0.37285836){\color[rgb]{0,0,0}\makebox(0,0)[lt]{\lineheight{1.25}\smash{\begin{tabular}[t]{l}2\end{tabular}}}}%
    \put(0.462,0.47919095){\color[rgb]{0,0,0}\makebox(0,0)[lt]{\lineheight{1.25}\smash{\begin{tabular}[t]{l}3\end{tabular}}}}%
    \put(0.92039742,0.40223609){\color[rgb]{0,0,0}\makebox(0,0)[lt]{\lineheight{1.25}\smash{\begin{tabular}[t]{l}$m=1$\end{tabular}}}}%
    \put(0.92039742,0.35388377){\color[rgb]{0,0,0}\makebox(0,0)[lt]{\lineheight{1.25}\smash{\begin{tabular}[t]{l}$m=2$\end{tabular}}}}%
    \put(0.92039742,0.30553145){\color[rgb]{0,0,0}\makebox(0,0)[lt]{\lineheight{1.25}\smash{\begin{tabular}[t]{l}$m=3$\end{tabular}}}}%
    \put(0.92039742,0.25717913){\color[rgb]{0,0,0}\makebox(0,0)[lt]{\lineheight{1.25}\smash{\begin{tabular}[t]{l}$m=4$\end{tabular}}}}%
    \put(0.92039742,0.20882681){\color[rgb]{0,0,0}\makebox(0,0)[lt]{\lineheight{1.25}\smash{\begin{tabular}[t]{l}$m=5$\end{tabular}}}}%
    \put(0.26212319,0.04351258){\color[rgb]{0,0,0}\makebox(0,0)[lt]{\lineheight{1.25}\smash{\begin{tabular}[t]{l}$\varrho^{a,m}(x)$\end{tabular}}}}%
    \put(0.535,0.53539088){\makebox(0,0)[lt]{\lineheight{1.25}\smash{\begin{tabular}[t]{l}$\Omega^{\text{a},m}(x)$\end{tabular}}}}%
    \put(0.93831483,0.14376376){\makebox(0,0)[lt]{\lineheight{1.25}\smash{\begin{tabular}[t]{l}$x$\end{tabular}}}}%
  \end{picture}%
\endgroup%

%% file: symetric.pdf_tex
\begingroup%
  \makeatletter%
  \providecommand\color[2][]{%
    \errmessage{(Inkscape) Color is used for the text in Inkscape, but the package 'color.sty' is not loaded}%
    \renewcommand\color[2][]{}%
  }%
  \providecommand\transparent[1]{%
    \errmessage{(Inkscape) Transparency is used (non-zero) for the text in Inkscape, but the package 'transparent.sty' is not loaded}%
    \renewcommand\transparent[1]{}%
  }%
  \providecommand\rotatebox[2]{#2}%
  \newcommand*\fsize{\dimexpr\f@size pt\relax}%
  \newcommand*\lineheight[1]{\fontsize{\fsize}{#1\fsize}\selectfont}%
  \ifx\svgwidth\undefined%
    \setlength{\unitlength}{361.7276097bp}%
    \ifx\svgscale\undefined%
      \relax%
    \else%
      \setlength{\unitlength}{\unitlength * \real{\svgscale}}%
    \fi%
  \else%
    \setlength{\unitlength}{\svgwidth}%
  \fi%
  \global\let\svgwidth\undefined%
  \global\let\svgscale\undefined%
  \makeatother%
  \begin{picture}(1,0.59260459)%
    \lineheight{1}%
    \setlength\tabcolsep{0pt}%
    \put(0,0){\includegraphics[width=\unitlength]{symetric.pdf}}%
    \put(0.465,0.32229769){\color[rgb]{0,0,0}\makebox(0,0)[lt]{\lineheight{1.25}\smash{\begin{tabular}[t]{l}   1\end{tabular}}}}%
    \put(0.465,0.48488688){\color[rgb]{0,0,0}\makebox(0,0)[lt]{\lineheight{1.25}\smash{\begin{tabular}[t]{l}   2\end{tabular}}}}%
    \put(0.91497478,0.39606249){\color[rgb]{0,0,0}\makebox(0,0)[lt]{\lineheight{1.25}\smash{\begin{tabular}[t]{l}$m=1$\end{tabular}}}}%
    \put(0.91497478,0.34779494){\color[rgb]{0,0,0}\makebox(0,0)[lt]{\lineheight{1.25}\smash{\begin{tabular}[t]{l}$m=2$\end{tabular}}}}%
    \put(0.91497478,0.29952739){\color[rgb]{0,0,0}\makebox(0,0)[lt]{\lineheight{1.25}\smash{\begin{tabular}[t]{l}$m=3$\end{tabular}}}}%
    \put(0.91497478,0.25125984){\color[rgb]{0,0,0}\makebox(0,0)[lt]{\lineheight{1.25}\smash{\begin{tabular}[t]{l}$m=4$\end{tabular}}}}%
    \put(0.91497478,0.20299229){\color[rgb]{0,0,0}\makebox(0,0)[lt]{\lineheight{1.25}\smash{\begin{tabular}[t]{l}$m=5$\end{tabular}}}}%
    \put(0.53697453,0.53104984){\makebox(0,0)[lt]{\lineheight{1.25}\smash{\begin{tabular}[t]{l}$\Omega^{\text{s},m}(x)$\end{tabular}}}}%
    \put(0.93451399,0.14212325){\makebox(0,0)[lt]{\lineheight{1.25}\smash{\begin{tabular}[t]{l}$x$\end{tabular}}}}%
    \put(0.21894738,0.04312854){\color[rgb]{0,0,0}\makebox(0,0)[lt]{\lineheight{1.25}\smash{\begin{tabular}[t]{l}$\varrho^{\text{s},m}(x)$\end{tabular}}}}%
    \put(0.08828591,0.13751501){\color[rgb]{0,0,0}\makebox(0,0)[lt]{\lineheight{1.25}\smash{\begin{tabular}[t]{l}-\end{tabular}}}}%
    \put(0.10281579,0.13751501){\color[rgb]{0,0,0}\makebox(0,0)[lt]{\lineheight{1.25}\smash{\begin{tabular}[t]{l}2\end{tabular}}}}%
    \put(0.28705397,0.13751501){\color[rgb]{0,0,0}\makebox(0,0)[lt]{\lineheight{1.25}\smash{\begin{tabular}[t]{l}-\end{tabular}}}}%
    \put(0.30158386,0.13751501){\color[rgb]{0,0,0}\makebox(0,0)[lt]{\lineheight{1.25}\smash{\begin{tabular}[t]{l}1\end{tabular}}}}%
    \put(0.6918558,0.13751501){\color[rgb]{0,0,0}\makebox(0,0)[lt]{\lineheight{1.25}\smash{\begin{tabular}[t]{l}1\end{tabular}}}}%
    \put(0.89062414,0.13751501){\color[rgb]{0,0,0}\makebox(0,0)[lt]{\lineheight{1.25}\smash{\begin{tabular}[t]{l}2\end{tabular}}}}%
  \end{picture}%
\endgroup%

%% file: asymm2contours.pdf_tex
\begingroup%
  \makeatletter%
  \providecommand\color[2][]{%
    \errmessage{(Inkscape) Color is used for the text in Inkscape, but the package 'color.sty' is not loaded}%
    \renewcommand\color[2][]{}%
  }%
  \providecommand\transparent[1]{%
    \errmessage{(Inkscape) Transparency is used (non-zero) for the text in Inkscape, but the package 'transparent.sty' is not loaded}%
    \renewcommand\transparent[1]{}%
  }%
  \providecommand\rotatebox[2]{#2}%
  \newcommand*\fsize{\dimexpr\f@size pt\relax}%
  \newcommand*\lineheight[1]{\fontsize{\fsize}{#1\fsize}\selectfont}%
  \ifx\svgwidth\undefined%
    \setlength{\unitlength}{737.68184291bp}%
    \ifx\svgscale\undefined%
      \relax%
    \else%
      \setlength{\unitlength}{\unitlength * \real{\svgscale}}%
    \fi%
  \else%
    \setlength{\unitlength}{\svgwidth}%
  \fi%
  \global\let\svgwidth\undefined%
  \global\let\svgscale\undefined%
  \makeatother%
  \begin{picture}(1,0.66333027)%
    \lineheight{1}%
    \setlength\tabcolsep{0pt}%
    \put(0,0){\includegraphics[width=\unitlength,page=1]{asymm2contours.pdf}}%
    \put(0.44261314,0.53603295){\makebox(0,0)[lt]{\lineheight{1.25}\smash{\begin{tabular}[t]{l}$c_+'$\end{tabular}}}}%
    \put(0.75028629,0.47508337){\makebox(0,0)[lt]{\lineheight{1.25}\smash{\begin{tabular}[t]{l}$c_-'$\end{tabular}}}}%
    \put(0.62566272,0.53394309){\makebox(0,0)[lt]{\lineheight{1.25}\smash{\begin{tabular}[t]{l}$\chi$\end{tabular}}}}%
    \put(0.70617067,0.40524137){\makebox(0,0)[lt]{\lineheight{1.25}\smash{\begin{tabular}[t]{l}$c_{2\chi}$\end{tabular}}}}%
    \put(0,0){\includegraphics[width=\unitlength,page=2]{asymm2contours.pdf}}%
    \put(0.44261314,0.64735727){\makebox(0,0)[lt]{\lineheight{1.25}\smash{\begin{tabular}[t]{l}$-\tilde{b}<x<b$\end{tabular}}}}%
    \put(0,0){\includegraphics[width=\unitlength,page=3]{asymm2contours.pdf}}%
    \put(0.26268155,0.48829367){\makebox(0,0)[lt]{\lineheight{1.25}\smash{\begin{tabular}[t]{l}$c_-$\end{tabular}}}}%
    \put(0.02882787,0.51075667){\makebox(0,0)[lt]{\lineheight{1.25}\smash{\begin{tabular}[t]{l}$c_+$\end{tabular}}}}%
    \put(0.39981064,0.48356847){\makebox(0,0)[lt]{\lineheight{1.25}\smash{\begin{tabular}[t]{l}$\Re$\end{tabular}}}}%
    \put(0.2134896,0.6462176){\makebox(0,0)[lt]{\lineheight{1.25}\smash{\begin{tabular}[t]{l}$\Im$\end{tabular}}}}%
    \put(0.02882787,0.6462176){\makebox(0,0)[lt]{\lineheight{1.25}\smash{\begin{tabular}[t]{l}$x>b$\end{tabular}}}}%
    \put(0,0){\includegraphics[width=\unitlength,page=4]{asymm2contours.pdf}}%
    \put(0.3469058,0.13366656){\makebox(0,0)[lt]{\lineheight{1.25}\smash{\begin{tabular}[t]{l}$c_-'$\end{tabular}}}}%
    \put(0.13122662,0.18149258){\makebox(0,0)[lt]{\lineheight{1.25}\smash{\begin{tabular}[t]{l}$c_+'$\end{tabular}}}}%
    \put(0.02882607,0.30148815){\makebox(0,0)[lt]{\lineheight{1.25}\smash{\begin{tabular}[t]{l}$x<-\tilde{b}$\end{tabular}}}}%
    \put(0,0){\includegraphics[width=\unitlength,page=5]{asymm2contours.pdf}}%
    \put(0.27217996,0.02912055){\makebox(0,0)[lt]{\lineheight{1.25}\smash{\begin{tabular}[t]{l}$c_1$\end{tabular}}}}%
    \put(0,0){\includegraphics[width=\unitlength,page=6]{asymm2contours.pdf}}%
  \end{picture}%
\endgroup%

%% file: symm3empty.pdf_tex
\begingroup%
  \makeatletter%
  \providecommand\color[2][]{%
    \errmessage{(Inkscape) Color is used for the text in Inkscape, but the package 'color.sty' is not loaded}%
    \renewcommand\color[2][]{}%
  }%
  \providecommand\transparent[1]{%
    \errmessage{(Inkscape) Transparency is used (non-zero) for the text in Inkscape, but the package 'transparent.sty' is not loaded}%
    \renewcommand\transparent[1]{}%
  }%
  \providecommand\rotatebox[2]{#2}%
  \newcommand*\fsize{\dimexpr\f@size pt\relax}%
  \newcommand*\lineheight[1]{\fontsize{\fsize}{#1\fsize}\selectfont}%
  \ifx\svgwidth\undefined%
    \setlength{\unitlength}{1169.79845211bp}%
    \ifx\svgscale\undefined%
      \relax%
    \else%
      \setlength{\unitlength}{\unitlength * \real{\svgscale}}%
    \fi%
  \else%
    \setlength{\unitlength}{\svgwidth}%
  \fi%
  \global\let\svgwidth\undefined%
  \global\let\svgscale\undefined%
  \makeatother%
  \begin{picture}(1,0.20977794)%
    \lineheight{1}%
    \setlength\tabcolsep{0pt}%
    \put(0,0){\includegraphics[width=\unitlength,page=1]{symm3empty.pdf}}%
    \put(0.19574824,0.08206382){\makebox(0,0)[lt]{\lineheight{1.25}\smash{\begin{tabular}[t]{l}$c_-$\end{tabular}}}}%
    \put(0.07629113,0.10862643){\makebox(0,0)[lt]{\lineheight{1.25}\smash{\begin{tabular}[t]{l}$c_+$\end{tabular}}}}%
    \put(0,0){\includegraphics[width=\unitlength,page=2]{symm3empty.pdf}}%
    \put(0.593156,0.08005529){\makebox(0,0)[lt]{\lineheight{1.25}\smash{\begin{tabular}[t]{l}$c_-'$\end{tabular}}}}%
    \put(0.37876133,0.11006031){\makebox(0,0)[lt]{\lineheight{1.25}\smash{\begin{tabular}[t]{l}$c_+'$\end{tabular}}}}%
    \put(0.55452865,0.05443775){\makebox(0,0)[lt]{\lineheight{1.25}\smash{\begin{tabular}[t]{l}$c_{2\chi}$\end{tabular}}}}%
    \put(0,0){\includegraphics[width=\unitlength,page=3]{symm3empty.pdf}}%
    \put(0.50357972,0.11143282){\makebox(0,0)[lt]{\lineheight{1.25}\smash{\begin{tabular}[t]{l}$\chi$\end{tabular}}}}%
    \put(0,0){\includegraphics[width=\unitlength,page=4]{symm3empty.pdf}}%
    \put(0.90232075,0.08034585){\makebox(0,0)[lt]{\lineheight{1.25}\smash{\begin{tabular}[t]{l}$c_-'$\end{tabular}}}}%
    \put(0.87120818,0.03323068){\makebox(0,0)[lt]{\lineheight{1.25}\smash{\begin{tabular}[t]{l}$c_1$\end{tabular}}}}%
    \put(0.78492476,0.10470111){\makebox(0,0)[lt]{\lineheight{1.25}\smash{\begin{tabular}[t]{l}$c_+'$\end{tabular}}}}%
    \put(0.00409419,0.20198413){\makebox(0,0)[lt]{\lineheight{1.25}\smash{\begin{tabular}[t]{l}$x>b$\end{tabular}}}}%
    \put(0.34601458,0.20198413){\makebox(0,0)[lt]{\lineheight{1.25}\smash{\begin{tabular}[t]{l}$x=0$\end{tabular}}}}%
    \put(0.68731195,0.20198413){\makebox(0,0)[lt]{\lineheight{1.25}\smash{\begin{tabular}[t]{l}$x<-b$\end{tabular}}}}%
    \put(0,0){\includegraphics[width=\unitlength,page=5]{symm3empty.pdf}}%
  \end{picture}%
\endgroup%

%% file: m1contoursat1.pdf_tex
\begingroup%
  \makeatletter%
  \providecommand\color[2][]{%
    \errmessage{(Inkscape) Color is used for the text in Inkscape, but the package 'color.sty' is not loaded}%
    \renewcommand\color[2][]{}%
  }%
  \providecommand\transparent[1]{%
    \errmessage{(Inkscape) Transparency is used (non-zero) for the text in Inkscape, but the package 'transparent.sty' is not loaded}%
    \renewcommand\transparent[1]{}%
  }%
  \providecommand\rotatebox[2]{#2}%
  \newcommand*\fsize{\dimexpr\f@size pt\relax}%
  \newcommand*\lineheight[1]{\fontsize{\fsize}{#1\fsize}\selectfont}%
  \ifx\svgwidth\undefined%
    \setlength{\unitlength}{327.10323265bp}%
    \ifx\svgscale\undefined%
      \relax%
    \else%
      \setlength{\unitlength}{\unitlength * \real{\svgscale}}%
    \fi%
  \else%
    \setlength{\unitlength}{\svgwidth}%
  \fi%
  \global\let\svgwidth\undefined%
  \global\let\svgscale\undefined%
  \makeatother%
  \begin{picture}(1,1.08678633)%
    \lineheight{1}%
    \setlength\tabcolsep{0pt}%
    \put(0,0){\includegraphics[width=\unitlength,page=1]{m1contoursat1.pdf}}%
    \put(0.5310739,0.03893603){\makebox(0,0)[lt]{\lineheight{1.25}\smash{\begin{tabular}[t]{l}$c_+$\end{tabular}}}}%
    \put(0.27265383,0.03924782){\makebox(0,0)[lt]{\lineheight{1.25}\smash{\begin{tabular}[t]{l}$c_-$\end{tabular}}}}%
    \put(0.38775051,0.0293257){\makebox(0,0)[lt]{\lineheight{1.25}\smash{\begin{tabular}[t]{l}$c_1$\end{tabular}}}}%
    \put(0,0.0293257){\makebox(0,0)[lt]{\lineheight{1.25}\smash{\begin{tabular}[t]{l}$m=1$\end{tabular}}}}%
    \put(0.48498741,0.52543214){\makebox(0,0)[lt]{\lineheight{1.25}\smash{\begin{tabular}[t]{l}$1$\end{tabular}}}}%
    \put(0.89629959,0.51363107){\makebox(0,0)[lt]{\lineheight{1.25}\smash{\begin{tabular}[t]{l}$\Re$\end{tabular}}}}%
    \put(0,0){\includegraphics[width=\unitlength,page=2]{m1contoursat1.pdf}}%
  \end{picture}%
\endgroup%

%% file: m2contoursat1.pdf_tex
\begingroup%
  \makeatletter%
  \providecommand\color[2][]{%
    \errmessage{(Inkscape) Color is used for the text in Inkscape, but the package 'color.sty' is not loaded}%
    \renewcommand\color[2][]{}%
  }%
  \providecommand\transparent[1]{%
    \errmessage{(Inkscape) Transparency is used (non-zero) for the text in Inkscape, but the package 'transparent.sty' is not loaded}%
    \renewcommand\transparent[1]{}%
  }%
  \providecommand\rotatebox[2]{#2}%
  \newcommand*\fsize{\dimexpr\f@size pt\relax}%
  \newcommand*\lineheight[1]{\fontsize{\fsize}{#1\fsize}\selectfont}%
  \ifx\svgwidth\undefined%
    \setlength{\unitlength}{327.32434948bp}%
    \ifx\svgscale\undefined%
      \relax%
    \else%
      \setlength{\unitlength}{\unitlength * \real{\svgscale}}%
    \fi%
  \else%
    \setlength{\unitlength}{\svgwidth}%
  \fi%
  \global\let\svgwidth\undefined%
  \global\let\svgscale\undefined%
  \makeatother%
  \begin{picture}(1,1.08611832)%
    \lineheight{1}%
    \setlength\tabcolsep{0pt}%
    \put(0,0){\includegraphics[width=\unitlength,page=1]{m2contoursat1.pdf}}%
    \put(0.52900467,0.03897588){\makebox(0,0)[lt]{\lineheight{1.25}\smash{\begin{tabular}[t]{l}$c_+$\end{tabular}}}}%
    \put(0.27075916,0.03928745){\makebox(0,0)[lt]{\lineheight{1.25}\smash{\begin{tabular}[t]{l}$c_-$\end{tabular}}}}%
    \put(0.38577809,0.02937203){\makebox(0,0)[lt]{\lineheight{1.25}\smash{\begin{tabular}[t]{l}$c_1$\end{tabular}}}}%
    \put(0,0.02937203){\makebox(0,0)[lt]{\lineheight{1.25}\smash{\begin{tabular}[t]{l}$m=2$\end{tabular}}}}%
    \put(0.4829493,0.52514334){\makebox(0,0)[lt]{\lineheight{1.25}\smash{\begin{tabular}[t]{l}$1$\end{tabular}}}}%
  \end{picture}%
\endgroup%

%% file: m3contoursat1.pdf_tex
\begingroup%
  \makeatletter%
  \providecommand\color[2][]{%
    \errmessage{(Inkscape) Color is used for the text in Inkscape, but the package 'color.sty' is not loaded}%
    \renewcommand\color[2][]{}%
  }%
  \providecommand\transparent[1]{%
    \errmessage{(Inkscape) Transparency is used (non-zero) for the text in Inkscape, but the package 'transparent.sty' is not loaded}%
    \renewcommand\transparent[1]{}%
  }%
  \providecommand\rotatebox[2]{#2}%
  \newcommand*\fsize{\dimexpr\f@size pt\relax}%
  \newcommand*\lineheight[1]{\fontsize{\fsize}{#1\fsize}\selectfont}%
  \ifx\svgwidth\undefined%
    \setlength{\unitlength}{334.05848441bp}%
    \ifx\svgscale\undefined%
      \relax%
    \else%
      \setlength{\unitlength}{\unitlength * \real{\svgscale}}%
    \fi%
  \else%
    \setlength{\unitlength}{\svgwidth}%
  \fi%
  \global\let\svgwidth\undefined%
  \global\let\svgscale\undefined%
  \makeatother%
  \begin{picture}(1,1.06422375)%
    \lineheight{1}%
    \setlength\tabcolsep{0pt}%
    \put(0,0){\includegraphics[width=\unitlength,page=1]{m3contoursat1.pdf}}%
    \put(0.53849928,0.03819018){\makebox(0,0)[lt]{\lineheight{1.25}\smash{\begin{tabular}[t]{l}$c_+$\end{tabular}}}}%
    \put(0.28545958,0.03849548){\makebox(0,0)[lt]{\lineheight{1.25}\smash{\begin{tabular}[t]{l}$c_-$\end{tabular}}}}%
    \put(0.39815994,0.02877993){\makebox(0,0)[lt]{\lineheight{1.25}\smash{\begin{tabular}[t]{l}$c_1$\end{tabular}}}}%
    \put(0,0.02877993){\makebox(0,0)[lt]{\lineheight{1.25}\smash{\begin{tabular}[t]{l}$m=3$\end{tabular}}}}%
    \put(0.49337232,0.51455723){\makebox(0,0)[lt]{\lineheight{1.25}\smash{\begin{tabular}[t]{l}$1$\end{tabular}}}}%
    \put(0,0){\includegraphics[width=\unitlength,page=2]{m3contoursat1.pdf}}%
  \end{picture}%
\endgroup%

%% file: evdensityplots.pdf_tex
\begingroup%
  \makeatletter%
  \providecommand\color[2][]{%
    \errmessage{(Inkscape) Color is used for the text in Inkscape, but the package 'color.sty' is not loaded}%
    \renewcommand\color[2][]{}%
  }%
  \providecommand\transparent[1]{%
    \errmessage{(Inkscape) Transparency is used (non-zero) for the text in Inkscape, but the package 'transparent.sty' is not loaded}%
    \renewcommand\transparent[1]{}%
  }%
  \providecommand\rotatebox[2]{#2}%
  \newcommand*\fsize{\dimexpr\f@size pt\relax}%
  \newcommand*\lineheight[1]{\fontsize{\fsize}{#1\fsize}\selectfont}%
  \ifx\svgwidth\undefined%
    \setlength{\unitlength}{835.60765458bp}%
    \ifx\svgscale\undefined%
      \relax%
    \else%
      \setlength{\unitlength}{\unitlength * \real{\svgscale}}%
    \fi%
  \else%
    \setlength{\unitlength}{\svgwidth}%
  \fi%
  \global\let\svgwidth\undefined%
  \global\let\svgscale\undefined%
  \makeatother%
  \begin{picture}(1,0.28276882)%
    \lineheight{1}%
    \setlength\tabcolsep{0pt}%
    \put(0,0){\includegraphics[width=\unitlength,page=1]{evdensityplots.pdf}}%
    \put(0.6641428,0.09972832){\color[rgb]{0,0,0}\makebox(0,0)[lt]{\lineheight{1.25}\smash{\begin{tabular}[t]{l}{\tiny 0.1}\end{tabular}}}}%
    \put(0,0){\includegraphics[width=\unitlength,page=2]{evdensityplots.pdf}}%
    \put(0.6641428,0.17490774){\color[rgb]{0,0,0}\makebox(0,0)[lt]{\lineheight{1.25}\smash{\begin{tabular}[t]{l}{\tiny 0.2}\end{tabular}}}}%
    \put(0,0){\includegraphics[width=\unitlength,page=3]{evdensityplots.pdf}}%
    \put(0.66414279,0.25008776){\color[rgb]{0,0,0}\makebox(0,0)[lt]{\lineheight{1.25}\smash{\begin{tabular}[t]{l}{\tiny 0.3}\end{tabular}}}}%
    \put(0,0){\includegraphics[width=\unitlength,page=4]{evdensityplots.pdf}}%
    \put(0.4727902,0.01494203){\color[rgb]{0,0,0}\makebox(0,0)[lt]{\lineheight{1.25}\smash{\begin{tabular}[t]{l}$-\pi$\end{tabular}}}}%
    \put(0.89008376,0.01494203){\color[rgb]{0,0,0}\makebox(0,0)[lt]{\lineheight{1.25}\smash{\begin{tabular}[t]{l}$\pi$\end{tabular}}}}%
    \put(0,0){\includegraphics[width=\unitlength,page=5]{evdensityplots.pdf}}%
    \put(0.95269687,0.1945629){\color[rgb]{0,0,0}\makebox(0,0)[lt]{\lineheight{1.25}\smash{\begin{tabular}[t]{l}$m=1$\end{tabular}}}}%
    \put(0,0){\includegraphics[width=\unitlength,page=6]{evdensityplots.pdf}}%
    \put(0.95269687,0.16895279){\color[rgb]{0,0,0}\makebox(0,0)[lt]{\lineheight{1.25}\smash{\begin{tabular}[t]{l}$m=2$\end{tabular}}}}%
    \put(0,0){\includegraphics[width=\unitlength,page=7]{evdensityplots.pdf}}%
    \put(0.95269687,0.14334269){\color[rgb]{0,0,0}\makebox(0,0)[lt]{\lineheight{1.25}\smash{\begin{tabular}[t]{l}$m=3$\end{tabular}}}}%
    \put(0,0){\includegraphics[width=\unitlength,page=8]{evdensityplots.pdf}}%
    \put(0.95269687,0.11773258){\color[rgb]{0,0,0}\makebox(0,0)[lt]{\lineheight{1.25}\smash{\begin{tabular}[t]{l}$m=4$\end{tabular}}}}%
    \put(0,0){\includegraphics[width=\unitlength,page=9]{evdensityplots.pdf}}%
    \put(0.95269687,0.09212248){\color[rgb]{0,0,0}\makebox(0,0)[lt]{\lineheight{1.25}\smash{\begin{tabular}[t]{l}$m=5$\end{tabular}}}}%
    \put(0.86337809,0.00298716){\makebox(0,0)[lt]{\lineheight{1.25}\smash{\begin{tabular}[t]{l}$\alpha$\end{tabular}}}}%
    \put(0.70538351,0.27143811){\makebox(0,0)[lt]{\lineheight{1.25}\smash{\begin{tabular}[t]{l}$\varrho^{\ss,m}(\alpha)$\end{tabular}}}}%
    \put(0.38979362,0.00298716){\makebox(0,0)[lt]{\lineheight{1.25}\smash{\begin{tabular}[t]{l}$\alpha$\end{tabular}}}}%
    \put(0,0){\includegraphics[width=\unitlength,page=10]{evdensityplots.pdf}}%
    \put(0.19116387,0.09855349){\color[rgb]{0,0,0}\makebox(0,0)[lt]{\lineheight{1.25}\smash{\begin{tabular}[t]{l}{\tiny 0.2}\end{tabular}}}}%
    \put(0,0){\includegraphics[width=\unitlength,page=11]{evdensityplots.pdf}}%
    \put(0.19116387,0.17255855){\color[rgb]{0,0,0}\makebox(0,0)[lt]{\lineheight{1.25}\smash{\begin{tabular}[t]{l}{\tiny 0.4}\end{tabular}}}}%
    \put(0,0){\includegraphics[width=\unitlength,page=12]{evdensityplots.pdf}}%
    \put(0.19116387,0.24656338){\color[rgb]{0,0,0}\makebox(0,0)[lt]{\lineheight{1.25}\smash{\begin{tabular}[t]{l}{\tiny 0.6}\end{tabular}}}}%
    \put(0,0){\includegraphics[width=\unitlength,page=13]{evdensityplots.pdf}}%
    \put(-0.00079427,0.01494203){\color[rgb]{0,0,0}\makebox(0,0)[lt]{\lineheight{1.25}\smash{\begin{tabular}[t]{l}$-\pi$\end{tabular}}}}%
    \put(0.41649928,0.01494203){\color[rgb]{0,0,0}\makebox(0,0)[lt]{\lineheight{1.25}\smash{\begin{tabular}[t]{l}$\pi$\end{tabular}}}}%
    \put(0.23179904,0.27143811){\makebox(0,0)[lt]{\lineheight{1.25}\smash{\begin{tabular}[t]{l}$\varrho^{\aa,m}(\alpha)$\end{tabular}}}}%
  \end{picture}%
\endgroup%

%% file: ydtomaya.pdf_tex
\begingroup%
  \makeatletter%
  \providecommand\color[2][]{%
    \errmessage{(Inkscape) Color is used for the text in Inkscape, but the package 'color.sty' is not loaded}%
    \renewcommand\color[2][]{}%
  }%
  \providecommand\transparent[1]{%
    \errmessage{(Inkscape) Transparency is used (non-zero) for the text in Inkscape, but the package 'transparent.sty' is not loaded}%
    \renewcommand\transparent[1]{}%
  }%
  \providecommand\rotatebox[2]{#2}%
  \newcommand*\fsize{\dimexpr\f@size pt\relax}%
  \newcommand*\lineheight[1]{\fontsize{\fsize}{#1\fsize}\selectfont}%
  \ifx\svgwidth\undefined%
    \setlength{\unitlength}{436.51785278bp}%
    \ifx\svgscale\undefined%
      \relax%
    \else%
      \setlength{\unitlength}{\unitlength * \real{\svgscale}}%
    \fi%
  \else%
    \setlength{\unitlength}{\svgwidth}%
  \fi%
  \global\let\svgwidth\undefined%
  \global\let\svgscale\undefined%
  \makeatother%
  \begin{picture}(1,0.54321648)%
    \lineheight{1}%
    \setlength\tabcolsep{0pt}%
    \put(0,0){\includegraphics[width=\unitlength,page=1]{ydtomaya.pdf}}%
    \put(0.66753818,0.47593065){\makebox(0,0)[lt]{\lineheight{1.25}\smash{\begin{tabular}[t]{l}$\lambda$\end{tabular}}}}%
    \put(0.81603791,0.00404032){\makebox(0,0)[lt]{\lineheight{1.25}\smash{\begin{tabular}[t]{l}$S(\lambda)$\end{tabular}}}}%
    \put(0,0){\includegraphics[width=\unitlength,page=2]{ydtomaya.pdf}}%
  \end{picture}%
\endgroup%